\documentclass{amsart}
\usepackage{geometry}
\usepackage{mathtools, amssymb, amsthm, bm, amsmath}
\usepackage{mathrsfs}
\usepackage{bm}
\usepackage{xcolor}
\usepackage{tikz-cd}
\usepackage{tikz}
\usepackage{thmtools}
\usepackage{thm-restate}
\usepackage[shortlabels]{enumitem}
\usepackage{float}
\usepackage{longtable}
\usepackage{hyperref}
\usepackage{cleveref}
\usepackage{autonum}

\DeclareMathOperator{\N}{\mathbb{N}}
\DeclareMathOperator{\Q}{\mathbb{Q}}

\DeclareMathOperator{\C}{\mathbb{C}}
\DeclareMathOperator{\R}{\mathbb{R}}
\DeclareMathOperator{\iR}{i\mathbb{R}}
\DeclareMathOperator{\Z}{\mathbb{Z}}
\DeclareMathOperator{\I}{\hat{\mathbb{I}}}
\newcommand{\ZZ}[1]{\mathbb{Z}/{#1}\mathbb{Z}}
\DeclareMathOperator{\F}{\mathbb{F}}
\DeclareMathOperator{\bbI}{\mathbb{I}}
\DeclareMathOperator{\bH}{\mathbb{H}}
\DeclareMathOperator{\bA}{\mathbb{A}}

\DeclareMathOperator{\cS}{\mathcal{S}}

\DeclareMathOperator{\E}{\mathcal{E}}

\DeclareMathOperator{\csup}{CSupp}
\DeclareMathOperator{\sA}{\mathscr{A}}

\DeclareMathOperator{\GL}{GL}
\DeclareMathOperator{\Ind}{Ind}
\DeclareMathOperator{\Irr}{Irr}

\DeclareMathOperator{\Res}{Res}
\DeclareMathOperator{\End}{End}
\DeclareMathOperator{\Hom}{Hom}

\DeclareMathOperator{\OO}{\mathcal{O}}

\DeclareMathOperator{\eps}{\varepsilon}

\DeclareMathOperator{\ind}{ind}

\DeclareMathOperator{\Ad}{Ad}
\DeclareMathOperator{\ad}{ad}

\DeclareMathOperator{\Sp}{Sp}

\DeclareMathOperator{\SL}{SL}

\renewcommand{\sl}{\mathfrak{sl}}
\newcommand{\gl}{\mathfrak{gl}}

\renewcommand{\sp}{\mathfrak{sp}}
\DeclareMathOperator{\SO}{SO}

\DeclareMathOperator{\mult}{mult}
\DeclareMathOperator{\sgn}{\rm{sgn}}

\DeclareMathOperator{\img}{im}

\DeclareMathOperator{\jordbp}{\Delta}
\DeclareMathOperator{\Psymp}{\mathcal{P}^{\mathrm{symp}}}
\DeclareMathOperator{\PPsymp}{{\bm{\mathcal{P}^{\mathrm{symp}}}}}

\DeclareMathOperator{\Port}{\mathcal{P}^{\mathrm{ort}}}
\DeclareMathOperator{\PPort}{ \bm{\mathcal{P}^{ \mathrm{ort} } }  }
\DeclareMathOperator{\PPPort}{ \bm{\mathcal{P}}^{\mathrm{ort},2}}

\DeclareMathOperator{\epsmax}{\eps^{\mathrm{max}}}
\DeclareMathOperator{\psimax}{\psi^{\mathrm{max}}}
\DeclareMathOperator{\lambdamax}{\lambda^{\mathrm{max}}}
\DeclareMathOperator{\lambdamin}{\lambda^{\mathrm{min}}}
\DeclareMathOperator{\epsmin}{\eps^{\mathrm{min}}}
\DeclareMathOperator{\slambdamin}{\prescript{s}{}{\lambda}^{\mathrm{min}}}
\DeclareMathOperator{\sepsmin}{\prescript{s}{}{\eps}^{\mathrm{min}}}
\DeclareMathOperator{\seps}{\prescript{s}{}{\eps}}
\DeclareMathOperator{\slambdamax}{{}^s\lambda^{\mathrm{max}}}
\DeclareMathOperator{\slambda}{\prescript{s}{}{\lambda}}
\DeclareMathOperator{\sepsmax}{\prescript{s}{}{\eps}^{\mathrm{max}}}

\newcommand{\Cmax}{C^{\mathrm{max}}}
\newcommand{\Emax}{\mathcal E^{\mathrm{max}}}
\newcommand{\Cmin}{C^{\mathrm{min}}}
\newcommand{\Emin}{\mathcal E^{\mathrm{min}}}
\newcommand{\emax}{e^{\mathrm{max}}}
\newcommand{\emin}{e^{\mathrm{min}}}

\DeclareMathOperator{\lambdaodd}{\lambda^{\mathrm{odd}}}
\DeclareMathOperator{\lambdaoddmax}{\lambda^{\mathrm{odd,max}}}
\DeclareMathOperator{\lambdaeven}{\lambda^{\mathrm{even}}}
\DeclareMathOperator{\lambdaevenmax}{\lambda^{\mathrm{even,max}}}

\DeclareMathOperator{\eodd}{e^{\mathrm{odd}}}

\DeclareMathOperator{\epsodd}{\eps^{\mathrm{odd}}}

\DeclareMathOperator{\psimin}{\psi^{\mathrm{min}}}

\DeclareMathOperator{\triv}{\text{triv}}
\DeclareMathOperator{\GSpr}{GSpr}
\DeclareMathOperator{\IM}{\mathsf{IM}}
\newcommand{\tIM}{\mathsf{IM}'}
\DeclareMathOperator{\gIM}{\mathbb{IM}}
\newcommand{\tgIM}{\mathbb{IM}'}
\DeclareMathOperator{\AZ}{\mathsf{AZ}}

\DeclareMathOperator{\Unip}{Unip}
\DeclareMathOperator{\diag}{diag}

\DeclareMathOperator{\pr}{pr}
\DeclareMathOperator{\sS}{\mathscr S}
\DeclareMathOperator{\sK}{\mathscr K}

\DeclareMathOperator{\sB}{\mathscr B}
\DeclareMathOperator{\sX}{\mathscr X}
\DeclareMathOperator{\sC}{\mathscr C}
\DeclareMathOperator{\cP}{\mathcal P}
\DeclareMathOperator{\cR}{\mathcal R}

\DeclareMathOperator{\cN}{\mathcal N}
\DeclareMathOperator{\cE}{\mathcal E}

\DeclareMathOperator{\cH}{\mathcal H}
\DeclareMathOperator{\cM}{\mathcal M}
\DeclareMathOperator{\sM}{\mathscr M}
\DeclareMathOperator{\cY}{\mathcal Y}
\DeclareMathOperator{\cMtemp}{\mathcal M_{\rm{temp}}}
\DeclareMathOperator{\cMtempreal}{\mathcal M_{\rm{temp,real}}}

\DeclareMathOperator{\cQ}{\mathcal Q}

\DeclareMathOperator{\cL}{\mathcal L}
\DeclareMathOperator{\cV}{\mathcal V}

\DeclareMathOperator{\NNG}{\mathcal{N}_{G}}

\DeclareMathOperator{\NNg}{\mathcal{N}_{G}}

\newcommand{\dG}{G^\vee}
\newcommand{\dM}{M^\vee}

\DeclareMathOperator{\fS}{\mathfrak S}

\DeclareMathOperator{\fg}{\mathfrak g}

\DeclareMathOperator{\dg}{\mathfrak g^\vee}

\DeclareMathOperator{\fz}{\mathfrak z}

\DeclareMathOperator{\ft}{\mathfrak t}
\DeclareMathOperator{\fs}{\mathfrak s}
\DeclareMathOperator{\fm}{\mathfrak m}
\DeclareMathOperator{\fn}{\mathfrak n}

\DeclareMathOperator{\sfk}{\mathsf k}
\DeclareMathOperator{\aG}{\bm{\mathsf G}}

\DeclareMathOperator{\Gk}{\bm{\mathsf G}(\mathsf k)}

\DeclareMathOperator{\perfp}{\langle\cdot,\cdot\rangle}
\DeclareMathOperator{\alphac}{\check\alpha}
\DeclareMathOperator{\cG}{\mathcal G}
\DeclareMathOperator{\cA}{\mathcal A}
\DeclareMathOperator{\cZ}{\mathcal Z}

\DeclareMathOperator{\stab}{stab}

\DeclareMathOperator{\Mod}{-mod}
\DeclareMathOperator{\Rep}{Rep}
\DeclareMathOperator{\Lie}{Lie}

\DeclareMathOperator{\cindKG}{c-Ind}

\DeclareMathOperator{\Phitemp}{\Phi_{\rm{temp}}}

\newcommand{\transp}[1]{{}^t{#1}}
\newcommand{\transps}[1]{{}^s{#1}}

\newcommand{\perf}[2]{\langle#1,#2\rangle}

\theoremstyle{definition}
\newtheorem{definition}{Definition}[section]
\newtheorem{example}[definition]{Example}
\newtheorem{remark}[definition]{Remark}

\newtheorem*{claim*}{Claim}

\theoremstyle{theorem}
\newtheorem{theorem}[definition]{Theorem}
\newtheorem*{theorem*}{Theorem}
\newtheorem{lemma}[definition]{Lemma}
\newtheorem{proposition}[definition]{Proposition}
\newtheorem{corollary}[definition]{Corollary}

\numberwithin{equation}{section}

\title[Iwahori--Matsumoto dual for tempered representations]{The Iwahori--Matsumoto dual for tempered representations of Lusztig's geometric Hecke algebras}
\author{Ruben La}
\address{Department of Mathematics, The University of Hong Kong, Run Run Shaw Building, Pok Fu Lam Road, Hong Kong}
\email{\href{mailto:ruben.la@outlook.com}{ruben.la@outlook.com}}

\begin{document}

\maketitle
\begin{abstract}
The graded Iwahori--Matsumoto involution $\mathbb{IM}$ is an algebra involution on a graded Hecke algebra closely related to the more well-known Iwahori--Matsumoto involution on an affine Hecke algebra.
It induces an involution on the Grothendieck group of complex finite-dimensional representations of $\mathbb{H}$.
When $\mathbb{H}$ is a geometric graded Hecke algebra (in the sense of Lusztig) associated to a connected complex reductive group $G$, the irreducible representations of $\mathbb{H}$ are parametrised
by a set $\mathcal{M}$ consisting of certain $G$-conjugacy classes of quadruples $(e,s,r_0,\psi)$ where $r_0 \in \mathbb{C}$, $e \in \mathrm{Lie}(G)$ is nilpotent, $s \in \mathrm{Lie}(G)$ is semisimple, and $\psi$ is some irreducible representation of the group of components of the simultaneous centraliser of $(e,s)$ in $G$.
Let $\bar Y$ be an irreducible tempered representation of $\mathbb{H}$ with real infinitesimal character.
Then $\mathbb{IM}(\bar Y) = \bar Y(e',s,r_0,\psi')$ for some $(e',s,r_0,\psi') \in \mathcal{M}$.
The main result of this paper is to give an explicit algorithm that computes the $G$-orbit of $e'$ for $G = \mathrm{Sp}(2n,\mathbb{C})$ and $G = \mathrm{SO}(N,\mathbb{C})$. 
As a key ingredient of the main result, we also prove a generalisation of the main theorems of Waldspurger 2019 (for $\mathrm{Sp}(2n,\mathbb{C})$) and La 2024 (for $\mathrm{SO}(N,\mathbb{C})$) regarding certain maximality properties of generalised Springer representations.
\end{abstract}

\tableofcontents

\section*{Introduction}

Let $G$ be a complex connected reductive group and fix a maximal torus $T$. Let $S_G$ be the set of triples $(L,C,\cL)$ where $L$ is a standard Levi subgroup of $G$, $C$ is a unipotent class in $L$ and $\cL$ is an $L$-equivariant local system on $C$.
Let $r$ be an indeterminate. To each $(L,C,\cL) \in S_G$, a $\C[r]$-algebra $\bH(G,L,C,\cL)$, called a geometric graded Hecke algebra, was attached in \cite{lusztig1988cuspidal, lusztig1995cuspidal}, and it was shown there that the irreducible representations of $\bH(G,L,C,\cL)$ are in bijection with the set $\cM$ of $G$-conjugacy classes of quadruples $(e,s,r_0,\psi)$ with $r_0 \in \C$, $e\in \fg := \Lie(G)$ nilpotent, $s \in \fg$ semisimple such that $\ad(s)e=2r_0e$, and $\psi \in A_G(e,s)_{(L,C,\cL)}^\wedge$, where $ A_G(e,s)_{(L,C\cL)}^\wedge$ is a certain subset of the set of isomorphism classes of irreducible representations of $A_G(e,s) = Z_G(e,s)/Z_G(e,s)^\circ$ (see \Cref{subsec:GHAcuspls} for the definition). For $(e,s,r_0,\psi) \in \cM$, denote by $\bar Y(e,s,r_0,\psi)$ the corresponding irreducible representation. 
The central character of $\bar Y(e,s,r_0,\psi)$ is determined by the $G$-orbit of $s$ and $r_0$ and we denote by $\Irr_{s,r_0}\bH(G,L,C,\cL)$ the set of isomorphism classes of irreducible representations of $\bH(G,L,C,\cL)$ with this central character.
The graded Iwahori--Matsumoto involution $\gIM$ is an algebra involution on $\bH$ and induces a corresponding involution (which we also denote by $\gIM$) on the Grothendieck group of finite-dimensional representations of a graded Hecke algebra that sends irreducibles to irreducibles up to sign. 
In this paper, $\gIM$ is defined differently than what is sometimes done in the literature (see for instance \cite{evens1997fourier}) so that it preserves central characters.
The graded Hecke algebra $\bH(G,L,C,\cL)$ contains a copy of the group algebra $\C[W_L]$ of the relative Weyl group $W_L = N_G(L)/L$ as a subalgebra. The multiplicities of the irreducible subrepresentations of the restriction $\bar Y(e,s,r_0,\psi)|_{W_L}$ to $\C[W_L]$ are completely described by Green functions \cite[\S10.13]{lusztig1995cuspidal}. 
For $e' \in \dg$ nilpotent and $\psi'$ is an irreducible representation of $A_G(e') = Z_G(e')/Z_G(e')^\circ$, denote by $\GSpr(e',\psi')$ the corresponding generalised Springer representation, which is an irreducible representation of $W_{L'}$ for some Levi subgroup $L'$ of $G$ with a cuspidal pair. We call $G\cdot e'$ the Springer support of the representation $\GSpr(e',\psi')$ and we note that the nilpotent orbits of $\dg$ are partially ordered by the closure order.
The following result is a key ingredient for the main result of this paper.

\begin{theorem}\label{thm:maintheorems}
Suppose $G = \Sp(2n,\C)$ or $G = \SO(N,\C)$ and let $(e,s,r_0,\psi) \in \cM$ such that $\bar Y(e,s,r_0,\psi)$ is tempered and $s$ is real (i.e. polar decomposition of $s$ has trivial compact part). 
Then we have $A_G(e,s) = A_G(s)$ (by \cite[(4.3.1) Lemma]{reeder2002isogenies})  and
\begin{enumerate}
\item $\bar Y(e,s,r_0,\psi)|_{W_L}$ has a unique subrepresentation $\bar\rho = \GSpr(\bar e,\bar\psi)$ with multiplicity $1$ and whose Springer support is maximal among the the Springer supports of all irreducible subrepresentations of $\bar Y(e,s,r_0,\psi)$ (\Cref{thm:maxeven} for $\SO(N,\C)$, \Cref{thm:maxSp} for $\Sp(2n,\C)$).
\item $\bar\rho \otimes \sgn$ has minimal Springer support among the Springer supports of all irreducible subrepresentations of $\bar Y(e,s,r_0,\psi)|_{W_L} \otimes \sgn$ (\Cref{thm:mineven} for $\SO(N,\C)$, \Cref{thm:minSp} for $\Sp(2n,\C)$).
\end{enumerate}
\end{theorem}
\Cref{thm:maxeven} and \Cref{thm:mineven}, are generalisations of \cite[Theorem 4.2, 4.6]{la2022maximality}, the difference being that the former two theorems do not have any assumptions on $e$, whereas in the latter two theorems it is assumed that $e$ is parametrised by an orthogonal partition with only odd parts. 
Similarly, \Cref{thm:maxSp}, and \Cref{thm:minSp} are generalisations of \cite[Th\'eor\`eme 4.5, 4.7]{waldspurger2019proprietes}, dropping the assumption that $e$ is parametrised by a symplectic partition with only even parts.
In \cite[\S4.3]{la2022maximality} an attempt was already made to generalise these results to arbitrary $e$, but it was only successful with the additional assumption that $(e,\rho)$ are of `ordinary' Springer type, that is, they appear in the Springer correspondence but not the generalised Springer correspondence.
In this paper, the generalisations are proved using almost entirely the same approach as in \emph{loc. cit.}, using the additional result \cite[Proposition 3.22(a)]{aubert2018graded} that allows us to drop the assumption that $(e,\rho)$ is of ordinary Springer type.

The main result of the paper is the following
\begin{theorem}\label{thm:maintheorem}
Suppose $G = \Sp(2n,\C)$ or $G = \SO(N,\C)$ and let $(e,s,r_0,\psi) \in \cM$ such that $\bar Y(e,s,r_0,\psi)$ is a tempered representation and $s$ is a real semisimple element of $G$. 
Let  $(e',s',r_0',\psi') \in \cM$ such that $s'=s$, $r_0'=r_0$ and $\gIM(\bar Y(e,s,r_0,\psi)) = \bar Y(e',s,r_0,\psi')$. Then $A_G(e',s) = A_G(e')$ and $\bar \rho \otimes \sgn = \GSpr(e',\psi')$ (\Cref{prop:gIMminimal}).
\end{theorem}
In \Cref{sec:algSO} and \Cref{sec:algSp}, we give an explicit algorithm from \cite[\S5]{la2022maximality} and \cite[\S5]{waldspurger2019proprietes} that computes $G\cdot e'$ and $\psi'$ for $\SO(N,\C)$ and $\Sp(2n,\C)$ respectively.
Using similar results as \Cref{thm:maintheorems} and \Cref{thm:maintheorem} for the case that $G$ is an exceptional group -- verified case by case using the Chevie package of GAP -- we will also give tables for the cases that $G$ is an exceptional group in \Cref{app:tablesexceptional} that were computed also using the Chevie package of GAP \cite{schonert1992gap,geck1996chevie,michel2015development}.

Next we briefly explain the connection of our results to the more well-known Iwahori--Matsumoto involution $\IM$ for an affine Hecke algebra $\cH$ and the Aubert--Zelevinsky involution for $p$-adic groups.
We view $\cH$ as a $\C[v,v^{-1}]$-algebra where $v$ is an indeterminate. 
The version of $\IM$ that we consider is the same as the involution considered in \cite{kato1993duality}.
We also denote by $\IM$ the induced involution on the Grothendieck group of finite-dimensional representations of $\cH$. 
Lusztig attached to a connected complex reductive group $\dG$ and a triple $(L,C,\cL)$, where $L$ is a standard pseudo-Levi subgroup (i.e. centraliser of an element in $T$) of $\dG$ and $(C,\cL)$ is a cuspidal pair in $L$, a `geometric' affine Hecke algebra $\cH(\dG,L,C,\cL)$ \cite{lusztig1995classification}. 
Central characters of $\cH(\dG,L,C,\cL)$ are parametrised by a pair $(G\cdot s,v_0)$ where $s \in \dG$ is semisimple and $v_0 \in \C^\times$. Let $\Irr_{s,v_0}\cH(\dG,L,C,\cL)$ denote the set of isomorphism classes of irreducible representations of $\cH(\dG,L,C,\cL)$ with central character corresponding to $(G\cdot s,v_0)$.
Let $s_r$ and $s_c$ denote the real and compact parts (polar decomposition) of $s$ respectively.
Using the reduction theorems in \cite{lusztig1989affine}, Lusztig constructed a bijection $\Omega$ between 
$\Irr_{s,v_0}\cH(\dG,L,C,\cL)$ 
and 
$\Irr_{\log s_r,r_0} \bH(Z_{\dG}(s_c),L,C,\cL)$.
The author believes that it is well-known that $\Omega \circ \IM = \gIM \circ \Omega$, but is unaware of an explicit proof in the literature (\cite{evens1997fourier} discusses this but for a certain twist of $\gIM$), so we include a proof (\Cref{corollary:diagramIM}).
Using the parametrisation of irreducible representations of $\bH(Z_{\dG}(s_c),L,C,\cL)$ described in the first paragraph, the irreducible representations of $\cH(\dG,L,C,\cL)$ on which $v$ acts as a real positive number can be parametrised by $\dG$-orbits of quadruples $(s,e,v_0,\psi)$ where $s \in \dG$ is semisimple with $L \subseteq Z_{\dG}(s_c)$, $e \in \dg  = \Lie(\dG)$ is nilpotent such that $\ad(s)e = v_0^2e$, and $\psi \in A_G(e,s_r)_0^\wedge$.
For $(e,s,v_0,\psi)$ as above, denote the corresponding irreducible representation by $\cY(e,s,v_0,\psi)$.
By the compatibility between $\IM$ and $\gIM$ mentioned above we have $\cY(e',s',v_0,\psi') = \IM(\cY(e,s,v_0,\psi))$ if and only if $\bar Y(e',\log s_r',\log v_0,\psi') = \gIM(\bar Y(e,\log s_r,\log v_0,\psi))$ (\Cref{subsec:gIMtoIM}).
 
Let $\sf k$ be a nonarchimedean local field with finite residue field of cardinality $q$ and let $\aG$ be a connected reductive algebraic group of adjoint type defined and split over $\sf k$.
The Aubert--Zelevinsky involution $\AZ$ is an involution on the Grothendieck group of finite-length representations of $\aG(\sfk)$. 
Lusztig constructed in \cite[\S5]{lusztig1995classification} a bijection between the set $\Unip(\aG)$ of unipotent representations of all inner forms of the split form of $G$ and the set $\bigsqcup_{L,C,\cL}\Irr_{q^{1/2}}\cH(\dG,L,C,\cL)$ where $L$ is a standard pseudo-Levi subgroup of $\dG$ and $(C,\cL)$ is a cuspidal pair in $L$. 
We obtain a bijection between $\Unip(\aG)$ and the set $\Phi(\dG)$ of $\dG$-conjugacy classes of triples  $(e,s,\phi)$, where $s \in \dG$ is semisimple, $e \in \dg$ is nilpotent such that $\ad(s)e = qe$ and $\phi$ is an irreducible representation of $A_{\dG}(s,e)$. 
In \Cref{sec:algAZ}, we use our algorithm for $\gIM$ to obtain an algorithm that computes the $\dG$-orbit of the nilpotent element in the triple of $\Phi(\dG)$ that parametrises $\AZ(X)$ for a tempered unipotent representation $X$ of $\aG = \SO(2n+1)$. We note that an explicit algorithm for $\AZ$ for all irreducible representations of $\SO(2n+1)$ and $\Sp(2n)$ was given in \cite{atobe2023explicit} using a completely different approach.
Finding an algorithm for $\AZ$ for $\aG =\Sp(2n)$ and $\aG = \SO(2n)$ requires a different approach since these groups are not adjoint and so we cannot directly use Lusztig's result as we did for $\SO(2n+1)$. In this case the author believes a possible approach would be to use results from \cite{solleveld2023local} (as well as results from \cite{aubert2018graded}, \cite{aubert2017affine}, and \cite{aubert2018generalizations} which \cite{solleveld2023local} is strongly based on). In this setting, we would also need analogues of our results for $\SO(N,\C)$ and $\Sp(2n,\C)$ in \Cref{sec:max} for the complex spin groups, 
which are unknown to the author

\subsection*{Acknowledgements}
This majority of this paper is part of my DPhil (PhD) thesis, which I wrote at the University of Oxford.
I would like to thank my supervisor Dan Ciubotaru for his support and invaluable guidance in this project. 
I would further like to thank Jonas Antor and Emile Okada for many helpful discussions, and for allowing me to use some results that came from our joint work for an upcoming paper. 
I also want to thank Anne-Marie Aubert and Yakov Kremnitzer for useful suggestions.

\section*{Notation}\label{sec:notation}

Let $\sf k$ be a nonarchimedean local field with with finite residue field $k := \F_q$ of cardinality $q$. 
Let $\aG$ be a connected reductive algebraic group defined and split over $\sf k$.
Let $\dG$ denote the complex reductive group whose root datum is dual to the root datum of $\aG$ and let $\dg$ denote its Lie algebra.

For an $l$-group $H$ (i.e. a Hausdorff group such that the identity has a basis of neighbourhoods which are open compact groups), denote by $\Rep(H)$ the category of smooth complex representations of $H$ and let $\cR(H)$ denote its Grothendieck group of the subcategory of $\Rep(H)$ of finite-length representations.
For a $\C$-algebra $A$, denote by $\cR(A)$ the Grothendieck group of finite-dimensional $A$-modules.

For any abstract group $\Gamma$, write $\Gamma^\wedge$ for the set of isomorphism classes of irreducible representations of $\Gamma$. If $\Gamma$ acts on a set $X$, we write $X^\Gamma$ for the set of fixed points of $X$.

\section{Generalised Springer correspondence}\label{sec:GSC}

Let $G$ be a connected complex reductive group with Lie algebra $\fg$. 
Let $\cN$ be the set of nilpotent orbits in $\fg$ and $\NNg$ the set of pairs $(C,\E)$, where $C \in \cN$ and $\cE$ is an irreducible $G$-equivariant local system on $C$. 
Denote by $S_G$ the set of triples $(L,C_L,\cL)$ where  $L$ is a Levi subgroup of $G$ containing a cuspidal pair $(C_L,\cL)$
(\cite[2.4 Definition]{lusztig1984intersection}). 
Lusztig showed that $W_L := N_G(L)/L$ is a finite Coxeter group. We call $W_L$ a \emph{relative Weyl group} of $G$. 
The generalised Springer correspondence is a certain bijection 
\begin{equation}\label{eq:genspringergeneral}
\GSpr \colon \NNg \to \bigsqcup_{(L,C_L,\cL) \in S_G} W_{L}^\wedge.
\end{equation}
Given $(C,\cE) \in \NNg$, we call $\GSpr(C,\cE)$ a $W_{L}$ \emph{generalised Springer representation} for $G$ associated to $(C,\cE)$.
By \eqref{eq:genspringergeneral}, each $(C,\cE)$ of $\NNg$ corresponds uniquely to a triple $(L,C_L,\cL) \in S_G$, which we call the \emph{cuspidal support} of $(C,\E)$ and we write $\csup(C,\cE):=(L,C_L,\cL)$.

For $C \in \cN$ and any $e \in C$, the set of irreducible $G$-equivariant local systems on $C$ is in bijection with the set of irreducible representations of $A(e) = Z_G(e)/Z_G(e)^\circ$, the component group of $e$. Note that $A(e)$ is uniquely determined by $C$ up to isomorphism, so we write $A(C)$. 
Abusing notation, we also write elements of $\NNG$ as $(e,\phi)$ where $e\in C$ and $\phi \in A(e)^\wedge$ rather than $(C,\cE)$ and we write
\begin{align}
\GSpr(e,\phi) = \GSpr(C,\cE), \quad \csup(e,\psi) = \csup(C,\phi).
\end{align}

Next, we will describe the generalised Springer correspondence for $G = \Sp(2n,\C)$ and $G = \SO(N,\C)$ in terms of symbols. The notation will be the same as in \cite[\S 4]{waldspurger2019proprietes} and is slightly different than the notation originally used in \cite{lusztig1984intersection}. We note that the exposition is almost exactly the same as in \cite[\S1,\S2]{la2022maximality}.

\subsection{Combinatorics of partitions}\label{subsec:partitions}
Let $\cR$ be the set of decreasing sequences $\lambda = (\lambda_1, \lambda_2, \dots)$ of real numbers such that for each $r \in \R$, there exist only finitely many $i \in \N$ such that $\lambda_i \geq r$.
For each $r \in \R$ and $\lambda \in \cR$, define $\mult_\lambda(r) = \#\{i \in \N \colon \lambda_i = r\}$.
For each $c \in \N$, we define $S_c(\lambda) = \lambda_1 + \dots + \lambda_c$. 
For $\lambda,\lambda' \in \cR$, write $\lambda \leq \lambda'$ if $S_c(\lambda) \leq S_c(\lambda')$ for all $c \in \N$.
For $\lambda, \lambda' \in \cR$, define $\lambda + \lambda' = (\lambda_1 + \lambda_1', \lambda_2+\lambda_2',\dots)$, and define $\lambda \sqcup \lambda'$ to be the unique element of $\cR$ such that for each $n \in \N$, we have $\mult_{\lambda \sqcup \lambda'}(n) = \mult_{\lambda}(n) + \mult_{\lambda'}(n)$.

For $a,s \in \R$, define $[a,-\infty[_s = (a,a-s,a-2s,\dots) \in \cR$, and for $A,B\in\R$, $\mu,\nu\in\cP$, define
\begin{align}\label{eq:Lambda}
\Lambda_{A,B;s}(\mu,\nu) &= (\mu + [A,-\infty[_s) \sqcup (\nu + [B,-\infty[_s) \in \cR.
\end{align}

A sequence $\lambda \in \cR$ is called a \emph{partition} if it is a finite decreasing sequence of non-negative integers and let $\mathcal{P} \subseteq \cR$ be the set of all partitions.
For $N \in \N$, we say that $\lambda$ is a \emph{partition of $N$} if $N = S(\lambda)$. Let $\cP(N)$ denote the set of partitions of $N$. 

\subsection{Symplectic group $\Sp(2n,\C)$}
Let $n \in \Z_{\geq0}$ and $G = \Sp(2n,\C)$ 
A partition $\lambda$ of $2n$ is called \emph{symplectic} if each odd part of $\lambda$ occurs with even multiplicity. We write $\Psymp(2n)$ for the set of symplectic partitions of $2n$. Let $\jordbp(\lambda) = \{i \in \N \colon i \text{ is even, } \mult_\lambda(i) \neq 0\}$. 
Denote by $\PPsymp(N)$ the set of pairs $(\lambda,\eps)$ with $\lambda \in \Port(2n)$ and $\eps \in \{\pm1\}^{\Delta(\lambda)}$. 

It is well-known that $\Psymp(2n)$ is in $1$--$1$ correspondence with $\cN$ and that if $C \in \cN$ is parametrised by $\lambda \in \Psymp(2n)$, then $\{\pm1\}^{\Delta(\lambda)}$ is in $1$--$1$ correspondence with the set of irreducible $G$-equivariant local systems on $C$ and also with $A(C)^\wedge$. 
We thus have a bijections $\NNG \to \PPsymp(2n)$. 

Let $(\lambda,\eps) \in \PPsymp(2n)$. 
Let $\lambda' = \lambda + [-1,-\infty[_{1} = (\lambda_1,\lambda_2-1,\lambda_3-2,\dots)$. There exist sequences $z,z' \in \cR$ with integer terms such that $\lambda' = (2z_1,2z_2,\dots) \sqcup (2z_1'+1,2z_2'+1,\dots)$. 
Let $A^\# = z' + [1,-\infty[_1$ and $B^\# = z + [0,-\infty[_1$. Then $A_1^\# \geq B_1^\# \geq A_2^\# \geq B_3^\# \geq \dots$.
A finite subset of $\Z$ is called an interval if it is of the form $\{i,i+1,i+2,\dots,j\}$ for some $i,j \in \Z$.
Let $C$ be the collection of intervals $I$ of $A\triangle B = (A \cup B) \setminus (A \cap B)$, with the property that for any $i \in (A \Delta B) \setminus I$, the set $I \cup \{i\}$ is not an interval. 
Then $C$ is in bijection with $\jordbp(\lambda)$.
There is an obvious total order on $C$: for $I, I' \in C$, $I$ is larger than $I'$ if any element of $I$ is larger than any element of $I'$.
Let $\jordbp(\lambda) \to C \colon i \mapsto C_i$ be the unique increasing bijection.
Let $t = t(\lambda)$. 
For $i \in \{1,\dots,t\}$, write $\eps(i) = \eps_{\lambda_i}$. For $u \in \{\pm1\}$, define
\begin{align}
J^u &= \{ i \in \{1,\dots,t\} \colon \eps(i)(-1)^{i} = u\}.
\end{align}
Define the \emph{symbol} of $(\lambda,\eps)$ to be the ordered pair $S_{\lambda,\eps} = (A_{\lambda,\eps},B_{\lambda,\eps})$, where
\begin{align}
A_{\lambda,\eps} &=
(A^\# \setminus \bigcup_{i \in \jordbp(\lambda); \eps_i = -1} (A^\# \cap C_i)) \cup ( \bigcup_{i \in \jordbp(\lambda); \eps_i = -1} (B^\# \cap C_i)),
\\
B_{\lambda,\eps} &=
(B^\# \setminus \bigcup_{i \in \jordbp(\lambda); \eps_i = -1} (B^\# \cap C_i)) \cup ( \bigcup_{i \in \jordbp(\lambda); \eps_i = -1} (A^\# \cap C_i)).
\end{align}

For $m \in \N$, let $W(B_m) = W(C_m) = S_m \ltimes (\mathbb Z / 2 \mathbb Z)^m$.
We have a bijection $\cP_2(m) \to W(C_m)^\wedge \colon (\alpha,\beta) \mapsto \rho_{(\alpha,\beta)}$.
Let $k \in \N$ such that $k(k+1)\leq 2n$ and let $(\alpha,\beta) \in \mathcal P_2(n - \frac{k(k-1)}{2})$.
If $k$ is even (resp. odd), let $A_{\alpha,\beta;k} = \alpha + [k,-\infty[_2$ and $B_{\alpha,\beta;k} = \beta + [-k-1,-\infty[_2$ (resp. $A_{\alpha,\beta;k} = \beta + [-k-1,-\infty[_2 $ and $B_{\alpha,\beta;k} = \alpha + [k,-\infty[_2$) and define the \emph{symbol of} $(\alpha,\beta)$ to be $S_{\alpha,\beta;k} = (A_{\alpha,\beta;k},B_{\alpha,\beta;k})_k$.
Let 
\begin{align}
\mathcal W = \bigsqcup_{k \in \N, k(k+1)\leq 2n} W(C_{n-\frac{k(k+1)}{2}})^\wedge.
\end{align} 
The $W(C_{n-\frac{k(k+1)}{2}})$ form all the relative Weyl groups of $G$ (see the start of this section). 
Hence the generalised Springer correspondence is a bijection
\begin{align}
\GSpr \colon \NNG \to \mathcal W.
\end{align}
Using the parametrisations of $\NNG$ and $\mathcal W$ described above, we rephrase the generalised Springer correspondence as follows \cite[\S12]{lusztig1984intersection}.

\begin{theorem}
\label{thm:gscSp}
Let $n \in \N$.
For each $(\lambda,\eps) \in \PPsymp(2n)$, there exist unique $k \in \Z_{\geq0}$ and unique $(\alpha,\beta) \in \mathcal P_2(n-\frac{k(k+1)}{2})$ such that $(A_{\lambda,\eps},B_{\lambda,\eps}) = (A_{\alpha,\beta;k},B_{\alpha,\beta;k})$. Conversely, for each $k \in \Z_{\geq0}$ such that $k(k+1) \leq 2n$, and for each $(\alpha,\beta) \in \mathcal P_2(n-\frac{k(k+1)}{2})$, there exists a unique $(\lambda,\eps) \in \PPsymp(2n)$ such that $(A_{\lambda,\eps},B_{\lambda,\eps}) = (A_{\alpha,\beta;k},B_{\alpha,\beta;k})$. Thus we have a bijection
\begin{align}
\Phi_N \colon \PPsymp(N) \to \bigsqcup_{k \in \Z_{\geq0},\, k(k+1) \leq 2n} \mathcal P_2(n-\frac{k(k+1)}{2}).
\end{align}
\end{theorem}

\subsection{Special orthogonal group $\SO(N,\C)$}\label{subsec:gscSO}
Let $N \in \Z_{\geq0}$ and suppose $G = \SO(N,\C)$
A partition $\lambda$ of $N$ is called \emph{orthogonal} if each even part of $\lambda$ occurs with even multiplicity. We write $\Port(N)$ for the set of orthogonal partitions of $N$. Let $\jordbp(\lambda) = \{i \in \N \colon i \text{ is odd, } \mult_\lambda(i) \neq 0\}$. 
Let $F_2$ be the field with two elements.
Let $F_2[\Delta(\lambda)] = \{\pm1\}^{\Delta(\lambda)}$ be the set of maps $\eps \colon \jordbp(\lambda) \to \{\pm1\} \colon \lambda_i \mapsto \eps_{\lambda_i}$, considered as an $F_2$-vector space. 
Let $F_2[\Delta(\lambda)]'$ be the quotient of $F_2[\Delta(\lambda)]$ by the line spanned by the sum of the canonical basis of this $F_2$-vector space.
Denote by $\PPort(N)$ the set of pairs $(\lambda,[\eps])$ where $\lambda \in \Port(N)$ and $[\eps] \in F_2[\jordbp(\lambda)]'$ is the image of $\eps \in F_2[\Delta(\lambda)]$ under the quotient map.
We define $\PPPort(N)$ to be the set of pairs $(\lambda,\eps)$ with $\lambda \in \Port(N)$ and $\eps \in F_2[\Delta(\lambda)]$.

It is well-known that $\Port(N)$ is in $1$--$1$ correspondence with $\cN$, except the orthogonal partitions of $N$ that only have even parts correspond to precisely two (degenerate) nilpotent orbits.
This correspondence has the following property.
Suppose $C, C' \in \cN$ are parametrised by $\lambda,\lambda' \in \Port(N)$ respectively. Then $C \preceq C'$ (closure order), if and only if $\lambda \leq \lambda'$.
We say that $\lambda$ is \emph{degenerate} if $\lambda$ only has even parts and we call $\lambda$ \emph{non-degenerate} otherwise.
If $C \in \cN$ is parametrised by $\lambda \in \Port(N)$, then $F_2[\Delta(\lambda)]'$ is in $1$--$1$ correspondence with
the set of irreducible $G$-equivariant local systems on $C$ and also with $A(C)^\wedge$. 
We obtain a surjective map $\NNG \to \PPort(N)$ such that the preimage of $(\lambda,\eps)$ has one element (resp. two elements) if $\lambda$ is non-degenerate (resp. degenerate). 
We denote the preimage of $(\lambda,[\eps])$ by $\{(C_\lambda^+,\E_{[\eps]}^+),(C_\lambda^-,\E_{[\eps]}^-)\}$. If $\lambda$ is non-degenerate, we write $(C_\lambda,\E_{[\eps]}) = (C_\lambda^+,\E_{[\eps]}^+) = (C_\lambda^-,\E_{[\eps]}^-)$, i.e. we drop the $+$ and $-$ from the notation. 
Note that $F_2[\Delta(\lambda)]'$ is also in $1$--$1$ correspondence with $A(C)^\wedge$.
Let $(\lambda,[\eps]) \in \PPort(N)$. 
Let $\lambda' = \lambda + [0,-\infty[_{1} = (\lambda_1,\lambda_2-1,\lambda_3-2,\dots)$. There exist sequences $z,z' \in \cR$ with integer terms such that $\lambda' = (2z_1,2z_2,\dots) \sqcup (2z_1'-1,2z_2'-1,\dots)$. 
Let $A^\# = z' + [0,-\infty[_1$ and $B^\# = z + [0,-\infty[_1$. Then $A_1^\# \geq B_1^\# \geq A_2^\# \geq B_3^\# \geq \dots$.
A finite subset of $\Z$ is called an interval if it is of the form $\{i,i+1,i+2,\dots,j\}$ for some $i,j \in \Z$.
Let $C$ be the collection of intervals $I$ of $(A^\# \cup B^\#) \setminus (A^\# \cap B^\#)$, with the property that for any $i \in (A^\# \cup B^\#) \setminus ((A^\# \cap B^\#)\cup I)$, we have that $I \cup \{i\}$ is not an interval. 
Then $C$ is in bijection with $\jordbp(\lambda)$. 
Consider the total order $>$ on $C$: for $I, I' \in C$, $I > I'$ if any element of $I$ is larger than any element of $I'$.
Let $\jordbp(\lambda) \to C \colon i \mapsto C_i$ be the unique increasing bijection.
Let $t = t(\lambda)$. 
For $i \in \{1,\dots,t\}$, write $\eps(i) = \eps_{\lambda_i}$. For $u \in \{\pm1\}$, define
\begin{equation}
J^u = \{ i \in \{1,\dots,t\} \colon \eps(i)(-1)^{i} = u\}.
\end{equation}
Let $M(\lambda,\eps) = M = |J^1| - |J^{-1}|$ and let $w(\lambda,\eps) = w = \sgn(M+1/2)$ so that $w= 1$ if $M = 0$ and $w = \sgn M$ if $M \neq 0$. 
Define
\begin{align}
A_{\lambda,\eps} &=
(A^\# \setminus \bigcup_{i \in \jordbp(\lambda); \eps_i = -w} (A^\# \cap C_i)) \cup ( \bigcup_{i \in \jordbp(\lambda); \eps_i = -w} (B^\# \cap C_i)),
\\
B_{\lambda,\eps} &=
(B^\# \setminus \bigcup_{i \in \jordbp(\lambda); \eps_i = -w} (B^\# \cap C_i)) \cup ( \bigcup_{i \in \jordbp(\lambda); \eps_i = -w} (A^\# \cap C_i)).
\end{align}
If $M \neq 0$ and $-\eps$ is the other representative of $[\eps]$, we have $A_{\lambda,\eps} = A_{\lambda,-\eps}$ and $B_{\lambda,\eps} = B_{\lambda,-\eps}$, so we define the \emph{symbol} of $(\lambda,[\eps])$ to be the ordered pair $S_{\lambda,[\eps]} = (A_{\lambda,\eps},B_{\lambda,\eps}) =  (A_{\lambda,-\eps},B_{\lambda,-\eps})$.
If $M = 0$, $A_{\lambda,\eps} = B_{\lambda,-\eps}$ and $B_{\lambda,\eps} = A_{\lambda,-\eps}$, we define the \emph{symbol} of $(\lambda,[\eps)]$ to be the unordered pair $S_{\lambda,[\eps]} = \{A_{\lambda,\eps},B_{\lambda,\eps}\}$.
Let $X$ be a set, $k \in \Z_{\geq 0}$ and $(x,y) \in X \times X$. Then let
\begin{align}\label{eq:k}
(x,y)_k
=
\begin{cases}
(x,y) \in X \times X &\text{if } k > 0,
\\
\{x,y\} \subseteq X &\text{if } k = 0.
\end{cases}
\end{align}
Following this notation, we have $S_{\lambda,[\eps]} = (A_{\lambda,\eps},B_{\lambda,\eps})_{|M|}$.

Define $p_{\lambda,[\eps]} = p_{\lambda,\eps} = A_{\lambda,\eps} \sqcup B_{\lambda,\eps} = A_{\lambda,\eps}^\# \sqcup B_{\lambda,\eps}^\#$.
Two pairs $(A,B), (A',B') \in \mathcal R \times \mathcal R$ are \emph{similar} if $A\sqcup B = A'\sqcup B'$. 
For any $(\lambda,[\eps]),(\lambda',[\eps']) \in \PPort(N)$, their symbols are similar if and only if $\lambda = \lambda'$.

\begin{remark}\label{rem:symbol}
Let $t = t(\lambda)$, $s = \lfloor t/2 \rfloor$ and $k = |M|$.
Let $A' = (A_{\frac{t+d}{2}}+s,A_{\frac{t+d}{2}-1}+s,\dots,A_1+s)$, $B'=(B_{\frac{t-d}{2}}+s,B_{\frac{t-d}{2}-1}+s,\dots,B_1+s)$ and note that these are increasing sequences.
Then $(A',B')_k$ is the usual symbol in the literature.
The \emph{defect} of $(A,B)$ is defined to be $|\#A'-\#B'| = k = |M|$. 
\end{remark}

Let $n \in \N$.
Let $W(D_n) = S_n \ltimes (\mathbb Z / 2 \mathbb Z)^{n-1}$ and $\theta\colon \mathcal P_2(n) \to \mathcal P_2(n)\colon (\alpha,\beta) \mapsto (\beta,\alpha)$ for $(\alpha,\beta) \in \mathcal P_2(n)$. Let $\mathcal P_2(n) / \theta$ be the set of $\theta$-orbits
and let $c_{(\alpha,\beta)}$ be the size of the $\theta$-orbit of $(\alpha,\beta)$.
Identify $\mathcal P_2(n)$ with the set of unordered pairs $\{\alpha,\beta\}$ where $(\alpha,\beta) \in \mathcal P_2(n)$.
Then we can parametrise
\begin{align}
W(D_n)^\wedge = \{\rho_{\{\alpha,\beta\},i} \colon \{\alpha,\beta\} \in \mathcal P_2(n) / \theta, i \in \{1,2/c_{(\alpha,\beta)}\}\}.
\end{align}

Let $k \in \N$ such that $k \equiv N \mod 2$ and $k^2\leq N$. Let $(\alpha,\beta) \in \mathcal P_2((N - k^2)/2)$. We define $A_{\alpha,\beta;k} = \alpha + [k,-\infty[_2$ and $B_{\alpha,\beta;k} = \beta + [-k,-\infty[_2$ and define the \emph{symbol of} $(\alpha,\beta)$ to be $S_{\alpha,\beta;k} = (A_{\alpha,\beta;k},B_{\alpha,\beta;k})_k$.

Let $W$ be the Weyl group of $G$. Then $W = W(B_{(N-1)/2})$ (resp. $W = W(D_{N/2})$) if $N$ is odd (resp. even). Let 
\begin{align}
\mathcal W = W^\wedge \sqcup \bigsqcup_{k \in \Z_{\geq2}, k^2\leq N, k \equiv N \mod 2} W(B_{(N-k^2)/2})^{\wedge}.
\end{align} 
The Weyl groups $W$ and $W(B_{(N-k^2)/2})$ with $k \in \Z_{\geq 2}$, $k \equiv N \mod 2$ and $k^2 \leq N$ form all the relative Weyl groups of $G$ (see the start of this section). 
In this case, the generalised Springer correspondence is a bijection
\begin{align}
\GSpr \colon \NNG \to \mathcal W.
\end{align}
Using the parametrisations of $\NNG$ and $\mathcal W$ described above, we rephrase the generalised Springer correspondence as follows \cite[\S13]{lusztig1984intersection}.
.

\begin{theorem}\label{thm:gscSO}
Let $N \in \N$.
\begin{enumerate}
\item Suppose $N = 2n+1$ is odd. For each $(\lambda,[\eps]) \in \PPort(N)$, there exists a unique odd integer $k(\lambda,[\eps]) = k \in \N$ and a unique pair $(\alpha,\beta) \in \mathcal P_2((N-k^2)/2)$ such that $(A_{\lambda,\eps},B_{\lambda,\eps}) = (A_{\alpha,\beta;k},B_{\alpha,\beta;k})$. Conversely, for each odd $k \in \N$ such that $k^2 \leq N$, and for each $(\alpha,\beta) \in \mathcal P_2((N-k^2)/2)$, there exists a unique $(\lambda,[\eps]) \in \PPort(N)$ such that $(A_{\lambda,[\eps]},B_{\lambda,[\eps]}) = (A_{\alpha,\beta;k},B_{\alpha,\beta;k})$. Thus we have a bijection
\begin{align}
\Phi_N \colon \PPort(N) \to \bigsqcup_{k \in \N \mathrm{odd},\, k^2 \leq N} \mathcal P_2((N-k^2)/2).
\end{align}
\item Suppose $N = 2n$ is even. Let $(\lambda,[\eps]) \in \PPort(N)$. Then one of the following is true:
\begin{itemize}
\item 
There exists a unique $\{\alpha,\beta\} \in {\mathcal P}_2(N/2) / \theta$ with $\alpha \neq \beta$ such that $\{A_{\lambda,[\eps]},B_{\lambda,[\eps]}\} = \{A_{\alpha,\beta;0},B_{\alpha,\beta;0}\}$,
\item There exists a unique even integer $k(\lambda,[\eps]) = k \in \N$ and a unique pair $(\alpha,\beta) \in \mathcal P_2((N-k^2)/2)$ such that $(A_{\lambda,[\eps]},B_{\lambda,[\eps]}) = (A_{\alpha,\beta;k},B_{\alpha,\beta;k})$. 
\end{itemize}
Conversely, we have
\begin{itemize}
\item For each $\{\alpha,\beta\} \in \mathcal P_2(N/2)/\theta$, there exists a unique $(\lambda,[\eps]) \in \PPort(N)$ such that $\{A_{\lambda,[\eps]},B_{\lambda,[\eps]}\} = \{A_{\alpha,\beta;0},B_{\alpha,\beta;0}\}$,
\item For each even $k \in \N$ such that $k^2 \leq N$ and for each $(\alpha,\beta) \in \mathcal P_2((N-k^2)/2)$, there exists a unique $(\lambda,[\eps]) \in \PPort(N)$ such that $(A_{\lambda,[\eps]},B_{\lambda,[\eps]}) = (A_{\alpha,\beta;k},B_{\alpha,\beta;k})$.

\end{itemize}
Thus we have a bijection
\begin{align}
\Phi_N \colon \PPort(N) \to (\mathcal P_2(N/2) / \theta) \sqcup \bigsqcup_{k \in \N \mathrm{even},\,k^2\leq N} \mathcal P_2((N-k^2)/2).
\end{align}
\end{enumerate}
\end{theorem}

\begin{remark}\label{rem:parity}
Note that in \Cref{thm:gscSO}, we have $N \equiv t(\lambda) \equiv k(\lambda,[\eps]) \mod 2$. Furthermore, it is well-known that $k(\lambda,[\eps])$ is equal to the defect of $(\lambda,[\eps])$ (see \cite[\S13]{lusztig1984intersection}), hence $k(\lambda,[\eps]) = |M| = |J^1| - |J^{-1}|$ by \Cref{rem:symbol}.
\end{remark}

\section{Hecke algebras}\label{sec:hecke}

\subsection{Hecke algebra of a Bernstein component}\label{subsec:HABernstein}
We give a brief overview of \cite[\S1--\S3]{bushnell1998smooth}.
For a $\sfk$-rational character $\phi \colon \aG(\sfk) \to \sfk^\times$ and $s \in \C$, define a smooth one-dimensional complex representation
\begin{align}
\phi_{\chi,s} \colon \aG(\sfk) \to \C^\times \colon g \mapsto \Vert\phi(g)\Vert_{\sfk}^s,
\end{align} 
where $\Vert\cdot\rVert_{\sfk}^s$ denotes the usual normalised absolute value of $\sfk$. Let $X_{\C}(\aG(\sfk))$ be the group generated by the $\phi_\chi,s$, ranging over $\chi$ and $s$ as above. 
Elements of $X_{\C}(\aG(\sfk))$ are called \emph{unramified quasicharacters of $\aG(\sfk)$}. 
Let $L$ be a Levi subgroup of $\Gk$ and $\sigma$ a supercuspidal representation of $L$. Two such pairs $(L_1,\sigma_1)$ and $(L_2,\sigma_2)$ are called \emph{inertially equivalent} if there exists a $g \in \aG(\sfk)$ and $\chi \in X_{\C}(L_2)$ such that
\begin{align}
L_1 = gL_2g^{-1} \quad\text{and}\quad \sigma_1^g \cong \sigma_2 \otimes \chi,
\end{align}
where $\sigma_1^g \colon x \mapsto \sigma_1(gxg^{-1})$. We write $[L,\sigma]_{\Gk}$ for the \emph{inertial equivalence class} of a pair $(L,\sigma)$ and $\sB(\aG(\sfk))$ for the set of inertial equivalence classes of $\aG(\sfk)$.

\begin{definition}
For $\fs =[L,\sigma]_{\Gk} \in \sB(\aG(\sfk))$, let $\Rep^{\fs}(\aG(\sfk))$ be the full subcategory of $\Rep(\aG(\sfk))$ consisting of all $\pi \in \Rep(\aG(\sfk))$ such that each irreducible subquotient of $\pi$ is isomorphic to a subquotient of the normalised parabolic induction of $\sigma$ to $G$.
These subcategories are called \emph{Bernstein components of} $\aG(\sfk)$.
\end{definition}

\begin{proposition}[{\cite[Proposition 2.10]{bernstein1984centre}}]
We have a direct product decomposition 
\begin{align}	
\Rep(\aG(\sfk)) = \prod_{\fs \in \sB(\aG(\sfk))} \Rep^{\fs}(\aG(\sfk)).
\end{align}
\end{proposition}

Fix a Haar measure on $\aG(\sfk)$.
Let $K$ be a compact open subgroup of $\aG(\sfk)$ and $(\rho,V)$ a smooth representation of $K$. Let $\sX$ be the space of compactly supported and locally constant functions $f \colon G \to V$ such that $f(kg) = \rho(k^{-1})f(g)$ for $k \in K$ and $g \in G$. Then $\aG(\sfk)$ acts on $\sX$ by right translation, and we obtain a smooth representation which we denote by $\cindKG_K^{\Gk}(\rho)$ and which we call a \emph{compactly induced representation}. 

\begin{definition}[{\cite[(2.6)]{bushnell1998smooth}}]
The \emph{Hecke algebra of compactly supported $\rho$-spherical functions on $\aG(\sfk)$} is defined to be the $\C$-algebra $\cH(\aG(\sfk),\rho) := \End_{\Gk}(\cindKG_K^{\Gk}(\rho))$. This is a $\C$-algebra with respect to convolution associated to the fixed Haar measure.
\end{definition}

\begin{definition}\label{prop:type}
We say that $(K,\rho)$ is a type if there exists a finite set $\fS(\rho) \subseteq \sB(\aG(\sfk))$ such that we have an equivalence of categories $\cH(\aG(\sfk),\rho)\Mod \cong \prod_{\fs \in \fS(e)} \Rep^{\fs}(\aG(\sfk))$ (see \cite[(3.5) Proposition, (3.7) Lemma, (4.4) Definition]{bushnell1998smooth}).
\end{definition}

\subsection{Affine Hecke algebras}\label{subsec:AHA}
A \emph{root datum} is a quadruple $\Phi = (X^*,X_*,R,R^\vee)$, where $X^*$ and $X_*$ are free abelian groups of finite rank with a perfect pairing
$\perfp \colon X^* \times X_* \to \Z$, $R$ and $R^\vee$ are finite subsets of $X^*$ and $X_*$ respectively with a bijection denoted by $R \to R^\vee \colon \alpha \mapsto \alphac$, such that
\begin{enumerate}
\item $\perf{\alpha}{\alphac} = 2$ for all $\alpha \in R$,
\item For each $\alpha \in R$, the reflections $s_\alpha \colon X^* \to X^* \colon x \mapsto x - \perf{x}{\alphac}\alpha$ and $s_{\alpha} \colon X_* \to X_* \colon y \mapsto y - \perf{\alpha}{y}\alphac$ leave $R$ and $R^\vee$ stable, respectively.
\end{enumerate}
The elements of $R$ and $R^\vee$ are called \emph{roots} and \emph{coroots}, respectively.
A \emph{basis} of $R$ is a subset $\Pi$ such that each $\alpha \in R$ can uniquely be expressed as a linear combination of elements of $\Pi$ with coefficients that are all non-negative or all non-positive. 
The tuple $(X^*,X_*,R,R^\vee,\Pi)$ is called a \emph{based root datum}.
Let $R^+$ be the set of \emph{positive roots} with respect to $\Pi$, i.e. the roots in $R$ that are linear combinations of elements of $\Pi$ with strictly non-negative integer coefficients.
Let $S := \{s_\alpha \colon \alpha \in \Pi\}$ be the set of \emph{simple reflections}.
Let $W := \langle S \rangle$ be the finite group generated by $S$, which we call the \emph{Weyl group}.
Define the \emph{length} $\ell(w)$ of $w \in W$ to be the smallest integer $\ell$ such that $w$ can be written as a product of $\ell$ simple reflections.
A \emph{parameter set} of $\Phi$ is a pair $(\lambda,\lambda^*)$ of $W$-equivariant functions $\lambda \colon R \to \C$, $\lambda^* \colon \{\alpha \in R \colon \alphac \in 2X_*\} \to \C$. We also write $\lambda(s_\alpha) = \lambda(\alpha)$ and $\lambda^*(s_\alpha) = \lambda^*(\alpha)$ for $\alpha \in R$.

Let $G = G(\Phi)$ be the connected complex reductive group whose root datum is $\Phi$. 
Given a maximal torus $T$ of $G$, we have canonical indentifications 
\begin{align}
X^* = \Hom(T,\C^\times), \chi \mapsto x, \quad X_* = \Hom(\C^\times,T), \phi \mapsto y
\end{align}
such that 
$(\chi \circ \phi)(z) = z^{\perf{x}{y}}$ for all $z \in \C^\times$. 
We also have isomorphisms $X_* \otimes_{\Z} \C^\times \cong T, (y\otimes z) \mapsto y(z)$ and $W \cong N_G(T) / T$.
The choice of $\Pi$ uniquely determines a Borel subgroup $B$ of $G$ containing $T$ such that the root subgroups of the positive roots are precisely the ones that lie in $B$.
Let $t$ be an indeterminate.

\begin{definition}[Bernstein presentation, {\cite[Prop. 3.6, 3.7]{lusztig1989affine}}]\label{def:AHA}
The \emph{affine Hecke algebra} $\cH = \cH^{\lambda,\lambda^*}(\Phi)$ associated to a based root datum $\Phi = (X^*,X_*,R,R^\vee,\Pi)$ with parameter set $(\lambda,\lambda^*)$ is defined to be the associative unital $\C[t,t^{-1}]$-algebra with $\C[t,t^{-1}]$-basis $\{T_w \colon w \in W\}\sqcup\{\theta_x \colon x \in X^*\}$ and relations
\begin{align}
&(T_{s} + 1)(T_{s} - t^{2\lambda(s)}) = 0 &&\text{for all $s \in S$,}
\\
&T_wT_{w'} = T_{ww'} &&\text{for all $w,w' \in W$ with $\ell(ww') = \ell(w)+\ell(w')$,}
\\
&\theta_x\theta_{x'} = \theta_{x+x'} &&\text{for all $x,x' \in X^*$,}
\\
&\theta_x T_s - T_s\theta_{s(x)} = (\theta_x - \theta_{s(x)})(\cG(s) - 1) &&\text{for $x \in X^*, s \in \Pi$}, 
\end{align}
where for $s = s_\alpha$ with $\alpha \in \Pi$, we have
\begin{equation}
\cG(s)=
\cG(\alpha)=
\begin{cases} 
\theta_\alpha \frac{z^{2\lambda(\alpha)}-1}{\theta_\alpha-1}
&\text{if } \alphac\notin 2X_*,
\\
\theta_\alpha \frac{z^{\lambda(\alpha)+\lambda^*(\alpha)}-1)(t^{\lambda(\alpha)-\lambda^*(\alpha)}+1)}{\theta_{2\alpha}-1}
&\text{if }\alphac\in 2X_*.
\end{cases}
\end{equation}
Up to isomorphism, $\cH$ does not depend on the choice of $\Pi$. 
\end{definition}
Let $\cA$ be the $\C[t,t^{-1}]$-subalgebra of $\cH$ generated by $\{\theta_x \colon x \in X^*\}$.

\begin{proposition}[{\cite[Prop. 3.11]{lusztig1989affine}}]
\label{prop:centreaha}
We have $\cZ = Z(\cH) = \cA^W$.
\end{proposition}

\subsection{Graded Hecke algebras}\label{sec:GHA}
Fix a finite $W$-invariant set $\Sigma$ in $T$.
Let $\mathscr C$ be the commutative associative unital $\C$-algebra with vector space basis $\{E_\sigma \colon \sigma \in \Sigma\}$ and relations
\begin{equation}
\sum_{\sigma \in \Sigma} E_\sigma = 1,
\quad
E_\sigma E_{\sigma'} = \delta_{\sigma,\sigma'}E_\sigma.
\end{equation}
Let $\sS$ be the symmetric algebra of $\ft^\vee = X^* \otimes_{\Z} \C = \Lie(T^\vee)$ where $T^\vee = X^* \otimes_{\Z} \C^\times$. 
Let $\bA$ be the commutative associative $\C$-algebra $\C[r] \otimes \mathscr C \otimes_{\C} \mathscr S$ with componentwise multiplication.
Consider the diagonal action of $W$ on $\Sigma \times R$ and let $\mu \colon \Sigma \times R \to \C$ be a $W$-invariant function. 

\begin{definition}[{\cite[Proposition 4.4]{lusztig1989affine}, \cite[Proposition 4.2]{barbasch1993reduction}}]\label{def:gha}
The \emph{graded Hecke algebra} $\bH_\Sigma = \bH_{\Sigma}^{\mu}(\Phi)$ is the associative $\C[r]$-algebra with $\C$-vector space structure $\C[t_w \colon w \in W]\otimes_{\C}\bA$, with unit $t_1$ and relations 
\begin{align}
&t_wt_{w'} = t_{ww'} &&\text{ for all $w,w'\in W$,}
\\
&E_\sigma t_\alpha = t_\alpha E_{s_{\alpha}(\sigma)} &&\text{ for all $\sigma \in \Sigma$, $\alpha \in \Pi$,}
\\
&\omega\cdot t_{s_\alpha} - t_{s_\alpha}\cdot s_\alpha(\omega) = r \perf{\omega}{\alphac} \sum_{\sigma \in \Sigma} E_\sigma \mu(\sigma,\alpha) &&\text{ for all $\omega \in \mathscr{S}$, $\alpha \in \Pi$.}
\end{align}
Given $(\lambda,\lambda^*)$ as in \Cref{subsec:AHA}, let $\mu_{\lambda,\lambda^*} \colon \Sigma \times R \to \C$ be defined by
\begin{align}\label{eq:mulambda}
\mu_{\lambda,\lambda^*}(\sigma,\alpha) 
=
\begin{cases}
0 &\text{ if $s_\alpha\sigma \neq \sigma$,}
\\
2\lambda(\alpha) &\text{ if $s_\alpha\sigma = \sigma$, $\alphac \notin 2X_*$,}
\\
\lambda(\alpha) + \lambda^*(\alpha)\theta_{-\alpha}(\sigma) &\text{ if $s_\alpha\sigma = \sigma$, $\alphac \in 2X_*$}.
\end{cases}
\end{align}
We write $\bH_\Sigma^{(\lambda,\lambda^*)}(\Phi)$ for $\bH_\Sigma^{\mu_{\lambda,\lambda^*}}(\Phi)$.
\end{definition}

Fix $\bH_\Sigma = \bH_{\Sigma}^{\mu}(\Phi)$ with $\mu$ arbitrary.
For $\sigma \in \Sigma$, let $W(\sigma) = \stab(\sigma) = \{w \in W \colon w(\sigma) = \sigma\}$.

\begin{proposition}[{\cite[Proposition 3.2]{barbasch1993reduction}}]\label{prop:ZbH}
It holds that
\begin{enumerate}[ref={\thecorollary(\arabic*)}]
\crefalias{enumi}{proposition}
\item\label{prop:ZbH1} For two disjoint finite $W$-invariant sets $\Sigma_1,\Sigma_2\subseteq T$, we have $\bH_{\Sigma_1\sqcup\Sigma_2} \cong \bH_{\Sigma_1} \oplus \bH_{\Sigma_2}$,
\item\label{prop:ZbH2} $Z(\bH_\Sigma) = \bA^W$,
\item\label{prop:ZbH3} Suppose $\Sigma = W\cdot\sigma$ is a single $W$-orbit. Then 
\begin{align}
\bA^W &\to E_\sigma \cdot ( \C[r] \otimes_{\C} \mathscr S^{W(\sigma)}) \cong \C[r] \otimes_{\C} \sS^{W(\sigma)}\colon 
\\
a &\mapsto E_\sigma \cdot a
\end{align}
is an isomorphism of $\C[r]$-algebras,
\item\label{prop:ZbH4} 
Write $\Sigma = \bigsqcup_{i=1}^m \Sigma_i$ as a finite union of $W$-orbits $\Sigma_i = W\cdot\sigma^i$. Then
\begin{align}
Z(\bH_\Sigma) 
\cong \bigoplus_{i=1}^m Z(\bH_{\Sigma_i})
\cong \bigoplus_{i=1}^m \C[r]\otimes_{\C}\sS^{W(\sigma^i)}.
\end{align}
\end{enumerate}
\end{proposition}

\subsection{Reduction theorems}\label{subsec:red}
Suppose $\Sigma = W \cdot \sigma$ for some $\sigma \in T$ and consider
\begin{align}
R_\sigma &= \{\alpha \in R \colon \theta_\alpha(\sigma) =
\begin{cases}
1 &\text{if $\alphac \notin 2X_*$,}
\\
\pm 1 &\text{if $\alphac \in 2X_*$.}
\end{cases}
\},
\\
R_{\sigma}^+ &= R_\sigma \cap R^+,
\\
\Pi_\sigma &= \text{ set of simple roots w.r.t. $R_\sigma^+$,}
\\
W_\sigma &= \langle s_\alpha \colon \alpha \in \Pi_\sigma \rangle,
\\
\Gamma_\sigma &= \{w \in W(\sigma) \colon w(R_\sigma^+) = R_\sigma^+\}.
\end{align}
Note that $(\lambda,\lambda^*)$ restricts to a parameter set on the based root datum $\Phi_\sigma = (X^*,X_*,R_\sigma,R_\sigma^\vee,\Pi_\sigma)$.
Let $\cH_\sigma = \cH^{(\lambda,\lambda^*)}(\Phi_\sigma)$. Note that $\stab_W(\sigma) = W(\sigma) = \Gamma_\sigma \ltimes W_\sigma$ and so $W_\sigma \cdot \sigma$ is a singleton. Let $\bH_\sigma = \bH_{\{\sigma\}}^{(\lambda,\lambda^*)}(\Phi_\sigma)$. The group $\Gamma_\sigma$ acts as algebra automorphisms of $\bH_\sigma$ via its action on $W_\sigma$ and $\bA$.
Let $\bH_{\sigma}'=\Gamma_\sigma \ltimes \bH_\sigma$ and note that $Z(\bH_{\sigma}') = \C[r] \otimes_{\C} \sS^{W(\sigma)} \cong Z(\bH_\Sigma)$. 

Write $\Sigma = \{\sigma = \sigma_1,\sigma_2,\dots,\sigma_m\}$. For $i,j\in\{1,\dots,m\}$, let $w_i \in W$ such that $w_i\sigma_1 = \sigma_i$ and let
\begin{equation}
E_{i,j} = t_{w_i^{-1}}E_\sigma t_{w_j}.
\end{equation}

\begin{theorem}[First reduction theorem {\cite[Theorem 8.6]{lusztig1989affine}, \cite[Theorem 3.3]{barbasch1993reduction}}]
\label{theorem:reduction1}
${}$
\begin{enumerate}
\item For each $i \in \{1,\dots,m\}$, $E_{\sigma_i} \cdot \bH_\Sigma \cdot E_{\sigma_i}$ is canonically isomorphic to $\bH_\sigma'$.
\item Let $\cM_m$ be the $\C$-algebra generated by $\{E_{i,j}\}_{i,j}$. Then $\cM_m \cong M_m(\C)$ in the obvious way and we have an isomorphism 
 of $\C[r]$-algebras $\bH_\Sigma \to M_m(\bH_\sigma') = \cM_m \otimes_{\C} (\bH_\sigma')$ such that for all $i,j \in \{1,\dots,m\}$ and $w \in w_iW(\sigma)w_j$, we have 
\begin{align}
E_{\sigma_i} t_w E_{\sigma_j} &\mapsto E_{i,j} \otimes t_{w_i^{-1}ww_j} 
\\
\omega = \sum_{a=1}^m E_{\sigma_a}\omega E_{\sigma_a} &\mapsto I_m \otimes \omega,
\end{align}
where and $I_m$ is the $m\times m$ identity matrix.
Note that $\bA^W$ is mapped isomorphically to $I_m\otimes\C[r]\otimes_{\C}\sS^{W(\sigma)}$.
\end{enumerate}
\end{theorem}

Now fix $\bH_\Sigma = \bH_\Sigma^{(\lambda,\lambda^*)}(\Phi)$.
For now, drop the assumption that $\Sigma$ is a single $W$-orbit.
Consider the algebras $\hat\C[r]$, $\hat{\mathscr S}$ of convergent power series in $\C[r]$ and $\mathscr S$, respectively. Define
\begin{equation}
\hat \bA = \hat\C[r] \otimes_{\C} \mathscr{C} \otimes_{\C} \hat{\mathscr{S}},
\quad
\hat \bH_\Sigma = \C[W] \otimes_{\C} \hat \bA.
\end{equation}
We have $Z(\hat\bH_\Sigma) = \hat \bA^W \cong \hat\C[r] \otimes_{\C} \hat\sS^{W(\sigma)}$, similar to \Cref{prop:ZbH}.

We have a polar decomposition $T = T_c \times T_r$, where
\begin{align}
T_c = X_* \otimes_{\Z} S^1 \text{(the compact part),}
\quad
T_r = X_* \otimes_{\Z} \R_{>0} \text{(the real part).}
\end{align}
Recall that $\ft = X_* \otimes_{\Z}\C$. 
We can write $\ft = \ft_{\R} \oplus \ft_{\iR}$, where $\ft_{\R} = \{x \in \ft \colon \bar x = x\}$ and $\ft_{\iR} = \{x \in \ft \colon \bar x = -x\}$. 
We have a morphism of algebraic groups
\begin{align}
\log \colon T_r \to \ft_{\R} \colon y \otimes u \mapsto y \otimes \log u.
\end{align}
with inverse given by
\begin{align}
\exp \colon \ft_{\R} \to T_r \colon y \otimes v \mapsto y \otimes \exp(v).
\end{align}
Let $\hat\sK$ be the field of fractions of $\hat\C[r]\otimes_{\C}\hat\sS$ and let $\hat\bH_\Sigma(\hat\sK) = \C[t_w\colon w \in W]\otimes_{\C}\sC\otimes_{\C}\hat\sK \supseteq \hat\bH_\Sigma$.

\begin{proposition}[{\cite[Proposition 4.1]{barbasch1993reduction}}]\label{prop:Psi}
Recall that $\Sigma$ is not necessarily a single $W$-orbit and assume that $\sigma \in T_c$.
We have a $\C$-algebra homomorphism
\begin{align}
\Psi \colon\cH &\to \hat\bH_{\Sigma}(\hat\sK)
\\
\theta_x &\mapsto \sum_{\sigma \in \Sigma} \theta_x(\sigma) \cdot E_\sigma \cdot e^x &&\text{for $x \in X^* \subseteq \mathscr S$,}
\\
t &\mapsto e^r,
\\
T_{s_\alpha}+1 &\mapsto \sum_{\sigma \in \Sigma} E_\sigma (t_{s_\alpha} + 1) \frac{\Psi(\cG(\alpha))}{1+\frac{r\mu(\sigma,\alpha)}{\alpha}} &&\text{for $\alpha \in \Pi$.}
\end{align}
\end{proposition}

Assume that $\Sigma = W\cdot\sigma$ is a single $W$-orbit again and that $\sigma \in T_c$.
Let $s_r \in T_r$ and $s = \sigma s_r \in T$.
Since $\cH$ is a finitely generated algebra over its centre $\cZ = Z(\cH) = \cA^W$, it follows by Dixmier's version of Schur's lemma that $\cH$ has central characters.
As $\cA$ is isomorphic to the coordinate ring of $T \times \C^\times = (X^* \otimes_{\Z} \C^\times)\times\C^\times$, the set of central characters of $\cH$ is in bijection with $W\backslash T \times \C^\times$.
Similarly, 
$\bH_\Sigma$, $\bH_\sigma'$ and $\hat\bH_\Sigma$ have central characters, and since $\sS\otimes_{\C}\C[r]$ is isomorphic to the coordinate ring of $\ft \times \C$, the set of central characters for each of the three algebras is in bijection with $(W(\sigma)\backslash\ft) \times \C$. 
Let $\chi$, $\bar\chi$, $\bar\chi'$, $\hat{\bar\chi}$ be central characters of $\cH$, $\bH$, $\bH_\sigma'$, $\hat{\bH}$, respectively, such that $\chi$ corresponds to $(W\cdot s,v_0)$ and $\bar\chi$, $\bar\chi'$ and $\hat{\bar\chi}$ correspond $(W(\sigma)\cdot \log s_r,r_0)$
where $e^{r_0} = v_0$.
Let $I_\chi = \ker \chi$, $I_{\bar\chi} = \ker \bar \chi$, $I_{\bar\chi'} = \ker\bar\chi$, and $I_{\hat{\bar\chi}} = \ker \hat{\bar\chi}$ and define 
\begin{align}
&\bH_{\bar\chi} 
= 
\bH_\Sigma / (I_{\bar\chi}\bH_\Sigma) 
\cong 
\hat\bH_\Sigma / (I_{\hat{\bar\chi}}\hat\bH_\Sigma) 
= 
\hat \bH_{\hat{\bar\chi}} \text{ (isomorphism of $\C[r]\otimes_{\C}\sS^{W(\sigma)}$-algebras),} 
\\
&\bH_{\sigma,\bar\chi'}' 
= 
\bH'_\sigma / I_{\bar\chi'}\bH'_\sigma,
\quad \cH_\chi 
= 
\cH / I_\chi\cH.
\label{eq:Heckequotientcentral}
\end{align}
\Cref{theorem:reduction1} gives an isomorphism of $\C[r]\otimes_{\C}\sS^{W(\sigma)}$-algebras
\begin{equation}\label{eq:reductioncentralcharacter}
\bH_{\bar\chi} \cong M_m(\bH_{\sigma,\bar\chi}').
\end{equation}

\begin{theorem}[Second reduction theorem {\cite[Theorem 9.3]{lusztig1989affine}, \cite[Theorem 4.3]{barbasch1993reduction}}]
\label{theorem:reduction2}
Recall that $\Sigma = W \cdot \sigma$ is a single orbit.
Let $\chi$ and $\bar\chi$ be as above with the extra assumption that $r_0 \in \R$. 
The map $\Psi$ in \Cref{prop:Psi} induces an isomorphism of $\C$-algebras
\begin{align}
\Psi_\chi \colon \cH_\chi &\to \hat \bH_{\bar \chi} (\cong \bH_{\chi}).
\end{align}
\end{theorem}

\begin{corollary}[{cf. \cite[Corollary 10.8]{lusztig1989affine}}]
\label{cor:reduction}
\begin{enumerate}[ref={\thecorollary(\arabic*)}]
\crefalias{enumi}{corollary}
\item\label{cor:reduction1} We have an isomorphism of $\C$-algebras
\begin{equation}
\Psi_\chi' \colon \cH_\chi \to M_m(\bH_{\sigma,\bar\chi'}')
\end{equation}
obtained by composing the isomorphisms \eqref{eq:reductioncentralcharacter}, $\Psi_\chi$, and the isomorphism $\Xi \colon \bH_{\bar \chi} \cong \bH_{\chi}$.
\item\label{cor:reduction2} 
For $x \in X^*$, we have $\Psi_\chi'(\theta_x) = \diag(\theta_x(\sigma_1)\Xi(e^x),\dots,\theta_x(\sigma_m)\Xi(e^x))$.
\end{enumerate}
\end{corollary}

We want to prove a variation (\Cref{cor:reductiongeneral}) of \Cref{theorem:reduction2} that will be used in the proof of \Cref{theorem:comdiagHecke}.
Let $m \in \N$ and $s^1,\dots,s^m \in T$ such that for each $i,j \in \{1,\dots,m\}$, $s_i$ and $s_j$ are not conjugate by $W$.
For each $i \in \{1,\dots,m\}$, let $\sigma^i$ be the compact part of $s^i$ and $s_r^i$ the real part of $s^i$. Consider the $W$-invariant sets $\cS = \bigsqcup_{i=1}^m \cS_i$, where $\cS_i = W\cdot s^i$, and $\Sigma = \bigsqcup_{i=1}^m \Sigma_i$ where $\Sigma_i = W\cdot\sigma^i$. 
Let $\chi_i$ be the central character of $\cH$ corresponding to $(\cS_i,v_0)$.
Recall that $\bH_{\Sigma} \cong \bigoplus_{i=1}^m \bH_{\Sigma_i}$ and $\hat\bH_{\Sigma} \cong \bigoplus_{i=1}^m \hat\bH_{\Sigma_i}$. Let $\bar\chi_i$ (resp. $\hat{\bar{\chi}}_i$) be the central character of $\bH_{\Sigma_i}$ (resp. $\hat\bH_{\Sigma_i}$) corresponding to $(\bar\cS_i,r_0)$ where $\bar\cS_i = W(\sigma^i)\cdot\log(s_r^i)\subseteq\ft$. 
Consider the ideals
\begin{align}\label{eq:HAideals}
I_{\cS_i,v_0} &:= \{f \in Z(\cH) \colon f(s',v_0) = 0 \text{ for all $s' \in \cS_i$}\} =  \ker(\chi_i),
\\
I_{\cS,v_0} &:= \{f \in Z(\cH) \colon f(s',v_0) = 0 \text{ for all $s' \in \cS$}\} = \prod_{i=1}^m I_{\cS_i,v_0},
\\
\bbI_{\bar\cS_i,r_0} &:= \{g \in Z(\bH_{\Sigma_i}) \cong \sS^{W(\sigma)}\otimes_{\C}\C[r]  \colon g(\bar s',r_0) = 0 \text{ for all $\bar s' \in \bar\cS_i$}\} = \ker(\bar\chi_i),
\\
\I_{\bar\cS_i,r_0} &:= \{g \in Z(\hat\bH_{\Sigma_i}) \cong \hat\sS^{W(\sigma)}\otimes_{\C}\hat\C[r]  \colon g(\bar s',r_0) = 0 \text{ for all $\bar s' \in \bar\cS_i$}\} = \ker(\hat{\bar{\chi}}_i).
\end{align}
Denote by $\bbI_{\bar\cS_i,r_0}\bH_{\Sigma}$ (resp. $\I_{\bar\cS_i,r_0}\hat\bH_{\Sigma}$) the ideal in $\bH_\Sigma$ (resp. $\hat\bH_\Sigma$) generated by $\bbI_{\bar\cS_i,r_0}$ (resp. $\hat\bbI_{\bar\cS_i,r_0}$) considered as a subset of $Z(\bH_\Sigma)$ (resp. $Z(\hat\bH_\Sigma)$) via the decomposition $Z(\bH_\Sigma) = \bigoplus_{i=1}^mZ(\bH_{\Sigma_i})$  (resp.  $Z(\hat\bH_\Sigma) = \bigoplus_{i=1}^mZ(\hat\bH_{\Sigma_i})$) from \Cref{prop:ZbH4}

\begin{proposition}\label{cor:reductiongeneral}
The map
\begin{align}
\Psi_{\cS,v_0} \colon
\cH / I_{\cS,v_0} \cH
\to 
\hat\bH_{\Sigma} \Biggm/
\prod_{i=1}^m \I_{\bar\cS_i,r_0} \hat\bH_{\Sigma}
\cong
\bH_{\Sigma} \Biggm/
\prod_{i=1}^m \bbI_{\bar\cS_i,r_0} \bH_{\Sigma}
\end{align}
induced by $\Psi \colon \cH \to \hat\bH_\Sigma(\hat\sK)$ in \Cref{prop:Psi} is an isomorphism of $\C$-algebras.
\end{proposition}
\begin{proof}

The Chinese remainder theorem gives 
\begin{align}
&\cH / I_{\cS,v_0}\cH \cong \prod_{i=1}^m \cH / I_{\cS_i,v_0}\cH \label{eq:CRT1}
\quad\quad
\text{(note that $I_{W\cdot s,v_0} = \prod_{i=1}^mI_{W_J \cdot s^i,v_0}$),}
\label{eq:CRT1}
\\
&\bH_{\Sigma} \Biggm/\prod_{i=1}^m \bbI_{\bar\cS_i,r_0} \bH_{\Sigma} 
\cong
\prod_{i=1}^m (\bH_{\Sigma}/\bbI_{\bar\cS_i,r_0}\bH_{\Sigma})
=
\prod_{i=1}^m (\bH_{\Sigma_i}/\bbI_{\bar\cS_i,r_0}\bH_{\Sigma_i}),
\label{eq:CRT2}
\\
&\hat\bH_{\Sigma} \Biggm/\prod_{i=1}^m \I_{\bar\cS_i,r_0} \hat\bH_{\Sigma} 
\cong
\prod_{i=1}^m (\hat\bH_{\Sigma}/\I_{\bar\cS_i,r_0}\hat\bH_{\Sigma})
=
\prod_{i=1}^m (\hat\bH_{\Sigma_i}/\I_{\bar\cS_i,r_0}\hat\bH_{\Sigma_i}).
\label{eq:CRT3}
\end{align}
For $i=1,\dots,m$, \Cref{theorem:reduction2} gives us an isomorphism
\begin{equation}\label{eq:reductionmapi}
\Psi_{i} \colon \cH/I_{\cS_i,v_0}\cH \to 
\hat\bH_{\Sigma_i} / \I_{\bar\cS_i,r_0}\hat\bH_{\Sigma_i}
\cong
\bH_{\Sigma_i} / \bbI_{\bar\cS_i,r_0} \bH_{\Sigma_i}
\end{equation}
induced from the isomorphism $\cH \to \hat\bH_{\Sigma_i}(\hat\sK)$ from \Cref{prop:Psi}.
Composing $\sum_{i=1}^m\Psi_i$ with the isomorphisms \eqref{eq:CRT1}, \eqref{eq:CRT2}, and \eqref{eq:CRT3} gives an isomorphism
\begin{equation}\label{eq:reductionmapgeneral}
\Psi'
\colon 
\cH / I_{\cS,v_0} \cH
\to 
\hat\bH_{\Sigma} \Biggm/
\prod_{i=1}^m \I_{\bar\cS_i,r_0} \hat\bH_{\Sigma}
\cong
\bH_{\Sigma} \Biggm/
\prod_{i=1}^m \bbI_{\bar\cS_i,r_0} \bH_{\Sigma}.
\end{equation}
By definition of $\Psi_i$ for each $i \in \{1,\dots,m\}$, it follows that for each Bernstein generator $T_s\theta_x$ (with $s \in S$, $x \in X^*$), we have
\begin{align}
\Psi'(T_s\theta_x \cdot  I_{\cS,v_0} \cH) = \Psi(T_s\theta_x) \cdot \prod_{i=1}^m \bbI_{\bar\cS_i,r_0} \bH_{\Sigma}.
\end{align}
Taking $\Psi_{(\cS,v_0)} := \Psi'$ finishes the proof.
\end{proof}

\subsection{Graded Hecke algebras attached to cuspidal local systems}\label{subsec:GHAcuspls}
Let $G$ be a connected complex reductive algebraic group and fix a Borel subgroup $B$ and a maximal torus $T$ with Lie algebra $\ft = \Lie(T)$. Let $S_G$ be the set of triples $(L,C_L,\cL)$ where $L$ is a Levi subgroup of a standard parabolic subgroup $P$ (containing $B$) of $G$, $C_L$ is a nilpotent orbit in $\Lie(L)$ and $\cL$ is a cuspidal $L$-equivariant local system on $C_L$. To each $(L,C_L,\cL) \in S_G$, Lusztig attached a graded Hecke algebra, denoted by $\bH(G,L,C_L,\cL)$, and classified all their representations in \cite{lusztig1988cuspidal} and \cite{lusztig1995cuspidal}. 

We will define $\bH(G,L,C_L,\cL)$ as in the two papers of Lusztig mentioned above, but we will present it in notation that is more similar to the notation in \cite[\S2,\S3]{ciubotaru2008unitary}. Let
\begin{align}
\fg &= \Lie(G),
\\
\fz_L &= \Lie(Z(L)^\circ),
\\
W_{L} &= N_G(L)/L,
\\
R(G,Z(L)^\circ) &= \{\alpha \in X^*(Z(L)^\circ) \colon \alpha \text{ appears in the adjoint action of $Z(L)^\circ$ on $\fg$}\}. \label{eq:rootGT}
\end{align}
Let $R_L = R(G,Z(L)^\circ)_{\mathrm{red}}$ be the reduced root system of $R(G,Z(L)^\circ$.
Let $R_L^\vee$ be a subset of $X_*(Z(L)^\circ)$ that forms a set of coroots of $R$, i.e. $R_L^\vee = \{\check\alpha\colon\alpha\in R_L\}$ and $s_\alpha(\beta) = \beta - \langle\beta,\check\alpha\rangle\alpha$. 

By \cite[Theorem 9.2]{lusztig1984intersection}, $W_{L}$ is the Weyl group with root system $R_{L}$.
Let $\Pi_{L}$ be the set of roots in $R_{L}$ that are simple with respect to the parabolic subgroup associated to $L$ that contains $B$.
Define a function $\mu_{L} \colon \Pi_{L} \to \Z_{\geq2}$ as follows. 
Let $u \in C_L$.
For each $\alpha \in \Pi_{L}$, let $\mu_{L}(\alpha)$ be the unique integer in $\Z_{\geq 2}$ such that
\begin{align}
\Ad(u)^{\mu_{L}(\alpha)-2} &\colon \fg_{\alpha} \oplus \fg_{2\alpha} \to \fg_{\alpha} \oplus \fg_{2\alpha} \text{ is non-zero,}
\\
\Ad(u)^{\mu_{L}(\alpha)-1} &\colon \fg_{\alpha} \oplus \fg_{2\alpha} \to \fg_{\alpha} \oplus \fg_{2\alpha} \text{ is zero.}
\end{align}
Then $\mu_{L}$ is $W$-invariant by \cite[Prop. 2.12]{lusztig1988cuspidal} and thus extends to a $W$-invariant a function $\mu_{L} \colon R_{L} \to \Z_{\geq2}$.

\begin{definition}\label{def:GHAcusp}
Let $(L,C_L,\cL) \in S_G$.
The \emph{graded Hecke algebra $\bH(G,L,C_L,\cL)$ attached to $(L,C_L,\cL)$} is defined to be the graded Hecke algebra 
$\bH_{\{1\}}^{\mu_L}(X^*(Z(L)^\circ),X_*(Z(L)^\circ),R_L,R_L^\vee,\Pi_L)$.
\end{definition}

Fix $(L,C_L,\cL) \in S_G$, let $\bH := \bH(G,L,C_L,\cL)$, and let $P = LU$ be a standard parabolic subgroup of $G$ with Levi subgroup $L$ and unipotent radical $U$. Let $\fn = \Lie(U)$.
Let 
\begin{align}
\dot{\fg}_{U} = \{(x,gP) \in \fg \times G/P \colon \Ad(g^{-1})x \in \fn\}.
\end{align}
The $L$-equivariant local system $\cL$ gives a $G\times\C^\times$-local system $\dot{\cL}$ on $\dot{\fg}_U$ via the map
\begin{align}
\dot{\fg}_U \to C_L \colon (x,gP) \mapsto \pr_{C_L}(\Ad(g^{-1})x).
\end{align}
Let $r_0 \in \C$ and $e \in \fg$ be a nilpotent element. Consider the variety
\begin{align}\label{eq:genflagvariety}
\cP_e := \cP_e^G := \{gP \in g/P \colon \Ad(g^{-1})e \in C_L + \fn\}.
\end{align}
The group $G \times \C^\times$ acts on $\fg$ by $(g,\lambda)\cdot x = \lambda^{-2}\Ad(g')x$ for all $x\in\fg$, $g' \in G$, $\lambda \in \C^\times$. 
In particular, we have $Z_{G\times\C^\times}(e) = \{(g,\lambda)\in G\times\C^\times \colon \Ad(g')x = \lambda^2x\}$ and 
\begin{align}
\fz_{G\times\C^\times}(e) := \Lie(Z_{G\times\C^\times}(e)) = \{(x,r_0) \in \fg \oplus \C \colon [x,e] = 2r_0e\}. 
\end{align}
The centraliser $Z_{G\times\C^\times}(e)$ acts on $\cP_e$ by $(h',\lambda)hP = (h'h)P$.
Lusztig constructed a $\bH\times A_{G\times\C^\times}(e)$-action on the $Z_{G\times\C^\times}^\circ(e)$-equivariant homology $H_{\bullet}^{Z_{G\times\C^\times}^\circ(e)}(\cP_e,\dot{\cL})$, see \cite[\S8.5, Theorem 8.13]{lusztig1988cuspidal}.
By \cite[\S8.7]{lusztig1988cuspidal}, the $Z_{G\times\C^\times}^\circ(e)$-equivariant cohomology $H^{\bullet}_{Z_{G\times\C^\times}^\circ(e)}(\cP_e,\dot{\cL})$ is the coordinate ring of the affine variety $\cV_e$ of semisimple $Z_{G\times\C^\times}^\circ(e)$-orbits on $\fz_{G\times\C^\times}(e)$.
For $(s,r_0) \in \cV_e$, let $\C_{s,r_0}$ be the $H^{\bullet}_{Z_{G\times\C^\times}^\circ(e)}$-module given by the evaluation map $H^{\bullet}_{Z_{G\times\C^\times}^\circ(e)} \to \C$ at $(s,r_0)$, and define the $\bH$-module
\begin{align}
Y_{\cL}(e,s,r_0) = \C_{s,r_0} \otimes_{H^\bullet_{Z_{G\times\C^\times}^\circ(e)}} H_{\bullet}^{Z_{G\times\C^\times}^\circ(e)}(\cP_e,\dot{\cL}).
\end{align}
Note that $A_G(e,s)$ is isomorphic to the stabiliser of $(s,r_0)$ in $A_{G\times\C^\times}(e)$. 
Let $A_{G}(e,s)_{L,C,\cL}^\wedge$ be the subset of $A_{G}(e,s)^\wedge$ consisting of $\psi \in A_{G}(e,s)^\wedge$ that appear in $Y_{\cL}(e,s,r_0)$.
For $\psi \in A_{G}(e,s)_{L,C,\cL}^\wedge$, let $Y_{\cL}(e,s,r_0,\psi)$ be the $\psi$-isotypic component.
By \cite[Proposition 8.16(a)]{lusztig1995cuspidal}, we have $ A_{G}(e,s)^\wedge = \bigsqcup_{(L,C,\cL)\in S_G} A_{G}(e,s)_{L,C,\cL}^\wedge$, so there exists precisely one triple $(L,C,\cL) \in S_G$ such that $Y_{\cL}(e,s,r_0)^\psi \neq 0$ and so we can write $Y(e,s,r_0,\psi) = Y_{\cL}(e,s,r_0,\psi)$.
We have the following isomorphism of $W_L\times A_G(e)$-modules by \cite[Proposition 7.2,\S8.9]{lusztig1988cuspidal} and \cite[10.12(d)]{lusztig1995cuspidal}
\begin{equation}\label{eq:genspringerfibre}
Y_{\cL}(e,0,0) \cong H_{\bullet}(\cP_e,\dot{\cL}) \text{\quad($\{1\}$-equivariant homology)}. \label{eq:genspringerfibre}
\end{equation}

Let $\cM_{s,r_0}$ be the set $Z_G(s)$-conjugacy classes of pairs $(e,\psi)$ where $e \in \fg$ is a nilpotent element satisfying $[s,e] = 2r_0e$ and $\psi \in A_G(e,s)_0^\wedge$ and define
\begin{equation}
\cM := \cM(G) := \text{$G$-conjugacy classes of } \{(e,s,r_0,\psi) \colon (s,r_0) \in \cV_e, (e,\psi) \in \cM_{s,r_0}\}. \label{eq:cM}
\end{equation}
Recall that central characters of $\bH$ are in bijection with $W\backslash\ft\otimes\C^\times$ (this is a consequence of \Cref{prop:ZbH2}).
Each semisimple element $s \in \fg$ is $G$-conjugate to an element $\tau \in \ft$.
It is well-known that $G\cdot s \cap \ft = W\cdot \tau$, and so $s$ uniquely corresponds to a $W$-orbit in $\ft$.
Let $\chi$ be the central character corresponding to $(W\cdot \tau,r_0)$ and define
\begin{equation}\label{eq:GHAcentralcharacter}
\Irr_{s,r_0} \bH := \Irr_{W\cdot \tau,r_0} \bH := \Irr \bH_\chi.
\end{equation}
\begin{theorem}[{\cite[Theorem 8.15]{lusztig1988cuspidal}, \cite[Corollary 8.18]{lusztig1995cuspidal}}]
${}$\label{theorem:standard}
\begin{enumerate}[ref={\thecorollary(\arabic*)}]
\crefalias{enumi}{theorem}
\item\label{theorem:standard2} Suppose $M$ is a simple $\bH$-module on which $r$ acts as multiplication by $r_0 \in \C$. Then there exists a unique $(e,s,r_0,\psi) \in \cM$ such that $M$ is a quotient $\bar Y(e,s,r_0,\psi)$ of $Y(e,s,r_0,\psi)$.
\item\label{theorem:standard3} We have a bijection
\begin{align}
\cM_{s,r_0} &\leftrightarrow \Irr_{s,r_0}\bH
\\
(e,\psi) &\mapsto \bar Y(e,s,r_0,\psi).
\end{align}
\end{enumerate}
\end{theorem}

Let $M$ be an irreducible $\bH$-module. We have a weight space decomposition
\begin{equation}
M = \bigoplus_{\lambda \in \fz_L} M_\lambda.
\end{equation}
We say that $\lambda \in \fz_L$ is a weight if $V_\lambda \neq 0$.

For a semisimple element $s \in \fg$, we can uniquely define semisimple elements $s_h,s_e \in \fg$ such that $s = s_h + s_e$ and such that for any $g \in G$ for which $\Ad(g)s \in \ft$, we have $\Ad(g)s_r \in \ft_{\R}$ and $\Ad(g)s_c \in \ft_{\iR}$. 
We call $s_r$ and $s_c$ the \emph{real part} and \emph{compact part} of $s$, respectively.
\begin{definition}\label{def:tempered}
An irreducible module $M$ is called \emph{tempered}\footnote{This is Casselman's notion of temperedness and agrees with the notion of $\tau$-temperedness in \cite{kazhdan1987proof} and \cite{lusztig2002cuspidal} where we take the group homomorphism $\tau \colon \C^* \to \R\colon z \mapsto \log |z|$.} if $\omega(\lambda) \neq 0$ for all fundamental weights $\omega \in \fz_L^\vee$ and all weights $\lambda$ of $M$. We write
\begin{align}\label{eq:cMtemp}
&\cMtemp = \cMtemp(G) = \{(e,s,r_0,\psi) \in \cM(G) \colon \bar Y(e,s,r_0,\psi) \text{ is tempered}\},
\\
&\cMtempreal = \cMtempreal(G) = \{(e,s,r_0,\psi) \in \cMtemp(G) \colon r_0\in\R,\,s_e = 0\}.
\end{align}
\end{definition}

\begin{proposition}[{\cite[Theorem 1.21]{lusztig2002cuspidal}}]\label{prop:temperedirreducible}
Let $(e,s,r_0,\psi) \in \cM$. The following are equivalent:
\begin{enumerate}
\item $(e,s,r_0,\psi) \in \cMtemp$,
\item There exists an $\sl_2$-triple $(e,h,f)$ in $\fg$ such that $[s,h] = 0$, $[s,f] = -2r_0f$ and the semisimple element $s - r_0h$ is compact (i.e. its real part is zero).
\end{enumerate}
If any of these conditions are satisfied, then $Y(e,s,r_0,\psi) = \bar Y(e,s,r_0,\psi)$ is irreducible.
\end{proposition}

\begin{proposition}\label{prop:tempered}
Let $(e,s,r_0,\psi) \in \cMtempreal$. Then the inclusion $A_G(e,s) \cong A_G(e)$ coming from $Z_G(e,s) \hookrightarrow Z_G(e)$ is an isomorphism.
\end{proposition}
\begin{proof}
Let $(e,h,f)$ be an $\sl_2$-triple such that $h \in \ft_{\R}$ and $s_0 = s - r_0 h$. By assumption $s$ and $h$ are real, and so $s_0$ real, and by temperedness, $s_0$ is compact, so $s_0 = 0$. By \cite[(4.3.1) Lemma]{reeder2002isogenies} the inclusion $Z_G(e,s) \hookrightarrow Z_G(e,s_0) = Z_G(e)$ induces an isomorphism $A_G(e,s) \cong A_G(e,s_0) = A_G(e)$.
\end{proof}

\begin{proposition}\label{prop:tempreal}
Fix $r_0 \in \R$. The following map is a bijection
\begin{align}
\{(e,s,r',\psi)\in\cMtempreal\colon r' = r_0\} \to \cN_G \colon (e,s,r_0,\psi) \to (e,\psi).
\end{align}
\end{proposition}
\begin{proof}
Consider $(e,s,r_0,\psi) \in \cMtempreal$. Then $s_0$ is compact and $s$ is real. 
By \Cref{prop:tempered}, we have $A_G(e,s) = A_G(e)$, so this map is well-defined.
Clearly this map is surjective.
Consider an $\sl_2$-triple $(e,h,f)$ as in \Cref{prop:temperedirreducible} such that additionally $h$ is real. Then $s_0 = s - r_0h \in \ft_{\R} \cap \ft_{i\R}$ is $0$.
For any $(e,s',r_0,\psi) \in \cMtempreal$, we can similarly pick an $\sl_2$-triple $(e,h',f')$. Note that
\begin{align}
[s_0,e] &= -2r_0e, & [s_0,e] &= -2r_0e,
\\
[s_0,h] &= 0, & [s_0,h'] &= 0,
\\
[s_0,f] &= 2r_0f, & [s_0,f'] &= 2r_0f,
\end{align}
so by \cite[Lemma 2.1]{gross2010arithmetic}, the $\sl_2$-triples $(e,h,f)$ and $(e',h',f')$ are conjugate by an element $g \in G$ that centralises $s_0 = 0$ and fixes $e$. Thus $(e,s,r_0,\psi)$ and $(e,s',r_0,\psi)$ are conjugate by $g$, so they represent the same class in $\cMtempreal$. Hence our map is also injective. 
\end{proof}

\section{Unipotent representations}\label{sec:unipotent}

\begin{definition}[{\cite{lusztig1995classification}}]\label{def:unipotentreps}
An irreducible representation $X$ of $\aG(\sf k)$ is called \emph{unipotent} if there exists a parahoric subgroup $P$ of $\Gk$ such that its unipotent radical $X^{U_P}$ contains a subrepresentation isomorphic to a cuspidal and unipotent
$(P/U_P)$-representation $(\rho,E)$. Write $\Unip(\Gk)$ for the set of unipotent representations of $\Gk$.
\end{definition}
By \cite{morris1993tamely} (and \cite{moy1994unrefined} when $\aG$ is not split over $\sfk$), $(P,\rho)$ in \Cref{def:unipotentreps} is a type. So an irreducible representation of $\Gk$ is unipotent if and only if it lies in $\prod_{\fs\in\fS(\rho)}\cR^{\fs}(\Gk)$ with $\fS(\rho)$ as in \Cref{prop:type}. For $\fs \in \fS(\rho)$, we call $\cR^{\fs}(\Gk)$ a \emph{unipotent Bernstein component}.

Let $\Phi(G^\vee)$ be the set of triples $\dG$-conjugacy classes of $(e,s,\phi)$ where $s \in G^\vee$ is a semisimple element, $e \in \fg^\vee$ is a nilpotent element such that $\ad(s)e = qe$, and $\phi \in A_{\dG}(s,e)^\wedge$.

\begin{theorem}[Deligne--Langlands--Lusztig correspondence, {\cite[Corollary 6.5]{lusztig1995classification}}]
\label{thm:DLL}
Suppose $\aG$ is adjoint and split over $\sf k$. Then there is a bijection 
\begin{align}
\Phi(G^\vee) \to \bigsqcup \Unip(\Gk_\zeta) 
\colon 
(e,s,\psi) \mapsto X(e,s,\psi),
\end{align}
where the disjoint union runs over all inner forms $\Gk_\zeta$ of the split form $\Gk$, and such that 
\begin{enumerate}[wide]
\item For all $(e,s,\psi) \in \Phi(G^\vee)$, we have $X(e,s,\phi) \in \Unip(\Gk_\zeta)$,
\item $X(e,s,\psi)$ is tempered if and only if $s_0 \in T_c^\vee$ and 
$
\overline{G^\vee(s)}e = 
\{x \in \fg^\vee \colon \Ad(s)x = qx\}
$.
\end{enumerate}
\end{theorem}

Fix an Iwahori subgroup $I$ and let $\sA$ be the finite set of $(P,\rho)$ where $P$ is a parahoric subgroup of $\Gk$ containing $I$ and $\rho$ is a cuspidal unipotent representation of $\bar P$. Since each $(P,\rho) \in \sA$ is a type, we have an equivalence of categories $\prod_{\fs\in\fS(\rho)}\cR^{\fs}(\Gk) \cong \prod_{(P,\rho)\in\sA} \cH(\Gk,\rho)\Mod$. 
A major result in \cite{lusztig1995classification} is that $\prod_{(P,\rho)\in\sA} \cH(\Gk,\rho)\Mod \cong \prod_{(L,C,\cL)} \cH(\dG,L,C,\cL)$, where the product runs through the triples $(L,C,\cL)$ where $L$ is a standard pseudo-Levi subgroup of $\dG$, $C$ is a unipotent class in $L$, and $\cL$ is an irreducible cuspidal $L$-equivariant local system on $C$ and $\cH(\dG,L,C,\cL)$ is a `geometric' affine Hecke algebra attached to the triple $(L,C,\cL)$.
The associated graded Hecke algebras of $\cH(\dG,L,C,\cL)$ are non-extended (by \cite[Lemma 4.4]{lusztig1995classification} since $\dG$ is simply connected) and of the form $\bH = \bH(Z_{\dG}(s_c),L,C,\cL)$ as in \Cref{subsec:GHAcuspls} where $s_c \in \dG$ is a compact semisimple element such that $L$ is a Levi-subgroup of $Z_{\dG}(s_c)$. 
The irreducible representations of these graded Hecke algebras are classified by \Cref{theorem:standard}, which can be used to parametrise the unipotent representations of the inner to split forms of $\aG$:
a unipotent representation $X$ is determined by $(e,s_r,\psi,s_c)$, i.e. by $(e,s,\psi) \in \Phi(\dG)$ with $s=s_cs_r$ (or by pairs $(x,\psi)$ where $x \in \dG$ has Jordan decomposition given by $e$ and $s$).

\section{Aubert--Zelevinsky duality and Iwahori--Matsumoto duality}\label{sec:AZIMgIM}

\subsection{Aubert--Zelevinsky involution}\label{subsec:AZ}
Fix a Borel subgroup $B$ of $\Gk$ and let $\cQ$ be the set of parabolic subgroups of $\Gk$ containing $B$. 
For $Q \in \cQ$, let $\Ind_{Q}^{\Gk}$ and $\Res_{Q}^{\Gk}$ denote the normalised induction and Jacquet functor, respectively (see for instance \cite[\S{III}.1]{bernstein1996lectures} for the definitions). Let 
$r_Q$ be the semisimple rank of $Q$. 
\begin{definition}[{\cite[\S1, Theorem 1.7]{aubert1995dualite}}] 
The \emph{Aubert--Zelevinsky involution} on $\cR(\Gk)$ is the involutive group homomorphism $\AZ = \AZ_{\aG(\sfk)} \colon \cR(\Gk) \to \cR(\Gk)$, defined as
\begin{align}
\AZ_{\aG(\sfk)} = \sum_{Q\in\cQ} (-1)^{r_{Q}} \Ind_{Q}^{\Gk} \circ \Res_{Q}^{\Gk}.
\end{align}
\end{definition}

\subsection{Iwahori--Matsumoto involution}\label{subsec:IM}

Let $\Phi = (X_*,X^*,R,R^\vee,\Pi)$ be a based root datum with parameter set $(\lambda,\lambda^*)$ and let $\cH = \cH^{(\lambda,\lambda^*)}(\Phi)$.

\begin{definition}
The \emph{Iwahori--Matsumoto involution} of $\cH$ is an involutive $\C[t,t^{-1}]$-algebra homomorphism $\IM \colon \cH \to \cH$ defined by
\begin{align}
\IM(T_s) &= (-t)^{\lambda(s)}T_{s}^{-1} &&\text{for $s \in S = \{s_\alpha \in W \colon \alpha \in \Pi\}$,}
\\
\IM(\theta_x) &= T_{w_0} \theta_{w_0(x)} T_{w_0}^{-1} &&\text{for $x \in X$.}
\end{align}
Note that in the literature, the Iwahori--Matsumoto involution is sometimes defined slightly differently, mapping $\theta_x$ to $\theta_{-x}$ instead. We denote this involution by $\tIM$.
\end{definition}

We obtain an involution on $\cR(\cH)$ by twisting by $\IM$.
Note that $\IM$ restricts to the identity on $Z(\cH) = \cA^{W}$, and so $\IM$ preserves central characters. 

Let $J \subseteq S$, let $W_J$ be the corresponding parabolic subgroup of $W$, and let $\cH_J$ be the subgalgebra of $\cH$ generated by $\{T_w \colon w \in W_J\}$ and $\cA$, which the Hecke algebra of a root datum with roots determined by $J$.
Let $\Res_J \colon \cH\Mod \to \cH_J\Mod$ denote the restriction functor and $\Ind_J \colon \cH_J\Mod \to \cH\Mod$ denote the induction functor $\cH\otimes_{\cH_J}-$. This defines abelian group morphisms $\Res_J \colon \cR(\cH) \to \cR(\cH_J)$ and $\Ind_J \colon \cR(\cH_J) \to \cR(\cH)$.

\begin{theorem}[{\cite[Theorem 2]{kato1993duality}}]\label{theorem:KatoIM}
We have the following equality of homomorphisms of $\cR(\cH)$
\begin{align}
\IM = \sum_{J\subseteq S} (-1)^{|J|}\Ind_J\circ\Res_J.
\end{align}
\end{theorem}

\subsection{Graded Iwahori--Matsumoto involution}\label{subsec:gIM}

Recall that $T = X_* \otimes_{\Z} \C^\times$. Let $\Sigma = W \cdot \sigma$ be a single $W$-orbit in $T$ and $\bH = \bH_\Sigma^{(\lambda,\lambda*)}(\Phi)$.

\begin{definition}\label{def:gIM}
Let $\gIM_\Sigma \colon \bH \to \bH$ be the involutive $\C[r]$-algebra homomorphism defined by
\begin{align}
\gIM_\Sigma(t_w) &= (-1)^{\ell(w)}t_w &\text{for $w \in W$,}
\\
\gIM_\Sigma(E_\sigma) &= E_\sigma &\text{for $\sigma \in \Sigma$,}
\\
\gIM_\Sigma(\omega) &= t_{w_0}w_0(\omega)t_{w_0}^{-1} &\text{for $\omega \in \sS$.}
\end{align}
\end{definition}

We obtain an involution on $\cR(\bH)$, also denoted by $\gIM_\Sigma$, such that $\gIM_\Sigma(M)$ has the same underlying vector space as $M$, but with $\bH$-action twisted by $\gIM_\Sigma$.
From \Cref{prop:ZbH}, we see that $\gIM_\Sigma$ restricts to the identity map on $Z(\bH)$, so $\gIM_\Sigma$ preserves central characters.

Let $\sigma \in T$ and recall the extended graded Hecke algebra $\bH_\sigma' = \bH_\sigma \rtimes \Gamma_c$. Let $\{w_1 = 1, w_2,\dots, w_m\}$ be a set of coset representatives of $W/W(\sigma)$ so that $\Sigma = \{w_1\sigma,w_2\sigma,\dots,w_m\sigma\}$.
Let $k \in \{1,\dots,m\}$ be such that $w_0 \in w_k W(\sigma)$.

\begin{definition}\label{def:gIM2}
The \emph{graded Iwahori--Matsumoto involution} of $\bH_\sigma'$ is an involutive $\C[r]$-algebra
homomorphism $\gIM \colon \bH_\sigma' \to \bH_\sigma'$ defined by
\begin{align}
\gIM(t_w) &= (-1)^{\ell(w)}t_w &\text{for $w \in W_\sigma$,}
\\
\gIM(\gamma) &= \gamma &\text{for $\gamma \in \Gamma_\sigma$,}
\\
\gIM(\omega) &= t_{w_k}^{-1}t_{w_0}w_0(\omega)t_{w_0}^{-1}t_{w_k} &\text{for $\omega \in \sS$.}
\end{align}
In the literature, the graded Iwahori--Matsumoto involution is sometimes defined differently, mapping $\omega$ to $-\omega$ instead. We denote this involution by $\tgIM$.
\end{definition}

\begin{remark}
Note that the restriction of $\gIM_\sigma$ to $\bH_\sigma$ is not the same as $\gIM$ defined on $\bH_{\{\sigma\}}^{(\lambda,\lambda_*)}(\Phi)$ from \Cref{def:gIM}, as \Cref{def:gIM2} involves the longest element of $W$ rather than that of $W_\sigma$. 
We defined $\gIM$ in \Cref{def:gIM2} precisely such that \Cref{corollary:diagramIM} holds. This idea comes from joint work with Jonas Antor and Emile Okada for an upcoming joint paper. 

For the geometric graded Hecke algebras, $\gIM$ is the involution of interest. We stress that whenever we use the phrase `graded Iwahori--Matsumoto involution', we refer to $\gIM$ instead of $\gIM_\Sigma$, which we also did not give a name in \Cref{def:gIM}.
\end{remark}

\begin{remark}\label{rem:gIMsign}
Let $M$ be a representation of $\bH_\sigma'$. Denote the restriction of $M$ to $\C[t_w\colon w\in W_\sigma] \cong \C[W_\sigma]$ by $M|_{W_\sigma}$. 
Let $\chi$ and $\chi'$ be the character of $M|_{W_\sigma}$ and $\gIM(M)|_{W_\sigma}$, respectively. For all $w \in W_\sigma$, we have $\chi'(t_w) = (-1)^{\ell(w)}\chi(t_w)$, so $\gIM(M)|_{W_\sigma} = M|_{W_\sigma}\otimes\sgn$.
\end{remark}

Given $J \subseteq S$, let $\bH_J$ be the subalgebra of $\bH$ generated by $\{t_w \colon w \in W_J\}$ and $\bA$.
Let $\Ind_J \colon \bH_J\Mod \to \bH\Mod$ be the induction functor $\bH \otimes_{\bH_J} -$, and let $\Res_J \colon \bH\Mod \to \bH_J\Mod$ denote the restriction functor.

\begin{theorem}\label{theorem:gIMindres}
We have the following equality of homomorphisms of $\cR(\bH)$
\begin{equation}
\gIM_\Sigma = \sum_{J\subseteq S} (-1)^{|J|}\Ind_J\circ\Res_J.
\end{equation}
\end{theorem}

Based on \cite{kato1993duality}, \Cref{theorem:gIMindres} was proved in \cite[Theorem 6.5]{chan2016duality} for the case that $\Sigma$ is a singleton and for a certain specialisation of $r$.
Our proof will be almost identical.

Write $\Pi = \{\alpha_1,\dots,\alpha_m\}$. 
For $J \subseteq \Pi$ and $M$ an $\bH$-module, let
\begin{equation}
C_J(M) = \Ind_J \Res_J X = \bH \otimes_{\bH_J} (\Res_J M).
\end{equation}
For $i \in \Z_{\geq0}$, let 
\begin{equation}
C_i(M) = \bigoplus_{J \subseteq \Pi, |J| = i} C_J(M).
\end{equation}
Given $J \subseteq J' \subseteq \Pi$ with $|J'| = |J| + 1$, write $\Pi \setminus J = \{\alpha_{i_1},\dots,\alpha_{i_{m-|J|}}\}$ with $1 \leq i_1 < \dots < i_{m-{J}} \leq m$. Let $j$ be the unique integer such that $\alpha_{i_j} \in J' \setminus J$. Let $\eps_J^{J'} = (-1)^{j+1}$. Define
\begin{align}
\pi_J^{J'} \colon C_J(M) \to C_{J'}(M) \colon h \otimes x \mapsto \eps_J^{J'} h \otimes x,
\end{align}
and for $i \in \Z_{\geq0}$, define
\begin{align}
\pi_i = \bigoplus_{\substack{J \subseteq \Pi \\ |J| =i}} \bigoplus_{\substack{J \subseteq J' \subseteq \Pi \\ |J'| = i + 1}} \pi_J^{J'} \colon C_i(M) \to C_{i+1}(M).
\end{align}

The proof of the following proposition is identical to the proof of \cite[Proposition 6.2]{chan2016duality}.

\begin{proposition}[{\cite[Proposition 6.2]{chan2016duality}, cf. \cite[Theorem 1]{kato1993duality}}]\label{prop:bHses}
We have an exact sequence of $\bH$-modules
\begin{figure}[h]
\centering
\begin{tikzcd}[column sep = small]
0 \arrow[r] & \ker \pi_0 \arrow[r] & C_0(M) \arrow[r, "\pi_0"] & C_1(M) \arrow[r,"\pi_1"] & \ldots \arrow[r,"\pi_{m-2}"] & C_{m-1}(M) \arrow[r,"\pi_{m-1}"] & X \arrow[r] & 0.
\end{tikzcd}
\end{figure}
\\
In particular, as elements of $\cR(\bH)$, we have 
\begin{align}
[\ker \pi_0] = \sum_{i = 0}^m (-1)^{i} [C_i(M)] = \sum_{J \subseteq \Pi} (-1)^{|J|} [\Ind_J\Res_JM].
\end{align}
\end{proposition}

Let $M$ be an $\bH$-module and consider the injective map
\begin{align}
\chi \colon M &\to \bH \otimes_{\bA} (\Res_{\bA} M),
\\
x &\mapsto \sum_{w\in W} (-1)^{\ell(w)} t_w \otimes t_w^{-1} \cdot x.
\end{align}

\begin{lemma}[{\cite[Lemma 6.3]{chan2016duality}}]\label{lemma:imchi}
We have an isomorphism of $\bH$-modules
\begin{align}
\img \chi \cong \gIM_\Sigma(M).
\end{align}
\end{lemma}
\begin{proof}
The proof of \cite[Lemma 6.3]{chan2016duality} already gives us
\begin{align}
t_s \cdot \chi(x) &= \chi(\gIM_\Sigma(t_s)\cdot x) &&\text{for $s \in S$,}
\\
\omega \cdot \chi(x) &= \chi(\gIM_\Sigma(\omega)\cdot x) &&\text{for $\omega \in \sS$.}
\end{align}
Furthermore, we have
\begin{align}
r \cdot \chi(x) &= \chi(r\cdot x) = \chi(\gIM_\Sigma(r)\cdot x),
\\
E_\sigma \cdot \chi(x) &= \chi(E_\sigma \cdot x) = \chi( \gIM_\Sigma(E_\sigma) \cdot x) &&\text{for $\sigma \in \Sigma$,}
\end{align}
hence $\img \chi \cong \gIM_\Sigma(M)$.
\end{proof}

\begin{lemma}[{\cite[Lemma 6.4]{chan2016duality}, cf. \cite[Theorem 1, Lemma 1]{kato1993duality}}]\label{lemma:kerpi0}
We have $\ker\pi_0 = \img\chi$.
\end{lemma}
\begin{proof}
We have $\ker \pi_0 = \bigcap_{s \in S} L_s$ where $L_s = \{ht_s \otimes x - h \otimes t_s x \colon h \in \bH, x \in M\} \subseteq C_0(M) = \bH \otimes_{\bA} M$. 
Note that $M^W = \bA^W \otimes_{\bA} M$ and consider the $\bA$-linear map
\begin{align}
\phi \colon M^W \to \bH \otimes_{\bA} M \colon (m_w)_{w\in W} \mapsto \sum_{w\in W} t_w \otimes t_w^{-1} m_w.
\end{align}
For $s \in S$, let 
\begin{align}
\bA_s^W = \{(x_w)_{w \in W} \in \bA^W \colon x_{ws} = -x_w \text{ for all $w \in W$}\}.
\end{align}
Note that 
\begin{align}\label{eq:intersectionimchi}
\bigcap_{s\in S} \phi(\bA_s^W \otimes_{\bA} M) = \img \chi.
\end{align}
We will show that $L_s = \phi(\bA_s^W \otimes_{\bA} M)$.
Clearly, we have $L_s \supseteq \phi(\bA_s^W \otimes_{\bA} M)$.
As an $\bA$-algebra, $L_s$ is generated by elements of the form 
$t_y \otimes x - t_{ys} \otimes t_s x$ with $y \in W$, $x \in M$.
For such an element and for $y \in W$ and $x \in M$, let $x_{ys} := -t_yx$ and $x_y := t_yx$.
For $w \in W \setminus \{y,ys\}$, let $x_w = 0$. Then 
\begin{align}
t_y \otimes x - t_{ys} \otimes t_s x
=
t_y \otimes t_y^{-1}x_y + t_{ys} \otimes t_{ys}^{-1} x_{ys} = \phi( (x_w)_{w\in W} ) \in \phi(\bA_s^W \otimes_{\bA} M),
\end{align}
thus $L_s \subseteq \phi(\bA_s^W \otimes_{\bA} M)$, hence equality.
By \eqref{eq:intersectionimchi} and the fact that $L_s = \phi(\bA_s^W \otimes_{\bA} M)$, we have
\begin{align}
\ker \pi_0 = \bigcap_{s\in S} L_s = \bigcap_{s\in S} \phi(\bA_s^W \otimes_{\bA} M) = \img \chi.
&\qedhere
\end{align}
\end{proof}

\begin{proof}[Proof of \Cref{theorem:gIMindres}]
By \Cref{prop:bHses}, \Cref{lemma:imchi}, and \Cref{lemma:kerpi0}, we have
\begin{align}
[\gIM_\Sigma(M)] = [\img \chi] = [\ker \pi_0] = \sum_{i = 0}^m (-1)^{i+1} [\Ind_J\Res_JM].
&\qedhere
\end{align}
\end{proof}

\subsection{A correspondence between $\AZ$, $\IM$, and $\gIM$}
\label{subsec:AZIMgIM}
Let $s \in T$ and let $\sigma$ be its compact part and $s_r$ its real part.
Let $\chi$, $\bar\chi$, $\bar\chi'$,  
be central characters of $\cH$, $\bH$, $\bH_\sigma'$, 
respectively, such that $\chi$ corresponds to $(W\cdot s,v_0)$, and $\bar\chi$ and $\bar\chi'$ 
correspond to $(W(\sigma)\cdot \log(s_r),r_0) \in (W(\sigma)\backslash\ft)\times\C$ where $e^{r_0} = v_0$.

\begin{theorem}\label{theorem:comdiagHecke}
Consider the isomorphism $\Psi_\chi$ from \Cref{theorem:reduction2}.
We have a commutative diagram
\begin{equation}
\begin{tikzcd}
\Irr(\cH_\chi) \arrow[r,"\IM"] \arrow[d, "\Psi_\chi"] & \Irr(\cH_\chi) \arrow[d, "\Psi_\chi"] 	
\\
\Irr(\bH_{\bar\chi}) \arrow[r, "\gIM_\Sigma"]  & \Irr(\bH_{\bar\chi}).
\end{tikzcd}
\end{equation}
\end{theorem}
\begin{proof}
This proof is a result of joint work with Jonas Antor and Emile Okada for a paper in preparation.
We can write $W \cdot s$ as a disjoint union of $W_J$-orbits $W\cdot s = \bigsqcup_{i=1}^m W_J \cdot s^i$.
For $i=1,\dots,m$, let $\sigma^i$ be the compact part of $s^i$ and $s_r^i$ the real part of $s^i$,
let $\chi_i$ be the central character of $\cH_J$ corresponding to $(W_J \cdot s^i,v_0)$, and let $\bar\chi_i$ be the central character of $\bH_J$ corresponding to $(W_J(\sigma^i)\cdot\log(s_r^i),r_0)$. 

Recall the ideals defined in \eqref{eq:HAideals} and let $\bbI_{W(\sigma)\cdot \log(s_r),r_0} = \bbI_{\bar\chi}$.
Let $M$ be an irreducible representation of $\cH$ with central character $\chi$. We can view $M$ as a representation of $\cH / I_{W\cdot s, v_0} \cH = \cH_\chi$. The representation $\Res_J(M)$ is annihilated by $I_{W\cdot s,v_0}$, so it is also a representation of $\cH_J / I_{W\cdot s, v_0} \cH_J$. By \Cref{theorem:reduction2}, $\Psi_{\chi}(M)$ is an irreducible representation of $\bH$ with central character $\bar\chi$, hence a representation of $\bH / \bbI_{W(\sigma)\cdot \log(s_r),r_0}\bH = \bH_{\bar\chi}$, and $\Res_J(\Psi_{\chi}(M))$ is a representation of 
$\bH_{J} / \mathbb{I}_{W(\sigma)\cdot\log(s_r),r_0}\bH_{J}$.
Note that $\bH_J = \bH_{W\cdot\sigma}^{(\lambda,\lambda^*)}(\Phi_J) = \bigoplus_{i=1}^m \bH_{W_J\cdot\sigma^i}^{(\lambda,\lambda^*)}(\Phi_J)$ where $\Phi_J$ where $\Phi_J$ is the `sub root datum' of $\Phi$ determined by $J$. Write $\bH_{J,i} = \bH_{W_J\cdot\sigma^i}^{(\lambda,\lambda^*)}(\Phi_J)$ and let $\bbI_{W_J(\sigma^i)\log(s_r^i),r_0}\bH_J$ be the ideal in $\bH_J$ generated by $\bbI_{W_J(\sigma^i)\cdot\log(s_r^i),r_0} \subseteq \bH_{J,i} \subseteq \bH_J$. Consider the isomorphism 
\begin{equation}
\Psi_{J,W\cdot s,v_0} \colon \cH_{J} / I_{W\cdot s, v_0} \cH_J 
\to 
\bH_J / \prod_{i=1}^m\bbI_{W_J(\sigma^i)\cdot\log(s_r^i),r_0}\bH_J
=
\bH_{J} / \bbI_{W(\sigma)\cdot\log(s_r),r_0}\bH_{J}
\end{equation}
obtained from \Cref{cor:reductiongeneral} applied to $\cH_J$ and $\bH_J$, and consider
the diagram
\begin{equation}\label{eq:resdiagram}
\begin{tikzcd}
\cR(\cH / I_{W\cdot s,v_0}\cH) 
\arrow[r,"\Res_J"] \arrow[d, "\Psi_\chi"] 
& 
\cR(\cH_J / I_{W\cdot s, v_0} \cH_J) 
\arrow[d, "\Psi_{J,W\cdot s,v_0}"]
\\
\cR(\bH / \mathbb{I}_{W(\sigma)\cdot\log(s_r),r_0}\bH) 
\arrow[r, "\Res_J"]  
& 
\cR(\bH_{J} / \mathbb{I}_{W(\sigma)\cdot\log(s_r),r_0}\bH_{J}).
\end{tikzcd}
\end{equation}
This diagram commutes since the action of the generators (in the sense of \Cref{def:gha}) of $\bH_{J}$ on $\Psi_{J,W\cdot s,v_0} \circ \Res_J (M)$ is the same as on $\Res_J \circ \Psi_\chi(M)$.
By \eqref{eq:CRT1}, we have a group isomorphism
\begin{equation}\label{eq:HArepsJ}
\cR(\cH_J / I_{W\cdot s, v_0} \cH_J) \tilde\to \bigoplus_{i=1}^m \cR(\cH_J/I_{W_J \cdot s^i,v_0}\cH_J).
\end{equation}
Recall from \Cref{prop:ZbH4} that $Z(\bH_J) = \bigoplus_{i=1}^m Z(\bH_{J,i})$ and note that $\bbI_{W(\sigma)\cdot\log(s_r),r_0}Z(\bH_J) = \prod_{i=1}^m(\bbI_{W_J(\sigma^i)\cdot\log(s_r^i),r_0}Z(\bH_J))$ is a product of pairwise coprime maximal ideals in $Z(\bH_J)$, so by the Chinese remainder theorem (cf. \eqref{eq:CRT2}), we have a group isomorphism
\begin{equation}\label{eq:GHArepsJ}
\cR(\bH_J / \bbI_{W(\sigma)\cdot\log(s_r),r_0}\bH_J) 
\tilde\to 
\bigoplus_{i=1}^m\cR(\bH_{J}/\bbI_{W_J(\sigma^i)\cdot\log(s_r^i),r_0}\bH_{J}).
\end{equation}
For $i=1,\dots,m$, let 
$\Psi_i \colon \cH_J/I_{W_J\cdot s^i,v_0}\cH_J 
\to 
\bH_{J,i} / \bbI_{W_J(\sigma^i)\cdot\log(s_r^i),r_0}\bH_{J,i} 
= 
\bH_{J} / \bbI_{W_J(\sigma^i)\cdot\log(s_r^i),r_0}\bH_{J}$ 
be the isomorphism coming from \Cref{theorem:reduction2} (note that \Cref{theorem:reduction2} applies since $\bH_{J,i} = \bH_{W_J\cdot\sigma^i}^{(\lambda,\lambda^*)}(\Phi_J)$ and $W_J\cdot\sigma^i$ is a single $W_J$-orbit). 
By \cite[Theorem 6.2]{barbasch1993reduction}, we have a commutative diagram
\begin{equation}
\begin{tikzcd}
\cR(\cH_J / I_{W_J\cdot s^i, v_0} \cH_J) 
\arrow[r,"\Ind_J"] \arrow[d, "\Psi_i"] 
& 
\cR(\cH / I_{W\cdot s,v_0}\cH) 
=
\cR(\cH_\chi) 
\arrow[d, "\Psi_\chi"] 	
\\
\cR(\bH_{J} / \mathbb{I}_{W_J(\sigma^i)\cdot\log(s_r^i),r_0}\bH_{J})
\arrow[r, "\Ind_J"]  
& 
\cR(\bH / \mathbb{I}_{W(\sigma)\cdot\log(s_r),r_0}\bH)
=
\cR(\bH_{\bar \chi}).
\end{tikzcd}
\end{equation}
\\
We noted in the proof of \Cref{cor:reductiongeneral} that $\Psi_{J,W\cdot s,v_0} = \sum_{i=1}^m \Psi_i$ (note that we can identify $\bH_{J,i} / \bbI_{W_J(\sigma^i)\cdot\log(s_r^i),r_0}\bH_{J,i} 
= 
\bH_{J} / \bbI_{W_J(\sigma^i)\cdot\log(s_r^i),r_0}\bH_{J}$), so by \eqref{eq:HArepsJ} and \eqref{eq:GHArepsJ} we obtain a commutative diagram

\begin{equation}\label{eq:inddiagram}
\begin{tikzcd}
\cR(\cH_J / I_{W\cdot s, v_0} \cH_J) 
\arrow[r,"\Ind_J"] \arrow[d, "\Psi_{J,W\cdot s,v_0}"] 
& 
\cR(\cH_\chi) 
\arrow[d, "\Psi_\chi"] 	
\\
\cR(\bH_{J} / \mathbb{I}_{W(\sigma)\cdot\log(s_r),r_0}\bH_{J})
\arrow[r, "\Ind_J"]  
& 
\cR(\bH_{\bar \chi}).
\end{tikzcd}
\end{equation}
Combining the diagrams \eqref{eq:resdiagram} and \eqref{eq:inddiagram}, the result then	 follows from \Cref{theorem:KatoIM}, \Cref{theorem:gIMindres}.
\end{proof}

The Morita equivalence coming from \Cref{subsec:red} gives us bijections
\begin{align}
&F \colon \Irr(\bH) \to \Irr(\bH_{\sigma}')\colon  M \mapsto \bH_\sigma'^{1\times m}\otimes_{\bH} M,
\\
&F^{-1} \colon \Irr(\bH_{\sigma}') \to \Irr(\bH)\colon M \mapsto M^m.
\end{align}
Note that $\bH_\sigma'^{1\times m}$ is canonically a right $M_m(\bH_\sigma')\cong \bH$-module and a left $\bH_\sigma'^{1\times m}$ module with diagonal action.
Consider the composition $\Omega = F\circ\Psi_{\chi} \colon \Irr(\cH_\chi) \to \Irr(\bH_{\sigma,\bar\chi'}')$.
The proof of \Cref{corollary:diagramIM} is based on ideas that the author discussed with Jonas Antor and Emile Okada in our joint work for an upcoming paper.

\begin{corollary}\label{corollary:diagramIM}
We have a commutative diagram
\begin{equation}
\centering
\begin{tikzcd}
\Irr(\cH_\chi) \arrow[r,"\IM"] \arrow[d, "\Omega"] & \Irr(\cH_\chi) \arrow[d, "\Omega"]
\\
\Irr(\bH_{\sigma,\bar\chi'}') \arrow[r, "\gIM"]  & \Irr(\bH_{\sigma,\bar\chi'}').
\end{tikzcd}
\end{equation}
\end{corollary}

\begin{proof}
By \Cref{theorem:comdiagHecke}, it suffices to show that the following diagram commutes
\begin{equation}\label{eq:diagGHAGHA}
\centering
\begin{tikzcd}
\Irr(\bH) \arrow[r,"\gIM_\Sigma"] \arrow[d,"F"] & \Irr(\bH) \arrow[d,"F"]
\\
\Irr(\bH_{\sigma}') \arrow[r, "\gIM"]  & \Irr(\bH_{\sigma}').
\end{tikzcd}
\end{equation}
Let $(M,\pi) \in \Irr(\bH_{\sigma}')$. We obtain $(\bH_{\sigma}'^{1\times m}\otimes_{\bH}M^m,F\IM_\Sigma F^{-1}(\pi))$, $(M,\gIM(\pi)) \in \Irr(\bH_\sigma')$. We have to show that these two representations are isomorphic.
Write $\pi' = F\IM_\Sigma F^{-1}(\pi)$ and $\tilde \pi = \gIM(\pi)$ and define
\begin{align}
f \colon M \to \bH_{\sigma}'^{1\times m}\otimes_{\bH}M^m \colon a \mapsto u \otimes F^{-1}(\pi)(t_{w_k}) \underline{a}
\end{align}
where $u$ is the element of $\bH_{\sigma}'^{1\times m}$ with entries all equal to $1$ and $\underline{a}$ the element of $M^m$ with entries all equal to $a$. For $\omega \in \sS = S(\ft^\vee)$ and $a \in M$, we have
\begin{align}
f(\tilde\pi(\omega)a) 
&= u \otimes F^{-1}(\pi)(t_k) \underline{(\pi(t_{w_k}^{-1}t_{w_0}w_0(\omega)t_{w_0}^{-1}t_{w_k})a}
\\
&= u \otimes F^{-1}(\pi)(t_{w_0}w_0(\omega)t_{w_0}^{-1}t_{w_k}) \underline{a} = 
\\
&= (t_{w_0}w_0(\omega)t_{w_0}^{-1} \cdot u) \otimes F^{-1}(\pi)(t_{w_k})\underline{a} 
\\
&= \pi'(\omega)f(a).
\end{align}
For $\gamma \in \Gamma_c$ and $w \in W_\sigma$ we have $f(\tilde \pi(\gamma)a) = \pi'(\gamma)f(a)$ and $f(\tilde\pi(t_w)a) = \pi'(t_w)f(a)$ since the parity of the length of $\gamma$ and $w$ in $W_\sigma$ is the same as in $W$. Furthermore, $f$ is clearly surjective, hence bijective since $\dim_{\C}M = \dim_{\C}(\bH_{\sigma}'^{1\times m}\otimes_{\bH}M^m)$, so $(\bH_{\sigma}'^{1\times m}\otimes_{\bH}M^m,\pi') \cong (M,\tilde\pi)$ (irreducible representations of $\bH$ are well-known to be finite-dimensional), so the diagram \eqref{eq:diagGHAGHA} commutes.
\end{proof}

\begin{remark}
An analogous statement \cite[Theorem 5.8]{evens1997fourier} was proved for $\tIM$ and $\tgIM$:
\begin{equation}
\begin{tikzcd}
\Irr(\cH_\chi) \arrow[r,"\tIM"] \arrow[d] & \Irr(\cH_{\check\chi}) \arrow[d] 	
\\
\Irr(\bH_{\bar\chi'}) \arrow[r, "\tgIM"]  & \Irr(\bH_{\check{\bar\chi}'}),
\end{tikzcd}
\end{equation}
where $\check\chi$ is the central character of $\cH$ corresponding to $(W\cdot s^{-1},v_0)$ and $\check{\bar \chi}'$ is the central character of $\bH_{\sigma^{-1}}$ corresponding to $(W(\sigma^{-1})\cdot\log(s_r^{-1}),r_0)$.
The proof for \cite[Theorem 5.8]{evens1997fourier} is however a lot different. It uses the fact that $\Psi_{\check\chi}\circ\tIM(\theta_x) = \tgIM\circ\Psi_{\chi}(\theta_x)$ for $x \in X^*$, whereas $\Psi_{\check\chi}\circ\IM(\theta_x) = \Psi_{\check\chi}(T_{w_0}\theta_{w_0(x)}T_{w_0}^{-1})$ is more difficult to determine since the image of $T_{w_0}$ under $\Psi_{\bar\chi}$ has a complicated form as can be seen from \Cref{prop:Psi}.
The author is not aware of a proof of \Cref{theorem:comdiagHecke} and \Cref{corollary:diagramIM} being well-documented in the literature.
\end{remark}

The proof for the following statement will be given in a forthcoming paper of the author with Emile Okada and Jonas Antor: if $X$ is an irreducible representation of $\Gk$ corresponding to the irreducible representation $\pi$ of the relevant geometric Hecke algebra, then we have $\AZ(X) = \IM(\pi)$.

\section{Maximality and minimality results for $\SO(N,\C)$ and $\Sp(2n,\C)$}\label{sec:max}

\subsection{Green functions}\label{subsec:green}
Let $G$ be a connected complex group.
Let $(e,\psi) \in \NNG$, $(L,C_L,\cL) = \csup(e,\psi) \in S_G$ and consider the following $W_L$-representation 
(see \eqref{eq:genspringerfibre})
\begin{equation}
Y(e,0,0)^\psi \cong H_{\bullet}(\cP_e,\dot{\cL})^\psi.
\end{equation}
For the cases $G = \SO(N,\C)$ or $G=\Sp(2n,\C)$, we want to find $(\emax,\psimax) \in \NNG$ such that $[H_{\bullet}(\cP_e,\dot{\cL})^\psi : \GSpr(\emax,\psimax)] = 1$ and such that for any $(e',\psi') \in \NNG$ with $[H_{\bullet}(\cP_e,\dot{\cL})^\psi : \GSpr(e',\psi')] > 0$, we have $G\cdot\emax \succ G\cdot e'$ (strict inequality) or $(G\cdot\emax,\psimax) = (G\cdot e',\psi')$. 
We will call $\GSpr(\emax,\psimax)$ the \emph{maximal $W_L$-subrepresentation of $H_{\bullet}(\cP_e,\dot{\cL})^\psi$}.
For $(e,\psi),(e',\psi') \in \NNG$ and $\rho = \GSpr(e,\psi)$, $\rho'=\GSpr(e',\psi')$, define the \emph{Green function} $\tilde P_{\rho',\rho}  \in \Z(t)$ by
\begin{align}
\tilde P_{\rho',\rho} 
:= \sum_{i\in \Z} [H_{2i}(\cP_e,\dot\cL)^\psi : \GSpr(e',\psi')] t^i.
\end{align}
If $(e,\psi),(e',\psi')\in\NNG$ correspond to $(C,\cE)$,$(C',\cE')\in\NNG$, define 
\begin{align}\label{eq:multgreen}
\mult(C,\cE;C',\cE'):=
\mult(e,\psi;e',\psi') := \tilde P_{\rho',\rho}(1) = [H_{\bullet}(\cP_e,\dot{\cL})^\psi : \GSpr(e',\psi')].
\end{align}
Also denote by $(\Cmax,\Emax)$ the element of $\NNG$ corresponding to $(\emax,\psimax)$.

\begin{theorem}[{\cite[Theorem 24.8(b)]{lusztig1986character}}]\label{theorem:lusztig}
There exist unique $|W_L^\wedge|\times|W_L^\wedge|$ matrices $\tilde P = \tilde P^{(k)}$ and $\tilde \Lambda = \tilde \Lambda^{(k)}$ over $\Q(t)$ such that
\begin{align}
\tilde P \tilde\lambda \tilde P^t &= \tilde \Omega, &&
\\
\tilde\lambda_{\rho,\rho'} &= 0 && \text{if } \rho \not\sim \rho', \label{LScondition1}
\\
\tilde p_{\rho,\rho'} &= 0 && \text{unless $\rho' \prec \rho$ and $\rho \not\sim \rho'$, or $\rho = \rho'$,} \label{LScondition2}
\\
\tilde p_{\rho,\rho} &= 1.
\end{align}
Furthermore, the entries of $\tilde P, \tilde \Lambda$ lie in $\Z[t]$.
\end{theorem}

\begin{example}\label{subsec:algA}
Let $G = \SL(N,\C)$.
See \cite[\S10.3]{lusztig1984intersection}, \cite[\S5]{lusztig1985generalized}, or \cite[\S5.6.1]{ciubotaru2023wavefront} for more details about the generalised Springer correspondence of $G$.
Let $(C,\E) \in \NNG$ and suppose $C$ is parametrised by some partition $\lambda = (\lambda_1,\dots,\lambda_t)$ of $N$. 
Then $A_G(C) \cong \ZZ{n'}$ where $n' = \gcd(\lambda_1,\dots,\lambda_t)$. 
Denote by $\psi$ the representation of $\ZZ{n'}$ corresponding to $\E$.
Let $d$ be the order of $\psi$. Lusztig showed that $d\mid\lambda_i$ for $i=1,\dots,t$. As a representation of $S_{n/d}$, $\GSpr(e,\phi)$ is parametrised by the partition $\lambda/d := (\lambda_1/d,\dots,\lambda_t/d)$ of $n/d$.
By the combinatorial results regarding the Green functions for type $A$ in \cite[III.6.]{macdonald1998symmetric}, $\GSpr(\Cmax,\Emax)$ is the representation of $S_{n/d}$ parametrised by the partition $(n/d)$ of $n/d$, which by the generalised Springer correspondence implies that $\Cmax$ is parametrised by $\lambdamax = (n)$ and $\Emax$ corresponds to $\epsmax = \eps$.
\end{example}

\subsection{$G = \SO(N,\C)$}\label{subsec:maxminSO}
The matrix $\tilde P$ in \Cref{theorem:lusztig} is the same as in \cite[Theorem 3.3]{la2022maximality}. 

\begin{theorem}
[{\cite[Theorem 4.2, Theorem 5.1]{la2022maximality}}]
\label{thm:max}
Suppose $(\lambda,[\eps]) \in \PPort(N)$ such that $\lambda$ only has odd parts. Then there exists a unique $(\lambda^{\text{max}},[\epsmax]) \in \PPort(N)$ such that $\lambdamax$ only has odd parts and 
\begin{enumerate}[(1)]
\item\label{thm:max1} $\mult(C_\lambda,\mathcal E_{[\eps]}; C_{\lambdamax},\mathcal E_{[\epsmax]}) = 1$,
\item\label{thm:max2} For all $(\lambda',[\eps']) \in \PPort(N)$ with $\mult(C_\lambda,\E_{[\eps]};C_{\lambda'}^+,\mathcal E_{[\eps']}^+)= \mult(C_\lambda,\E_{[\eps]};C_{\lambda'}^-,\mathcal E_{[\eps']}^-) \neq 0$, we have $\lambda' < \lambda^{\text{max}}$ or $(\lambda',[\eps']) = (\lambda^{\text{max}},[\epsmax])$.
\end{enumerate}
\end{theorem}

Let $n \in \N$ and let $(\alpha,\beta) \in \mathcal P_2(n)$. As representations of $W(B_n)$, we have $\sgn\otimes\rho_{(\alpha,\beta)} = \rho_{(\transp\beta,\transp\alpha)}$, where $\sgn$ is the sign representation of $W(B_n)$. 
Let $i\in\{1,2/c_{(\alpha,\beta)}\}$ and note that $c_{(\alpha,\beta)} = c_{(\transp\beta,\transp\alpha)}$. As representations of $W(D_n)$, we have $\sgn\otimes\rho_{(\alpha,\beta)_i} = \rho_{(\transp\beta,\transp\alpha)_i}$, where $\sgn$ is the sign representation of $W(D_n)$. 
Let $(\lambda,[\eps]) \in \PPort(N)$, $k = k(\lambda,[\eps])$. If $(\alpha,\beta)_k = \Phi_N(\lambda,[\eps])$, then let $(\transps\lambda,\transps\eps) = \Phi_N^{-1}((\transp\beta,\transp\alpha)_k)$.

\begin{theorem}[{\cite[Theorem 4.6]{la2022maximality}}]\label{thm:min}
Let $(\lambda,[\eps]) \in \PPort(N)$ with $\lambda$ only consisting of odd parts. Then there exists a unique $(\lambdamin,\epsmin) \in \PPort(N)$ such 
\begin{enumerate}
\item $\mult(C_\lambda,\E_{[\eps]};C_{\slambdamin},\E_{[\sepsmin]}) = 1$,
\item For all $(\lambda',[\eps']) \in \PPort(N)$ with 
$\mult(C_\lambda,\mathcal E_{[\eps]};C_{\slambda'}^+,\E_{[\seps']}^+) 
= 
\mult(C_\lambda,\mathcal{E}_{[\eps]};C_{\slambda'}^-,\E_{[\seps']}^-) \neq 0$, 
we have $\lambdamin < \lambda'$ or $(\lambda',[\eps']) = (\lambdamin,[\epsmin])$.
\end{enumerate}
Furthermore, we have $(\lambdamax,[\epsmax]) = (\slambdamin,[\sepsmin])$.
\end{theorem}

The results \cite[Theorem 4.9, Theorem 4.10]{la2022maximality} are generalisations of \Cref{thm:max} and \Cref{thm:min}, removing the assumption that $\lambda$ only has odd parts, but with the assumption that $k:=k(\lambda,\eps) = 0$ or $k = 1$ (i.e. $(\lambda,\eps)$ appears in the ordinary Springer correspondence), since \cite[Proposition 4.8]{la2022maximality} is a result specific to Springer fibres.
\Cref{prop:induction} generalises \cite[Proposition 4.8]{la2022maximality}, and replacing each usage of \cite[Proposition 4.8]{la2022maximality} by \Cref{prop:induction} as well as replacing the Springer fibres by the varieties $H_{\bullet}(\cP_e^G,\dot{\cL})^\phi$ for appropriate $(e,\phi)$, the exact same proofs for \cite[Theorem 4.9, Theorem 4.10]{la2022maximality} given in \cite[\S4.3]{la2022maximality} give us generalisations \Cref{thm:maxeven} and \Cref{thm:mineven} without any assumptions on $k$.

Let $(e,\phi) \in \NNG$ and $(L,C_L,\cL) = \csup(e,\phi)$.
Let $M$ be a Levi subgroup of $G$ 
such that $\fm$ intersects $G\cdot e$. We  may assume that $e \in \fm$ otherwise we replace $e$ by a $G$-conjugate.
Note that $S_M$ is a subset of $S_G$ so the cuspidal support of $(e,\phi)$ in $M$ is $(L,C_L,\cL)$ as well. The corresponding relative Weyl group $W_L^M = N_M(L)/L$ in $M$ is a subgroup of $W_L$.
We have the following induction theorem, which is a generalisation of the main theorem of \cite{lusztig2004induction}, which was stated in \cite{alvis1982springer} without proof; see also \cite[Proposition 3.3.3]{reeder2001euler}. 
\begin{proposition}\label{prop:induction}
As $W$-representations, we have
\begin{equation}
H_{\bullet}(\cP_e^G,\dot{\cL})^\phi \cong \ind_{W_L^M}^{W_L} H_{\bullet}(\cP_e^M,\dot{\cL})^\phi.
\end{equation}
\end{proposition}
\begin{proof}
Taking $\sigma = r = 0$ and using \eqref{eq:genspringerfibre}, \cite[Proposition 3.22(a)]{aubert2018graded} gives an isomorphism of $\bH(G,L,C_L,\cL)$-modules
\begin{align}
H_{\bullet}(\cP_e^G,\dot{\cL})^\phi 
&\cong 
\bH(G,L,C_L,\cL)\otimes_{\bH(Q,L,C_L,\cL)} H_{\bullet}(\cP_e^M,\dot{\cL})^\phi
\\
&\cong
(\C[W_L] \otimes \C[r] \otimes \bA) \otimes_{\C[W_L^M] \otimes \C[r] \otimes \bA} H_{\bullet}(\cP_e^M,\dot{\cL})^\phi,
\label{eq:WLmoduleH}
\end{align}
where $\bA$ is as in \Cref{sec:GHA}. Restricting to $W_L$ yields us the desired result.
\end{proof}

Suppose $(e,\phi)$ corresponds to $(\lambda,[\eps]) \in \PPort(N)$ and let $k = k(\lambda,\eps)$. 
Let $\epsodd = \eps$ and write $\lambda = \lambdaodd \sqcup \lambdaeven$ where $\lambdaodd$ (resp. $\lambdaeven$) consists of all the odd (resp. even) parts of $\lambda$. 
Write $\lambdaeven = (a_1,a_1,a_2,a_2,\dots,a_\ell,a_\ell)$ with $a_1 \geq a_2 \geq \dots \geq a_\ell$ and let $a = a_1 + \dots + a_\ell$.
Let $N' = t(\lambdaodd)$, $n' = (N' - k^2)/2$, and $n = (N-k^2)/2$.  
Suppose that
\begin{align}
M \cong \SO(N') \oplus \GL(a_1) \oplus \GL(a_2) \oplus \dots \oplus \GL(a_\ell).
\end{align}
Note that $A_M(e) = A_G(e)$.
We have
\begin{align}
W_L^M &\cong W_{n'} \times \prod_{i=1}^\ell S_{a_i},
&
W_L = W_{n},
\end{align}
where $W_n = W(B_n)$ and $W_{n'} = W(B_{n'})$ if $k \neq 0$, and $W_{n} = W(D_{n})$ and $W_{n'} = W(D_{n'})$ if $k = 0$.
Let $\eodd \in \fm$ be a nilpotent element parametrised by $\lambdaodd$.
By \Cref{prop:induction}, we have 
\begin{equation}\label{eq:induction}
H_{\bullet}(\cP_e^G,\dot{\cL})^\phi 
\cong 
\ind_{W_L^M}^{W_L} (H_{\bullet}(\cP_{\eodd}^M,\dot\cL)^\phi \boxtimes \prod_{i=1}^\ell \triv_{S_{a_i}}),
\end{equation}
where $\triv_{S_{a_i}}$ is the trivial representation of $S_{a_i}$. 
We can explicitly determine the $W_L$-subrepresentations of $H_{\bullet}(\cP_e^G,\dot{\cL})^\phi$
using the Pieri's rule, see \cite[\S6]{geck2000characters} for type $B$, \cite{taylor2015induced} for type $D$, or \cite{stembridgepractical} for an overview of type $B$ and $D$. 
Let $I_0$ be the set of $(\alpha,\beta) \in \mathcal P_2(n')$ such that $\rho_{(\alpha,\beta)}$ is a $W_L$-subrepresentation of $H_{\bullet}(\cP_{\eodd}^M,\dot\cL)^\phi$.
For $i=0,\dots,\ell$, let $n_i = n' + a_1 + \dots + a_i$.
We recursively define sets $I_i$ as follows. 
Suppose $i \in \{1,\dots,\ell\}$.
Let $I_i$ be the set of $(\gamma,\delta) \in \mathcal P_2(n_i)$ such that there exist a $(\tilde\gamma,\tilde\delta) \in I_{i-1}$ such that for all $j \in \Z$, we have
\begin{equation}\label{eq:pieri}
\tilde\gamma^t_j \leq \gamma^t_j \leq \tilde\gamma^t_j + 1
\quad \text{and} \quad
\tilde\delta^t_j \leq \delta^t_j \leq \tilde\delta^t_j + 1.
\end{equation}
By Pieri's rule for type $B$ and $D$, $I_i$ consists of precisely all $(\gamma,\delta)  \in \mathcal P_2(n_i)$ such that $\rho_{(\gamma,\delta)}$ is a $W_{n_i}$-subrepresentation of $\ind_{W_{n_{i-1}}\times S_{a_i}}^{W_{n_i}}(\rho_{(\tilde\gamma,\tilde\delta)} \boxtimes \triv_{S_{a_i}})$. 
Since induction and direct products are compatible, it follows from \eqref{eq:induction} that $I_\ell$
consists of all irreducible constituents of $H_{\bullet}(\cP_e^G,\dot{\cL})^\phi$.

Let $(\lambda,[\eps]) \in \PPort(N)$ and $k:= k(\lambda,[\eps])$.
Let $((\lambdaodd)^{\text{max}},[(\epsodd)^{\text{max}}])$ be as in \Cref{thm:max} for $(\lambdaodd,[\epsodd])$.
Define $(\lambdamax,[\epsmax]) \in \PPort(N)$ such that for all $j \in \Z_{\geq2}$, we have
\begin{align}
\lambda^{\text{max}}_1 &= (\lambdaodd)^{\text{max}}_1 + \sum_{i\in\N} \lambda^{\text{even}}_i =  (\lambdaodd)^{\text{max}}_1 + 2a,
\\
\lambda^{\text{max}}_j &= (\lambdaodd)^{\text{max}}_j,
\end{align}
and such that $\epsmax = (\eps^{\text{odd}})^{\text{max}}$.
Note that $\lambdamax$ is degenerate.
\begin{theorem}\label{thm:maxeven}
Let $(\lambda,[\eps]) \in \PPort(N)$ and $k:= k(\lambda,[\eps])$. 
Then $(\lambdamax,[\epsmax])$ is the unique element of $\PPort(N)$ such that 
\begin{enumerate}
\item\label{thm:maxeven1} $\mult(C_\lambda,\E_{[\eps]};C_{\lambdamax},\E_{[\epsmax]}) = 1$,
\item\label{thm:maxeven2} For all $(\lambda',[\eps']) \in \PPort(N)$ with $\mult(C_\lambda,\E_{[\eps]}^\pm;C_{\lambda'},\E_{[\eps']}^\pm) \neq 0$, we have $\lambda' < \lambda^{\text{max}}$ or $(\lambda',[\eps']) = (\lambdamax,[\epsmax])$.
\end{enumerate}
\end{theorem}
\begin{proof}
The proof is exactly the same as the proof of \cite[Theorem 4.9]{la2022maximality}, which we will now restate.
Let $n = (N-k^2)/2$, let $(\alpha,\beta)_k = \Phi_N(\lambdaodd,[\epsodd])$ and pick $(\mu,\nu) \in \cP_2(n)$ such that $(\mu,\nu)_k = \Phi_N((\lambdaodd)^{\text{max}},[(\epsodd)^{\text{max}}])$.
If 
$\alpha_1 + 2k > \beta_1$, 
let $(\bm\alpha,\bm\beta) = (\mu + (a,0,0,\dots),\nu)$
and if 
$\alpha_1 + 2k < \beta_1$, 
let $(\bm\alpha,\bm\beta) = (\mu,\nu + (a,0,0,\dots))$.
Then $(\bm\alpha,\bm\beta)_k = \Phi_N(\lambdamax,[\epsmax])$.
We have $(\bm\alpha,\bm\beta) \in I_\ell$.
Recall \eqref{eq:Lambda} and define
$\Lambda = \Lambda_{k,-k;2}(\mu,\nu)$ and
$\bm\Lambda = \Lambda_{k,-k;2}(\bm\alpha,\bm\beta)$.
Let $(\bm\alpha',\bm\beta') \in I_\ell$.
By definition of $I_\ell$, there exists an $(\alpha',\beta') \in I_0$ and $x \in \Z_{\geq0}^{t(\alpha)},y \in \Z_{\geq0}^{t(\beta)}$ such that
\begin{equation}
\bm\alpha' = \alpha' + x
\quad
\bm\beta' = \beta' + y,
\quad
\sum_j (x_j + y_j) = a.
\label{eq:gammadelta}
\end{equation}
If $(\alpha',\beta') \neq (\mu,\nu)$, then $(\bm\alpha',\bm\beta') \neq (\bm\alpha,\bm\beta)$, and by \Cref{thm:max}, the multiplicity of $\rho_{(\mu,\nu)}$ in $H_{\bullet}(\cP_{e'}^G,\dot{\cL})^{\phi'}$ is $1$. 
Furthermore, the multiplicity of $\rho_{(\bm\alpha,\bm\beta)}$ in 
\begin{equation}\label{eq:indLR}
\ind_{W_{n'}\times \prod_{i=1}^\ell S_{a_i}}^{W_{n}}(\rho_{(\alpha,\beta)} \boxtimes \prod_{i=1}^\ell \triv_{S_{a_i}})
\end{equation}
can be computed using results  from \cite[\S2.A, \S2.B]{stembridgepractical} for type $B$ and \cite[\S3.A, \S3.C]{stembridgepractical} for type $D$, and it follows that this multiplicity is $1$.
Thus part \ref{thm:maxeven1} of the theorem follows. 
Let $z$ be the decreasing sequence in $\mathcal R$ obtained by rearranging the terms of $x \sqcup y$ and adding zeroes at the end.
Let
$\Lambda' = \Lambda_{k,-k;2}(\alpha',\beta')$
and
$\bm\Lambda' = \Lambda_{k,-k;2}(\bm\alpha',\bm\beta')$. 
Let $i=1,\dots,\ell$.
By \eqref{eq:gammadelta}, we have $S_i(\Lambda') + S_i(z) \geq S_i(\bm\Lambda')$ and $a_1+\dots+a_\ell \geq S_i(z)$. 
Furthermore, $\Lambda_i \geq \Lambda_i'$ by \cite[Lemma 1.2]{la2022maximality}, so 
\begin{equation}
S_i(\bm\Lambda)
= S_i(\Lambda) + a_1 + \dots + a_\ell 
\geq S_i(\Lambda') + S_i(z) \geq 
S_i(\bm\Lambda').
\end{equation}
Hence $\bm\Lambda \geq \Lambda_{k,-k;2}(\bm\alpha',\bm\beta')$ for all $(\bm\alpha',\bm\beta') \in I_\ell$, so part \ref{thm:maxeven2} of the theorem follows from \cite[Lemma 4.1]{la2022maximality}.
To show that $(\lambdamax,[\epsmax])$ is unique, suppose that $\bm\Lambda = \bm\Lambda'$.
Then $\bm\Lambda_1' = \bm\Lambda_1 = \Lambda_1+a$. Since $\Lambda_1 \geq \Lambda'_1$ and $\bm\Lambda_1' \leq \Lambda'_1+a$, we have $\Lambda_1 = \Lambda_1'$. From this and \eqref{eq:gammadelta}, it follows that $\Lambda_i = \bm\Lambda_i = \bm\Lambda_i' = \Lambda_i'$ for $i\geq2$, and so $\Lambda = \Lambda'$.  By \cite[Lemma 1.2, 1.5]{la2022maximality} and the last sentence of the proof of \cite[Theorem 4.2]{la2022maximality} it follows that $(\alpha',\beta') = (\mu,\nu)$, and using that $\bm\Lambda_1' = \Lambda_1' + a$, we find that $(\bm\alpha,\bm\beta) = (\bm\alpha',\bm\beta')$. Thus $(\lambdamax,[\epsmax])$ is unique. 
\end{proof}

\begin{theorem}\label{thm:mineven}
Let $(\lambda,[\eps]) \in \PPort(N)$ and $k:= k(\lambda,[\eps])$ and consider the notation as above.
Then there exists a unique $(\lambdamin,[\epsmin]) \in \PPort(N)$ such that $\slambdamin$ is non-degenerate and
\begin{enumerate}
\item $\mult(C_\lambda,\E_{[\eps]};C_{\slambdamin},\E_{[\sepsmin]}) = 1$,
\item For all $(\lambda',[\eps']) \in \PPort(N)$ with $\mult(C_\lambda,\E_{[\eps]};C_{\lambda'}^\pm,\E_{[\eps']}^\pm) \neq 0$, we have $\lambdamin < \lambda'$ or $(\lambda',[\eps']) = (\slambdamin,{}^s[\epsmin])$.
\end{enumerate}
Furthermore, we have $(\lambdamax,[\epsmax]) = (\slambdamin,[\sepsmin])$.
\end{theorem}
\begin{proof}
The proof is exactly the same as the proof of \cite[Theorem 4.10]{la2022maximality}.
Let $(\bm\alpha,\bm\beta),(\bm\alpha',\bm\beta')$ as in the proof of \Cref{thm:maxeven}.
By the same arguments as in the proof of \Cref{thm:maxeven}, we can show that $\Lambda_{k/2,-k/2;2}(\bm\alpha,\bm\beta) \geq \Lambda_{k/2,-k/2;2}(\bm\alpha',\bm\beta')$, with equality if and only if $(\bm\alpha,\bm\beta) = (\bm\alpha',\bm\beta')$. 
By the arguments in the proof of \Cref{thm:min} (which is given in \cite{la2022maximality} as the proof of \cite[Theorem 4.6]{la2022maximality}), it follows that $(\lambdamin,[\epsmin]) = (\slambdamax,[\sepsmax])$ is the unique element of $\PPort(N)$ satisfying the two conditions of the theorem. 
\end{proof}

\begin{remark}\label{rem:maxodd}
By similar arguments and using analogous results for $\Sp(2n,\C)$ from \cite{waldspurger2019proprietes}, we similarly obtain the analogues of \Cref{thm:maxeven} and \ref{thm:mineven} for $\Sp(2n,\C)$ for any $(\lambda,[\eps]) \in \PPsymp(2n)$.
\end{remark}

\subsection{$G = \Sp(2n,\C)$}\label{subsec:maxminSp}
The $\Sp(2n,\C)$ analogues of \Cref{thm:max} and \Cref{thm:min} are \cite[Theorem 4.5, 4.7]{waldspurger2019proprietes}.
Completely similarly as in \Cref{subsec:maxminSO}, we can prove the analogues of \Cref{thm:maxeven} and \Cref{thm:mineven} for $\Sp(2n\C)$, generalising \Cref{thm:max} and \Cref{thm:min} by dropping the assumption that $\lambda$ only has even parts.
We shall state these generalisations without proof.

\begin{theorem}[{\cite[Theorem 4.5]{waldspurger2019proprietes}}]\label{thm:maxSp}
Let $(\lambda,\eps) \in \PPsymp(2n)$. Then there exists a unique $(\lambdamax,\epsmax) \in \PPsymp(2n)$ such that 
\begin{enumerate}
\item\label{thm:max1} $\mult(C_\lambda,\mathcal E_{\eps}; C_{\lambdamax},\mathcal E_{\epsmax}) = 1$,
\item\label{thm:max2} For all $(\lambda',\eps') \in \PPort(N)$ with $\mult(C_\lambda,\E_{\eps};C_{\lambda'},\mathcal E_{\eps'}) \neq 0$, we have $\lambda' < \lambdamax$ or $(\lambda',\eps') = (\lambdamax,\epsmax)$.
\end{enumerate}
\end{theorem}

\begin{theorem}[{\cite[Theorem 4.7]{waldspurger2019proprietes}}]\label{thm:minSp}
Let $(\lambda,\eps) \in \PPsymp(2n)$. Then there exists a unique $(\lambdamin,\epsmin) \in \PPsymp(2n)$ such that $\slambdamin$ 
\begin{enumerate}
\item $\mult(C_\lambda,\E_{\eps};C_{\slambdamin},\E_{\sepsmin}) = 1$,
\item For all $(\lambda',\eps') \in \PPort(N)$ with 
$\mult(C_\lambda,\E_{\eps};C_{\slambda'},\E_{[\seps']}) 
\neq 0$, 
we have $\lambdamin < \lambda'$ or $(\lambda',\eps') = (\lambdamin,\epsmin)$.
\end{enumerate}
Furthermore, we have $(\lambdamax,\epsmax) = (\slambdamin,\sepsmin)$.
\end{theorem}

\subsection{Algorithm for $(\lambdamax,\epsmax)$ for $\SO(N,\C))$}\label{sec:algSO}
The following is almost exactly the first part \cite[\S5]{la2022maximality}.
Let $N\in \N$.
Recall $\PPPort(N)$ defined in \Cref{subsec:gscSO} and let $(\lambda,\eps) \in \PPPort(N)$.
Suppose $\lambda$ only has odd parts. We will give an algorithm that outputs an element $(\bar\lambda,\bar\eps) \in \PPPort(N)$. The algorithm is a direct analogue of \cite[\S5]{waldspurger2019proprietes}.
such that $(\bar\lambda,[\bar\eps]) = (\lambdamax,[\epsmax]) \in \PPort(N)$. 

Let $t = t(\lambda)$. 
We view $\eps$ as a map $\{1,\dots,t\} \to \{\pm1\}$ by setting $\eps(i) = \eps_{\lambda_i}$ for all $i \in \N$. Let $u \in \{\pm1\}$ and consider the finite sets
\begin{align}
\mathfrak S &= \{1\} \cup \{i \in \{2,\dots,t\} \colon \eps(i) = \eps(i-1)\} = \{s_1 < s_2 < \dots < s_p\}, \text{ for some $p \in \N$, }
\\
J^u &= \{ i \in \{1,\dots,t\} \colon \eps(i)(-1)^{i+1} = u\},
\\
\tilde J^u &= J^u \setminus \mathfrak S.
\end{align}
Recall from \Cref{subsec:gscSO} that $M(\lambda,\eps) = |J^1| - |J^{-1}|$ and that $k(\lambda,\eps) = |M|$. Note that $\mathfrak S, J^u, M$ depend on the choice of representative $\eps$ of $[\eps]$. In particular, we have $M(\lambda,\eps) = -M(\lambda,-\eps)$.

We define $(\bar\lambda,\bar\eps)$ by induction. For $N=0,1$. Let $(\bar\lambda,\bar\eps) = (\lambda,\eps)$. Suppose $N > 1$ and suppose that we have defined $(\bar\lambda',\bar\eps')$ for any $(\lambda',\eps') \in \PPPort(N')$ with $N'<N$.
Let 
\begin{align}
\bar\lambda_1 
&=
\sum_{i \in \mathfrak S} \lambda_{i} - 2|\tilde J^{-\eps(1)}| - \frac{1 + (-1)^{|\mathfrak S|}}{2}
=
\begin{cases}
\sum_{i \in \mathfrak S} \lambda_{i} - 2|\tilde J^{-\eps(1)}| &\text{if $|\mathfrak S|$ is odd,}
\\
\sum_{i \in \mathfrak S} \lambda_{i} - 2|\tilde J^{-\eps(1)}| - 1 &\text{if $|\mathfrak S|$ is even,}
\end{cases}
\\
\bar \eps(1) &= \eps(1).
\end{align}
Let $r' = |\tilde J^1| + |\tilde J^{-1}| = t(\lambda) - |\mathfrak S|$ and $\phi \colon \{1,\dots,r'\} \to \tilde J^1 \cup \tilde J^{-1}$ be the unique increasing bijection. 
Let 
$N' = N - \bar\lambda_1$. 
We define $\lambda' \in \Z^{r'+(1+(-1)^{|\mathfrak S|})/2}$ and $\eps' \colon \{1,\dots,r'+\frac{1+(-1)^{|\fS|}}{2}\} \to \{\pm1\}$ as follows.
\begin{align}\label{eq:lambdaprime}
(\lambda'_j,\eps'(j)) &= 
\begin{cases}
(\lambda_{\phi(j)},\eps(\phi(j))) &\text{ if } j\leq r', \phi(j) \in \tilde J^{\eps(1)},
\\
(\lambda_{\phi(j)}+2,-\eps(\phi(j))) &\text{ if } j\leq r', \phi(j) \in \tilde J^{-\eps(1)},
\\
(1,(-1)^{r'}\eps(1)) &\text{ if } j = r'+1 \text{ and $|\mathfrak S|$ is even},
\\
0 &\text{ else.}
\end{cases}
\end{align}
Note that we can similarly define $(\lambda',(-\eps)')$, in which case we have $(-\eps)' = -\eps'$, i.e. $[-\eps'] = [\eps']$.
We have $N' < N$ by \cite[Proposition 5.4(2)]{la2022maximality}.
By induction, we have defined $(\bar\lambda', \bar \eps') \in \PPPort(N')$ by applying the algorithm to $(\lambda',\eps') \in \PPPort(N')$. 

Let $\ell = t(\bar\lambda')$. We define $\bar\lambda \in \Z^{\ell+1}$ and $\bar\eps \colon \{1,\dots,\ell+1\} \to \{\pm1\}$ as follows: for $i = 1,\dots,\ell$, let
\begin{align}\label{eq:lambdabar}
\bar\lambda_{i+1} &:= \bar\lambda'_{i-1},
\\
\bar \eps(i+1) &:= \bar\eps'(i).
\end{align}

\begin{theorem}\cite[Theorem 5.1]{la2022maximality}\label{thm:algorithm}
Let $(\lambda,\eps) \in \PPPort(N)$ such  that $\lambda$ only has odd parts and let $(\lambdamax,[\epsmax]) \in \PPort(N)$ be as in \Cref{thm:max}.
Then it holds that
\begin{equation}
(\bar\lambda,\bar\eps) = (\lambdamax,[\epsmax])
\end{equation}
\end{theorem}

For $(\lambda,\eps) \in \PPPort(N)$ arbitrary, write $\lambda = \lambdaeven \sqcup \lambdaodd$ where $\lambdaeven$ only consists of even parts and $\lambdaodd$ only consists of odd parts. We can apply \Cref{thm:algorithm} to compute $(\lambdaoddmax,\epsmax)$ and then use \Cref{thm:maxeven} to compute $(\lambdamax,\epsmax)$: 
\begin{align}
(\lambdamax)_1 &= (\lambdaoddmax)_1 + S(\lambdaeven),
\\
(\lambdamax)_i &=  (\lambdaoddmax)_i &\text{for $i \geq 2$.}
\end{align}
From this, we can simply compute $(\lambdamin,\epsmin) = (\slambdamax,\sepsmax)$ by \Cref{thm:mineven}.

\subsection{Algorithm for $(\lambdamax,\epsmax)$ for $\Sp(2n,\C)$}\label{sec:algSp}
The algorithm for $\Sp(2n,\C)$ is given in \cite[\S5]{waldspurger2019proprietes}. We shall give the algorithm here.

Suppose $(\lambda,\eps) \in \PPsymp(2n)$ such that $\lambda$ only has even parts. If $n = 0$, let $(\bar\lambda,\bar\eps) = (\lambda,\eps)$. If $n > 0$. Similarly as we did previously for $\SO(N,\C)$, we write $\eps$ as a function $\eps \colon \N \to \{\pm1\}$ with $\eps(i) = \eps_{\lambda_i}$ for all $i \in \N$.
Let $\eta \in \{\pm 1\}$ and recall/define
\begin{align}
J^\eta &= \{i \in \N \colon (-1)^{i+1} \eps(i) = \eta \},
\\
\OO &:= \{ i \in \N_{\geq2} \colon \eps(i) = \eps(i-1) \} \cup \{1\},
\\
\tilde J^\eta &:= J_1 \setminus \OO.
\end{align}
Note that $\N \setminus \OO$ is finite, hence $\tilde J^\eta$ is also finite.
Define
\begin{align}
\bar \lambda_1 &= \sum_{i \in \OO} \lambda_i - 2| \tilde J^{-\eps(1)}|,
\\
\bar \eps_{\bar \lambda_1} &= \eps(1).
\end{align}
Also define
\begin{equation}
\lambda' = (\lambda_j \colon j \in \tilde J^{\eps(1)}) \sqcup (\lambda_j + 2 \colon j \in \tilde J^{-\eps(1)}).
\end{equation}
For each $i \in \lambda'$, there exists a $j \in J^{\eps(1)}$ such that $i = \lambda_j$ or a $j \in \tilde J^{-\eps(1)}$ such that $i = \lambda_j +2$. Fix such a $j$ and set
\begin{align}
\eps_i' =
\begin{cases}
\eps(j) &\text{if } j \in \tilde J^{\eps(1)},\\
-\eps(j) &\text{if } j \in \tilde J^{-\eps(1)}.
\end{cases}
\end{align}
This defines a pair $(\lambda',\eps') \in \PPsymp(2n - \bar \lambda_1)$.
We can repeat the steps above for $(\lambda',\eps')$ to obtain $\overline{\lambda'}_1 = \bar \lambda_2$ and $\overline{\eps'}_{\overline{\lambda'}_1} = \bar \eps_{\bar \lambda_2}$ as well as a pair $(\lambda'',\eps'')$, etc. to obtain an element $(\bar\lambda,\bar\eps) \in \PPsymp(2n)$.

\begin{theorem}[{\cite[Th\'eor\`eme 5.1]{waldspurger2019proprietes}}]\label{thm:algorithmSp}
We have $(\bar\lambda,\bar\eps) = (\lambdamax,\epsmax)$.
\end{theorem}

Following \Cref{rem:maxodd}, we can compute $(\lambdamax,\epsmax)$ for arbitrary $(\lambda,\eps) \in \Psymp(2n)$ as follows: write $\lambda = \lambdaeven \sqcup \lambdaodd$ where $\lambdaeven$ only consists of even parts and $\lambdaodd$ only consists of odd parts. We can apply \Cref{thm:algorithmSp} to compute $(\lambdaevenmax,\epsmax)$ and then use \Cref{thm:maxSp} to compute $(\lambdamax,\epsmax)$: 
\begin{align}
(\lambdamax)_1 &= (\lambdaevenmax)_1 + S(\lambdaeven),
\\
(\lambdamax)_i &=  (\lambdaevenmax)_i & \text{for $i \geq 2$.}
\end{align}
From this, we can compute $(\lambdamin,\epsmin) = (\slambdamax,\sepsmax)$ by \Cref{thm:minSp}.

{

\section{$\gIM$ for tempered representations with real infinitesimal character}\label{sec:alggIM}
\subsection{$\gIM$ for graded Hecke algebras of connected complex groups}

Consider the same setting as \Cref{subsec:GHAcuspls}. 
Fix $r_0 \in \C$ and let $(e,\sigma,\psi,r_0) \in \cMtempreal$. 
Following \Cref{prop:tempreal}, we write $Y(e,\psi) = Y(e,\sigma,r_0,\psi)$.
Recall that $Y:=Y(e,\psi)$ is irreducible by \Cref{prop:temperedirreducible}.
We have $A_G(e,\sigma) = A_G(e)$ by \Cref{prop:tempered}, so we view $\psi$ as a representation of $A_G(e)$. 
Let $(L,C_L,\cL) = \csup (e,\psi) \in S_G$.
Let $(e',\psi') \in \cM_{\sigma,r_0}$ such that 
\begin{equation}
\bar Y' := \bar Y(e',\sigma,r_0,\psi') = \gIM(Y).
\end{equation}
Let $Y'=Y(e',\sigma,r_0,\psi')$. Since $\gIM(Y)$ may be non-tempered, we may have $Y' \neq \bar Y'$.

\begin{theorem}\label{prop:gIMminimal}
Suppose there is a unique $(\emin,\psimin) \in \NNG$, 
such that 
$[Y'|_{W_L}:\GSpr(\emin,\psimin)] = 1$ and such that for all $(e'',\psi'') \in \NNG$ with
$[Y'|_{W_L}:\GSpr(e'',\psi'')] >0$,
we have either $G\cdot\emin \prec G\cdot e''$ (strict) or $(G\cdot \emin,\psimin) = (G\cdot e'',\psi'')$. 
Then for some $e' \in G \cdot \emin$, we have
\begin{equation}\label{eq:gIMminimal}
\gIM(Y) = \bar Y(e',\sigma,r_0,\psimin).
\end{equation}
\end{theorem}
\begin{proof}
Note that $\psi' \in A_G(e',\sigma)^\wedge$. Since $\sigma$ is real, $A_G(\sigma)$ is a Levi subgroup of $G$ and so the inclusion $Z_G(e',\sigma) \hookrightarrow Z_G(e')$ induces an injection $A_G(e',\sigma)$ into $A_G(e')$. 
It holds that $\bar Y'|_{W_L}$ is a $W_L$-subrepresentation of 
\begin{equation}\label{eq:gIMspringerfibres}
 Y'|_{W_L}
\cong H_{\bullet}(\cP_{e'},\dot{\cL})^{\psi'} 
\cong \bigoplus_{\phi' \in A_G(e')^\wedge} [\phi' : \psi'] H_{\bullet}(\cP_{e'},\dot{\cL})^{\phi'}.
\end{equation}
Recall the multiplicities \eqref{eq:multgreen}.
By \cite[Theorem 24.8(b)]{lusztig1986character}, for each $\phi' \in A_G(e')^\wedge$ we have $\mult(e',\phi';e',\phi') = 1$ and
for each $(e'',\phi'') \in \NNG$ with $\mult(e',\phi';e'',\phi'') > 0$, we have $G\cdot e'' > G\cdot e'$ (strictly) or $(G\cdot e'',\phi'') = (G\cdot e',\phi')$.
It follows from \cite[Proposition 10.5]{lusztig1995cuspidal} that $\GSpr(e',\phi')$ is a $W_L$-subrepresentation of $\bar Y'|_{W_L}$ (see also \cite[(2.5.5)]{ciubotaru2022wavefront}). 
By \cite[\S10.13]{lusztig1995cuspidal} and \eqref{eq:genspringerfibre}, we have $Y|_{W_L} \cong Y(e,0,0,\psi) \cong H_{\bullet}(\cP_e,\dot{\cL})^\psi$, and so by \Cref{rem:gIMsign}, we also have
\begin{equation}
\bar Y'|_{W_L} = Y|_{W_L} \otimes \sgn \cong H_{\bullet}(\cP_e,\dot{\cL})^\psi \otimes \sgn.
\end{equation}
By the uniqueness assumption of $(\emin,\psimin)$, we have $e' \in G\cdot \emin$, $A_G(e') = A_G(e',\sigma)$, and $\psi' = \psimin$.
\end{proof}

\begin{remark}\label{remark:nonuniquetriple}
In other words, if we find such $(\emin,\psimin)$ in \Cref{prop:gIMminimal}, then we know that $\emin$ is $G$-conjugate to the element $e'$. This does not however imply that $\bar Y(\emin,\sigma,r_0,\psimin)$ is equal to $\bar Y' = \gIM(Y)$:
the intersection of $G\cdot \emin$ with $\fg_q = \{x \in \fg \colon [s,x] = qx\}$ may be a union of multiple $Z_G(s_c)$-orbits, in which case we can pick $\emin$ as in \Cref{prop:gIMminimal} such that it is not $Z_G(s_c)$-conjugate to $e'$, in which case $(e',s',\psi')$ and $(\emin,s',\psimin)$ are not $G$-conjugate. 
\end{remark}
Note that the equality $A_G(e') = A_G(e',\sigma)$ in \Cref{prop:gIMminimal} follows from the uniqueness of $(\emin,\psimin) \in \NNG$: if $A_G(e',\sigma)$ is strictly smaller than $A_G(e')$, then \eqref{eq:gIMspringerfibres} shows that there are various $\phi' \in A_G(e')^\wedge$ such that $(e',\phi')$ satisfies the hypothesis for $(\emin,\psimin)$ of \Cref{prop:gIMminimal}.

\begin{proposition}\label{prop:gIMA}
Suppose $G = \SL(N,\C)$. 
Let $(e,\sigma,\psi) \in \cMtempreal$ and suppose $e$ is parametrised by a partition $\lambda = (\lambda_1,\dots,\lambda_t)$ of $n$.
We use the notation of \Cref{subsec:algA}.
The representation $\rho_{(n/d)} \otimes \sgn = \rho_{(1)^{n/d}}$ is the representation $\GSpr(\emin,\psimin)$ in \Cref{prop:gIMminimal} in the current setting. From the generalised Springer correspondence, we find that $G \cdot \emin = G\cdot e_{(d)^{n/d}}$ and $\psimin = \psi$. 
\end{proposition}

\begin{theorem}\label{thm:gIMB}
Suppose $G = \SO(N,\C)$.
Let $(e,\sigma,\psi)\in \cMtempreal$ and suppose that $(e,\psi)$ corresponds to $(\lambda,[\eps]) \in \PPort(N)$. 
The element $(\emin,\psimin) \in \NNG$ that 
corresponds to $(\lambdamin,[\epsmin])$ in \Cref{thm:mineven} -- which we can explicitly compute using  \Cref{thm:algorithm} -- satisfies the condition in \Cref{prop:gIMminimal}. 
\end{theorem}

\begin{theorem}\label{thm:gIMC}
Suppose $G = \Sp(2n,\C)$. Let $(e,\sigma,\psi,r_0) \in \cMtempreal$ and suppose that $(e,\psi)$ corresponds to $(\lambda,\eps) \in \PPsymp(2n)$. 
The element $(\emin,\psimin) \in \NNG$ that 
corresponds to $(\lambdamin,[\epsmin])$ in \Cref{thm:minSp} -- which we can explicitly compute using  \Cref{thm:algorithmSp} -- satisfies the condition in \Cref{prop:gIMminimal}. 
\end{theorem}

The strategy for exceptional groups is the same as above (see \Cref{app:tablesexceptional}). 

\subsection{From $\gIM$ to $\IM$}\label{subsec:gIMtoIM}
Let $L$ be a standard pseudo-Levi subgroup of $\dG$, $C$ a unipotent class of $L$ and $\cL$ a cuspidal $L$-equivariant local system on $C$.
As discussed in the introduction, the irreducible representations of a geometric affine Hecke algebra $\cH(\dG,L,C,\cL)$ on which $v$ acts as a positive real number are in bijection with the set $\sM$ of $\dG$-orbits of quadruples $(e,s,v_0,\psi)$ where $s \in \dG$ is semisimple, $e \in \dg  = \Lie(\dG)$ is nilpotent such that $\ad(s)e = v_0^2e$, $\psi \in A_{\dG}(e,s)^\wedge$ such that when $\psi$ is viewed as a representation of $A_{Z_{\dG}(s_c)}(e,\log s_r)$, then $\psi \in A_{Z_{\dG}(s_c)}(e,\log s_r)_{(L,C,\cL)}^\wedge$. 
Fix $(e,s,v_0,\psi) \in \sM$ with $s \in T$ and $v_0 >0$. 
By \cite{lusztig1995classification}, $\cH(\dG,L,C,\cL)$ is an affine Hecke algebra $\cH$ as in \Cref{def:AHA}, and $\bH(Z_{\dG}(s_c),L,C,\cL)$ is isomorphic to $\bH_{s_c}$ ($\Gamma_{s_c} = 1$ by \cite[Lemma 4.4]{lusztig1995classification}) associated to $\cH$ in the sense of \Cref{subsec:red}. 
Let $(e,s,v_0,\psi) \in \sM$ and denote the corresponding irreducible representation by $\cY(e,s,v_0,\psi)$, which corresponds to $\bar Y(e,\log s_r,r_0,\psi)$ via the map in \Cref{cor:reduction} (see \cite[\S5.20]{lusztig1995classification}). Additionally, $\cY(e,s,v_0,\psi)$ is tempered if and only if $\bar Y(e,\log s_r,r_0,\psi)$ is tempered.
By \Cref{corollary:diagramIM} we then have $\cY(e',s,v_0,\psi') = \IM(\cY(e,s,v_0,\psi))$ if and only if $\bar Y(e',\log s_r,r_0,\psi') = \gIM(\bar Y(e,\log s_r,r_0,\psi))$. 
In particular, by \Cref{thm:gIMB} and \Cref{thm:gIMC}, we can compute $G\cdot e'$ and $\psi$ when $Z_{\dG}(s_c)$ is a product of symplectic and special orthogonal groups (or type $A$ groups, for which computing $\IM$ and $\gIM$ is well-known, see for instance \cite{moeglin1986involution}).

\section{$\AZ$ for $\SO(2n+1,\sfk)$}\label{sec:algAZ}
Let $\aG = \SO(2n+1)$, $G^\vee = \Sp(2n,\C)$. 
Let $\Phitemp(\dG)$ be the subset of $\Phi(\dG)$ consisting of triples $(e,s,\psi)$ such that $X(e,s,\psi)$ is tempered.
There exist $n_0^+,n_0^-,n^+,n^-n_1,\dots,n_t \in \Z_{\geq0}$ and $I \subseteq \{1,\dots,t\}$ such that $n_1+\dots+n_t+n_0^++n_0^-=n$ and $\sum_{i\in I}n_i + n^+ + n^- = n$ and such that we have isomorphisms
\begin{align}\label{eq:centraliserM} 
&f_{\dM} \colon Z_{\dM}(s_c) \cong \prod_{i=1}^t \GL(n_i,\C) \times \Sp(2n_0^+) \times \Sp(2n_0^-), 
\\
&f_{\dG} \colon Z_{\dG}(s_c) \cong  \prod_{i\in I} \GL(n_i,\C)\times\Sp(2n^+,\C) \times \Sp(2n^-,\C).
\end{align}
Denote by $\iota$ the inclusion $Z_{\dM}(s_c)\hookrightarrow Z_{\dG}(s_c)$ and let $\iota' := f_{\dG} \circ \iota \circ f_{\dM}^{-1} \colon \prod_{i=1}^t \GL(n_i,\C) \times \Sp(2n_0^+) \times \Sp(2n_0^-) \hookrightarrow \prod_{i\in I} \GL(n_i,\C)\times\Sp(2n^+,\C) \times \Sp(2n^-,\C)$. There exist finite sets $J^\pm$ such that $I \cup J^+ \cup J^- = \{1,\dots,t\}$, the direct factors $\prod_{j\in J^\pm}\GL(n_j,\C) \otimes \Sp(2n_0^\pm)$ are mapped into $\Sp(2n^\pm)$.
Also note that $n_0^\pm$ is the the multiplicity of the eigenvalue $\pm1$ of $s_c$, \cite[p.8]{waldspurger2016representations}. Note that $e\in \Lie(Z_{\dG}(s_c)) \cong \prod_{i\in I}\gl(n_i,\C)\times\sp(2n^+,\C) \oplus \sp(2n^-,\C)$. 
By restricting $\eps \colon \Delta(\lambda) \to \{\pm1\}$, we obtain $\eps^+ \in \{\pm1\}^{\Delta(\lambda^+)}$ and $\eps^- \in \{\pm1\}^{\Delta(\lambda^-)}$.

Note that $\psi$ is an irreducible representation of $A_{\dG}(s,e) = A_{Z_{\dG}(s_c)}(\log s_r,e) = A_{Z_{\dG}(s_c)}(e)$ where the last equality comes from \Cref{prop:tempered}.
Consider the tempered (hence irreducible) representation $Y(e,\psi)$. 
From the definition of standard modules, $Y(e,\psi) \cong \bigotimes_{i\in I} Y(e_i,\log s_{i,r},\triv) \otimes Y(e^+,\log s_r^+,\psi^+) \otimes Y(e^-,\log s_r^-,\psi^-)$, where for each $i \in I$, $e_i \in \gl(n_i,\C)$ is parametrised by the partition $(n_i)$ of $n_i$, $(e^\pm,\psi^\pm) \in \cN_{\sp(2n^\pm)}$ corresponds to $(\lambda^\pm,\eps^\pm) \in \PPsymp(2n^\pm)$ and $((s_{i,r})_{i\in I},s_r^+,s_r^-) = f_{\dG}(s_r)$.
By \Cref{prop:temperedirreducible}, each direct factor is tempered.
For each $i \in I$, $Y(e_i,\log s_{i,r},\triv)$ is the Steinberg representation, so $\gIM(Y(e_i,\log s_{i,r},\triv))$ is trivial, i.e. it is equal to $Y(e_i',\log s_{i,r},\triv)$ where $e_i' \in \gl(n_i,\C)$ is the identity (i.e. a nilpotent element parametrised by the partition $(1,\dots,1)$ of $n_i$).
Since $\gIM$ behaves well with tensor products, we have
\begin{equation}\label{eq:AZSp}
\gIM(Y(e,\psi)) 
\cong
\bigotimes_{i\in I} \bar Y(e_i',\log s_{i,r},\triv)
\otimes
\bar Y(e^{+,\mathrm{min}},\log s_r^+,\psi^{+,\mathrm{min}}) \otimes \bar Y(e^{-,\mathrm{min}},\log s_r^-,\psi^{-,\mathrm{min}}) 
\end{equation}
where
$(e^{\pm\mathrm{min}},\psi^{\pm\mathrm{min}}) \in \cN_{\sp(2n^\pm)}$ is parametrised by $(\lambda^{\pm,\mathrm{min}},\eps^{\pm,\mathrm{min}}) \in \PPsymp(2n^\pm)$, which we can compute with the algorithm in \Cref{sec:algSp}. 
From this, we then see again that
\begin{align}
\gIM(Y(e,\psi)) \cong \bar Y(e',\log s_r,\psi'),
\end{align}
where $e'\in\dg$ is some nilpotent element parametrised by the symplectic partition $\lambda^{+,\mathrm{min}} \sqcup \lambda^{-,\mathrm{min}} \sqcup (1,\dots,1)$ of $2n$ where $(1,\dots,1)$ is a partition of $2(n_1+\dots n_t)$ and $\psi' \in A_{\dG}(s,e')^\wedge \cong A_{Z_{\dG}(s_c)}(\log s_r,e')^\wedge = A_{Z_{\dG}(s_c)}(e')^\wedge$ (the last equality follows from \eqref{eq:gIMminimal}) corresponds to the element of $\{\pm1\}^{\Delta(\lambda)}$ whose restriction to $\Delta(\lambda^\pm)$ is $\eps^{\pm,\mathrm{min}}$.

We conclude that 
\begin{equation}
\AZ(X(e,s,\psi)) = X(e'',s,\psi'),
\end{equation}
for some $e'' \in G\cdot e'$ with $e'$ as above.
As noted in \Cref{remark:nonuniquetriple}, we stress again that we were only able to determine the $\dG$-orbit that $e'$. However, the result can still for instance be applied to compute certain wavefront sets following of unipotent representations with real infinitesimal character (in which case we have to take $s_c$ above to be trivial) \cite[Theorem 3.0.8]{ciubotaru2022wavefront} and \cite[Theorem 6.0.3]{ciubotaru2023wavefront}.
}

\appendix
\section{Tables for Exceptional groups}\label{app:tablesexceptional}

Let $G$ be a simply connected complex exceptional group. For each $(e,\psi) \in \NNG$, we showed using the Chevie package of GAP \cite{schonert1992gap,geck1996chevie,michel2015development} that there exists a unique $(\emin,\psimin) \in \NNG$ (up to $G$-conjugation) satisfying the properties as in \Cref{prop:gIMminimal}, and we computed $(G\cdot\emin,\psimin)\in\NNG$, which is displayed in the tables below.
We also verified with GAP that there exists a unique $(\Cmax,\Emax) \in \NNG$ as in \Cref{subsec:green} with $\mult(C,\cE,\Cmax,\Emax) = 1$ and we verified that $\GSpr(\Cmin,\Emin) = \GSpr(\Cmax,\Emax) \otimes \sgn$, where $\sgn$ is the sign representation of the relevant relative Weyl group.

We note that Pramod Achar also wrote a GAP implementation for the Lusztig--Shoji algorithm \cite{achar2008implementation}. Instead of  this, we used the ICCTable function in Chevie for the Lusztig--Shoji algorithm, but we did use the GenApp and TexRepName functions from \cite{achar2008implementation}.

\newpage

\begingroup\scriptsize
\linespread{0}\selectfont
\begin{table}[H]
\begin{minipage}{.45\linewidth}
\caption{$G_2$}
\centering
\begin{tabular}{l|l|l|l|l|l}
$C$ & $A(C)$ & $\varepsilon$ & $C^{\text{min}}$ & $A(C^{\text{min}})$ 
 & $\varepsilon^{\text{min}}$
\\ 
\hline 
$G_2$ & $\langle 1 \rangle$ & $\varnothing$
& $1$ & $\langle 1 \rangle$ & $\varnothing$
\\
$G_2(a_1)$ & $A_2$ & $(3)$
& $1$ & $\langle 1 \rangle$ & $\varnothing$
\\
& &$(21)$
& $A_1$ & $\langle 1 \rangle$ & $\varnothing$
\\
& &$(1^3)$
& $G_2(a_1)$ & $A_2$ & $(1^3)$
\\
$\tilde A_1$ & $\langle 1 \rangle$ & $\varnothing$
& $1$ & $\langle 1 \rangle$ & $\varnothing$
\\
$A_1$ & $\langle 1 \rangle$ & $\varnothing$
& $1$ & $\langle 1 \rangle$ & $\varnothing$
\\
$1$ & $\langle 1 \rangle$ & $\varnothing$
& $1$ & $\langle 1 \rangle$ & $\varnothing$
\\
\end{tabular}
\caption{$F_4$}
\centering
\begin{tabular}{l|l|l|l|l|l}
$C$ & $A(C)$ & $\varepsilon$ & $C^{\text{min}}$ & $A(C^{\text{min}})$ 
 & $\varepsilon^{\text{min}}$
\\ 
\hline 
$F_4$ & $\langle 1 \rangle$ & $\varnothing$
& $1$ & $\langle 1 \rangle$ & $\varnothing$
\\
$F_4(a_1)$ & $A_1$ & $(2)$
& $1$ & $\langle 1 \rangle$ & $\varnothing$
\\
& &$(1^2)$
& $A_1$ & $\langle 1 \rangle$ & $\varnothing$
\\
$F_4(a_2)$ & $A_1$ & $(2)$
& $1$ & $\langle 1 \rangle$ & $\varnothing$
\\
& &$(1^2)$
& $\tilde A_1$ & $A_1$ & $(1^2)$
\\
$B_3$ & $\langle 1 \rangle$ & $\varnothing$
& $1$ & $\langle 1 \rangle$ & $\varnothing$
\\
$C_3$ & $\langle 1 \rangle$ & $\varnothing$
& $1$ & $\langle 1 \rangle$ & $\varnothing$
\\
$F_4(a_3)$ & $A_3$ & $(4)$
& $1$ & $\langle 1 \rangle$ & $\varnothing$
\\
& &$(31)$
& $A_1$ & $\langle 1 \rangle$ & $\varnothing$
\\
& &$(2^2)$
& $\tilde A_1$ & $A_1$ & $(2)$
\\
& &$(21^2)$
& $A_2$ & $A_1$ & $(1^2)$
\\
& &$(1^4)$
& $F_4(a_3)$ & $A_3$ & $(1^4)$
\\
$C_3(a_1)$ & $A_1$ & $(2)$
& $1$ & $\langle 1 \rangle$ & $\varnothing$
\\
& &$(1^2)$
& $A_1$ & $\langle 1 \rangle$ & $\varnothing$
\\
$B_2$ & $A_1$ & $(2)$
& $1$ & $\langle 1 \rangle$ & $\varnothing$
\\
& &$(1^2)$
& $A_1$ & $\langle 1 \rangle$ & $\varnothing$
\\
$\tilde A_2{+}A_1$ & $\langle 1 \rangle$ & $\varnothing$
& $1$ & $\langle 1 \rangle$ & $\varnothing$
\\
$A_2{+}\tilde A_1$ & $\langle 1 \rangle$ & $\varnothing$
& $1$ & $\langle 1 \rangle$ & $\varnothing$
\\
$A_2$ & $A_1$ & $(2)$
& $1$ & $\langle 1 \rangle$ & $\varnothing$
\\
& &$(1^2)$
& $\tilde A_1$ & $A_1$ & $(1^2)$
\\
$\tilde A_2$ & $\langle 1 \rangle$ & $\varnothing$
& $1$ & $\langle 1 \rangle$ & $\varnothing$
\\
$A_1{+}\tilde A_1$ & $\langle 1 \rangle$ & $\varnothing$
& $1$ & $\langle 1 \rangle$ & $\varnothing$
\\
$\tilde A_1$ & $A_1$ & $(2)$
& $1$ & $\langle 1 \rangle$ & $\varnothing$
\\
& &$(1^2)$
& $A_1$ & $\langle 1 \rangle$ & $\varnothing$
\\
$A_1$ & $\langle 1 \rangle$ & $\varnothing$
& $1$ & $\langle 1 \rangle$ & $\varnothing$
\\
$1$ & $\langle 1 \rangle$ & $\varnothing$
& $1$ & $\langle 1 \rangle$ & $\varnothing$
\\
\end{tabular}
\end{minipage}
\begin{minipage}{.45\linewidth}
\caption{$E_6$}
\centering
\begin{tabular}{l|l|l|l|l|l}
$C$ & $A(C)$ & $\varepsilon$ & $C^{\text{min}}$ & $A(C^{\text{min}})$ 
 & $\varepsilon^{\text{min}}$
\\ 
\hline 
$E_6$ & $\mathbb{Z}/3\mathbb{Z}$ & $(1)$
& $1$ & $\langle 1 \rangle$ & $\varnothing$
\\
& &$\chi_3$
& $2A_2$ & $\mathbb{Z}/3\mathbb{Z}$ & $\chi_3$
\\
& &$\chi_3^2$
& $2A_2$ & $\mathbb{Z}/3\mathbb{Z}$ & $\chi_3^2$
\\
$E_6(a_1)$ & $\mathbb{Z}/3\mathbb{Z}$ & $(1)$
& $1$ & $\langle 1 \rangle$ & $\varnothing$
\\
& &$\chi_3$
& $2A_2{+}A_1$ & $\mathbb{Z}/3\mathbb{Z}$ & $\chi_3$
\\
& &$\chi_3^2$
& $2A_2{+}A_1$ & $\mathbb{Z}/3\mathbb{Z}$ & $\chi_3^2$
\\
$D_5$ & $\langle 1 \rangle$ & $\varnothing$
& $1$ & $\langle 1 \rangle$ & $\varnothing$
\\
$E_6(a_3)$ & $\mathbb{Z}/6\mathbb{Z}$ & $(1)$
& $1$ & $\langle 1 \rangle$ & $\varnothing$
\\
& &$-\chi_3^2$
& $E_6(a_3)$ & $\mathbb{Z}/6\mathbb{Z}$ & $-\chi_3^2$
\\
& &$\chi_3$
& $2A_2$ & $\mathbb{Z}/3\mathbb{Z}$ & $\chi_3$
\\
& &$(-1)$
& $A_1$ & $\langle 1 \rangle$ & $\varnothing$
\\
& &$\chi_3^2$
& $2A_2$ & $\mathbb{Z}/3\mathbb{Z}$ & $\chi_3^2$
\\
& &$-\chi_3$
& $E_6(a_3)$ & $\mathbb{Z}/6\mathbb{Z}$ & $-\chi_3$
\\
$A_5$ & $\mathbb{Z}/3\mathbb{Z}$ & $(1)$
& $1$ & $\langle 1 \rangle$ & $\varnothing$
\\
& &$\chi_3$
& $2A_2$ & $\mathbb{Z}/3\mathbb{Z}$ & $\chi_3$
\\
& &$\chi_3^2$
& $2A_2$ & $\mathbb{Z}/3\mathbb{Z}$ & $\chi_3^2$
\\
$D_5(a_1)$ & $\langle 1 \rangle$ & $\varnothing$
& $1$ & $\langle 1 \rangle$ & $\varnothing$
\\
$A_4{+}A_1$ & $\langle 1 \rangle$ & $\varnothing$
& $1$ & $\langle 1 \rangle$ & $\varnothing$
\\
$D_4$ & $\langle 1 \rangle$ & $\varnothing$
& $1$ & $\langle 1 \rangle$ & $\varnothing$
\\
$A_4$ & $\langle 1 \rangle$ & $\varnothing$
& $1$ & $\langle 1 \rangle$ & $\varnothing$
\\
$D_4(a_1)$ & $A_2$ & $(1^3)$
& $A_2$ & $A_1$ & $(1^2)$
\\
& &$(21)$
& $A_1$ & $\langle 1 \rangle$ & $\varnothing$
\\
& &$(3)$
& $1$ & $\langle 1 \rangle$ & $\varnothing$
\\
$A_3{+}A_1$ & $\langle 1 \rangle$ & $\varnothing$
& $1$ & $\langle 1 \rangle$ & $\varnothing$
\\
$A_3$ & $\langle 1 \rangle$ & $\varnothing$
& $1$ & $\langle 1 \rangle$ & $\varnothing$
\\
$2A_2{+}A_1$ & $\mathbb{Z}/3\mathbb{Z}$ & $(1)$
& $1$ & $\langle 1 \rangle$ & $\varnothing$
\\
& &$\chi_3$
& $2A_2$ & $\mathbb{Z}/3\mathbb{Z}$ & $\chi_3$
\\
& &$\chi_3^2$
& $2A_2$ & $\mathbb{Z}/3\mathbb{Z}$ & $\chi_3^2$
\\
$2A_2$ & $\mathbb{Z}/3\mathbb{Z}$ & $(1)$
& $1$ & $\langle 1 \rangle$ & $\varnothing$
\\
& &$\chi_3$
& $2A_2$ & $\mathbb{Z}/3\mathbb{Z}$ & $\chi_3$
\\
& &$\chi_3^2$
& $2A_2$ & $\mathbb{Z}/3\mathbb{Z}$ & $\chi_3^2$
\\
$A_2{+}2A_1$ & $\langle 1 \rangle$ & $\varnothing$
& $1$ & $\langle 1 \rangle$ & $\varnothing$
\\
$A_2{+}A_1$ & $\langle 1 \rangle$ & $\varnothing$
& $1$ & $\langle 1 \rangle$ & $\varnothing$
\\
$A_2$ & $A_1$ & $(1^2)$
& $A_1$ & $\langle 1 \rangle$ & $\varnothing$
\\
& &$(2)$
& $1$ & $\langle 1 \rangle$ & $\varnothing$
\\
$3A_1$ & $\langle 1 \rangle$ & $\varnothing$
& $1$ & $\langle 1 \rangle$ & $\varnothing$
\\
$2A_1$ & $\langle 1 \rangle$ & $\varnothing$
& $1$ & $\langle 1 \rangle$ & $\varnothing$
\\
$A_1$ & $\langle 1 \rangle$ & $\varnothing$
& $1$ & $\langle 1 \rangle$ & $\varnothing$
\\
$1$ & $\langle 1 \rangle$ & $\varnothing$
& $1$ & $\langle 1 \rangle$ & $\varnothing$
\\
\end{tabular}
\end{minipage}
\end{table}
\endgroup

\begingroup\tiny
\linespread{0}\selectfont
\begin{table}[H]
\caption{$E_7$}
\centering
\begin{tabular}{l|l|l|l|l|l}
$C$ & $A(C)$ & $\varepsilon$ & $C^{\text{min}}$ & $A(C^{\text{min}})$ 
 & $\varepsilon^{\text{min}}$
\\ 
\hline 
$E_7$ & $A_1$ & $(1^2)$
& $3A_1''$ & $A_1$ & $(1^2)$
\\
& &$(2)$
& $1$ & $\langle 1 \rangle$ & $\varnothing$
\\
$E_7(a_1)$ & $A_1$ & $(1^2)$
& $4A_1$ & $A_1$ & $(1^2)$
\\
& &$(2)$
& $1$ & $\langle 1 \rangle$ & $\varnothing$
\\
$E_7(a_2)$ & $A_1$ & $(1^2)$
& $3A_1''$ & $A_1$ & $(1^2)$
\\
& &$(2)$
& $1$ & $\langle 1 \rangle$ & $\varnothing$
\\
$E_6$ & $\langle 1 \rangle$ & $\varnothing$
& $1$ & $\langle 1 \rangle$ & $\varnothing$
\\
$E_7(a_3)$ & $A_1A_1$ & $(1^2) \boxtimes (1^2)$
& $4A_1$ & $A_1$ & $(1^2)$
\\
& &$(1^2) \boxtimes (2)$
& $A_1$ & $\langle 1 \rangle$ & $\varnothing$
\\
& &$(2) \boxtimes (1^2)$
& $A_2{+}3A_1$ & $A_1$ & $(1^2)$
\\
& &$(2) \boxtimes (2)$
& $1$ & $\langle 1 \rangle$ & $\varnothing$
\\
$E_6(a_1)$ & $A_1$ & $(1^2)$
& $A_1$ & $\langle 1 \rangle$ & $\varnothing$
\\
& &$(2)$
& $1$ & $\langle 1 \rangle$ & $\varnothing$
\\
$D_6$ & $A_1$ & $(1^2)$
& $3A_1''$ & $A_1$ & $(1^2)$
\\
& &$(2)$
& $1$ & $\langle 1 \rangle$ & $\varnothing$
\\
$E_7(a_4)$ & $A_1A_1$ & $(1^2) \boxtimes (1^2)$
& $A_2{+}3A_1$ & $A_1$ & $(1^2)$
\\
& &$(1^2) \boxtimes (2)$
& $4A_1$ & $A_1$ & $(2)$
\\
& &$(2) \boxtimes (1^2)$
& $3A_1''$ & $A_1$ & $(1^2)$
\\
& &$(2) \boxtimes (2)$
& $1$ & $\langle 1 \rangle$ & $\varnothing$
\\
$A_6$ & $\langle 1 \rangle$ & $\varnothing$
& $1$ & $\langle 1 \rangle$ & $\varnothing$
\\
$D_5{+}A_1$ & $A_1$ & $(1^2)$
& $3A_1''$ & $A_1$ & $(1^2)$
\\
& &$(2)$
& $1$ & $\langle 1 \rangle$ & $\varnothing$
\\
$D_6(a_1)$ & $A_1$ & $(1^2)$
& $4A_1$ & $A_1$ & $(1^2)$
\\
& &$(2)$
& $1$ & $\langle 1 \rangle$ & $\varnothing$
\\
$E_7(a_5)$ & $A_2A_1$ & $(1^3) \boxtimes (1^2)$
& $E_7(a_5)$ & $A_2A_1$ & $(1^3) \boxtimes (1^2)$
\\
& &$(1^3) \boxtimes (2)$
& $A_2$ & $A_1$ & $(1^2)$
\\
& &$(21) \boxtimes (1^2)$
& $(A_3{+}A_1)''$ & $A_1$ & $(1^2)$
\\
& &$(21) \boxtimes (2)$
& $A_1$ & $\langle 1 \rangle$ & $\varnothing$
\\
& &$(3) \boxtimes (1^2)$
& $3A_1''$ & $A_1$ & $(1^2)$
\\
& &$(3) \boxtimes (2)$
& $1$ & $\langle 1 \rangle$ & $\varnothing$
\\
$D_5$ & $\langle 1 \rangle$ & $\varnothing$
& $1$ & $\langle 1 \rangle$ & $\varnothing$
\\
$D_6(a_2)$ & $A_1$ & $(1^2)$
& $3A_1''$ & $A_1$ & $(1^2)$
\\
& &$(2)$
& $1$ & $\langle 1 \rangle$ & $\varnothing$
\\
$E_6(a_3)$ & $A_1$ & $(1^2)$
& $A_1$ & $\langle 1 \rangle$ & $\varnothing$
\\
& &$(2)$
& $1$ & $\langle 1 \rangle$ & $\varnothing$
\\
$A_5'$ & $\langle 1 \rangle$ & $\varnothing$
& $1$ & $\langle 1 \rangle$ & $\varnothing$
\\
$D_5(a_1){+}A_1$ & $A_1$ & $(1^2)$
& $4A_1$ & $A_1$ & $(1^2)$
\\
& &$(2)$
& $1$ & $\langle 1 \rangle$ & $\varnothing$
\\
$A_5{+}A_1$ & $A_1$ & $(1^2)$
& $3A_1''$ & $A_1$ & $(1^2)$
\\
& &$(2)$
& $1$ & $\langle 1 \rangle$ & $\varnothing$
\\
$A_4{+}A_2$ & $\langle 1 \rangle$ & $\varnothing$
& $1$ & $\langle 1 \rangle$ & $\varnothing$
\\
$A_5''$ & $A_1$ & $(1^2)$
& $3A_1''$ & $A_1$ & $(1^2)$
\\
& &$(2)$
& $1$ & $\langle 1 \rangle$ & $\varnothing$
\\
$D_5(a_1)$ & $A_1$ & $(1^2)$
& $A_1$ & $\langle 1 \rangle$ & $\varnothing$
\\
& &$(2)$
& $1$ & $\langle 1 \rangle$ & $\varnothing$
\\
$D_4{+}A_1$ & $A_1$ & $(1^2)$
& $3A_1''$ & $A_1$ & $(1^2)$
\\
& &$(2)$
& $1$ & $\langle 1 \rangle$ & $\varnothing$
\\
$A_4{+}A_1$ & $A_1$ & $(1^2)$
& $A_1$ & $\langle 1 \rangle$ & $\varnothing$
\\
& &$(2)$
& $1$ & $\langle 1 \rangle$ & $\varnothing$
\\
$D_4$ & $\langle 1 \rangle$ & $\varnothing$
& $1$ & $\langle 1 \rangle$ & $\varnothing$
\\
$A_3{+}A_2{+}A_1$ & $A_1$ & $(1^2)$
& $3A_1''$ & $A_1$ & $(1^2)$
\\
& &$(2)$
& $1$ & $\langle 1 \rangle$ & $\varnothing$
\\
$A_4$ & $A_1$ & $(1^2)$
& $A_1$ & $\langle 1 \rangle$ & $\varnothing$
\\
& &$(2)$
& $1$ & $\langle 1 \rangle$ & $\varnothing$
\\
$A_3{+}A_2$ & $A_1$ & $(1^2)$
& $4A_1$ & $A_1$ & $(2)$
\\
& &$(2)$
& $1$ & $\langle 1 \rangle$ & $\varnothing$
\\
$D_4(a_1){+}A_1$ & $A_1A_1$ & $(1^2) \boxtimes (1^2)$
& $4A_1$ & $A_1$ & $(1^2)$
\\
& &$(1^2) \boxtimes (2)$
& $A_1$ & $\langle 1 \rangle$ & $\varnothing$
\\
& &$(2) \boxtimes (1^2)$
& $A_2{+}3A_1$ & $A_1$ & $(1^2)$
\\
& &$(2) \boxtimes (2)$
& $1$ & $\langle 1 \rangle$ & $\varnothing$
\\
$D_4(a_1)$ & $A_2$ & $(1^3)$
& $A_2$ & $A_1$ & $(1^2)$
\\
& &$(21)$
& $A_1$ & $\langle 1 \rangle$ & $\varnothing$
\\
& &$(3)$
& $1$ & $\langle 1 \rangle$ & $\varnothing$
\\
$A_3{+}2A_1$ & $A_1$ & $(1^2)$
& $3A_1''$ & $A_1$ & $(1^2)$
\\
& &$(2)$
& $1$ & $\langle 1 \rangle$ & $\varnothing$
\\
$(A_3{+}A_1)'$ & $\langle 1 \rangle$ & $\varnothing$
& $1$ & $\langle 1 \rangle$ & $\varnothing$
\\
$(A_3{+}A_1)''$ & $A_1$ & $(1^2)$
& $3A_1''$ & $A_1$ & $(1^2)$
\\
& &$(2)$
& $1$ & $\langle 1 \rangle$ & $\varnothing$
\\
$2A_2{+}A_1$ & $\langle 1 \rangle$ & $\varnothing$
& $1$ & $\langle 1 \rangle$ & $\varnothing$
\\
$2A_2$ & $\langle 1 \rangle$ & $\varnothing$
& $1$ & $\langle 1 \rangle$ & $\varnothing$
\\
$A_3$ & $\langle 1 \rangle$ & $\varnothing$
& $1$ & $\langle 1 \rangle$ & $\varnothing$
\\
$A_2{+}3A_1$ & $A_1$ & $(1^2)$
& $3A_1''$ & $A_1$ & $(1^2)$
\\
& &$(2)$
& $1$ & $\langle 1 \rangle$ & $\varnothing$
\\
$A_2{+}2A_1$ & $\langle 1 \rangle$ & $\varnothing$
& $1$ & $\langle 1 \rangle$ & $\varnothing$
\\
$A_2{+}A_1$ & $A_1$ & $(1^2)$
& $A_1$ & $\langle 1 \rangle$ & $\varnothing$
\\
& &$(2)$
& $1$ & $\langle 1 \rangle$ & $\varnothing$
\\
$4A_1$ & $A_1$ & $(1^2)$
& $3A_1''$ & $A_1$ & $(1^2)$
\\
& &$(2)$
& $1$ & $\langle 1 \rangle$ & $\varnothing$
\\
$A_2$ & $A_1$ & $(1^2)$
& $A_1$ & $\langle 1 \rangle$ & $\varnothing$
\\
& &$(2)$
& $1$ & $\langle 1 \rangle$ & $\varnothing$
\\
$3A_1''$ & $A_1$ & $(1^2)$
& $3A_1''$ & $A_1$ & $(1^2)$
\\
& &$(2)$
& $1$ & $\langle 1 \rangle$ & $\varnothing$
\\
$3A_1'$ & $\langle 1 \rangle$ & $\varnothing$
& $1$ & $\langle 1 \rangle$ & $\varnothing$
\\
$2A_1$ & $\langle 1 \rangle$ & $\varnothing$
& $1$ & $\langle 1 \rangle$ & $\varnothing$
\\
$A_1$ & $\langle 1 \rangle$ & $\varnothing$
& $1$ & $\langle 1 \rangle$ & $\varnothing$
\\
$1$ & $\langle 1 \rangle$ & $\varnothing$
& $1$ & $\langle 1 \rangle$ & $\varnothing$
\\
\end{tabular}
\end{table}
\endgroup

\begingroup\scriptsize
\linespread{0}\selectfont
\begin{table}[H]
\caption{$E_8$}
\begin{minipage}{.45\linewidth}
\centering
\begin{tabular}{l|l|l|l|l|l}
$C$ & $A(C)$ & $\varepsilon$ & $C^{\text{min}}$ & $A(C^{\text{min}})$ 
 & $\varepsilon^{\text{min}}$
\\ 
\hline 
$E_8$ & $\langle 1 \rangle$ & $\varnothing$
& $1$ & $\langle 1 \rangle$ & $\varnothing$
\\
$E_8(a_1)$ & $\langle 1 \rangle$ & $\varnothing$
& $1$ & $\langle 1 \rangle$ & $\varnothing$
\\
$E_8(a_2)$ & $\langle 1 \rangle$ & $\varnothing$
& $1$ & $\langle 1 \rangle$ & $\varnothing$
\\
$E_8(a_3)$ & $A_1$ & $(1^2)$
& $A_1$ & $\langle 1 \rangle$ & $\varnothing$
\\
& &$(2)$
& $1$ & $\langle 1 \rangle$ & $\varnothing$
\\
$E_7$ & $\langle 1 \rangle$ & $\varnothing$
& $1$ & $\langle 1 \rangle$ & $\varnothing$
\\
$E_8(a_4)$ & $A_1$ & $(1^2)$
& $A_1$ & $\langle 1 \rangle$ & $\varnothing$
\\
& &$(2)$
& $1$ & $\langle 1 \rangle$ & $\varnothing$
\\
$E_8(b_4)$ & $A_1$ & $(1^2)$
& $4A_1$ & $\langle 1 \rangle$ & $\varnothing$
\\
& &$(2)$
& $1$ & $\langle 1 \rangle$ & $\varnothing$
\\
$E_7(a_1)$ & $\langle 1 \rangle$ & $\varnothing$
& $1$ & $\langle 1 \rangle$ & $\varnothing$
\\
$E_8(a_5)$ & $A_1$ & $(1^2)$
& $2A_1$ & $\langle 1 \rangle$ & $\varnothing$
\\
& &$(2)$
& $1$ & $\langle 1 \rangle$ & $\varnothing$
\\
$E_8(b_5)$ & $A_2$ & $(1^3)$
& $A_2$ & $A_1$ & $(1^2)$
\\
& &$(21)$
& $A_1$ & $\langle 1 \rangle$ & $\varnothing$
\\
& &$(3)$
& $1$ & $\langle 1 \rangle$ & $\varnothing$
\\
$D_7$ & $\langle 1 \rangle$ & $\varnothing$
& $1$ & $\langle 1 \rangle$ & $\varnothing$
\\
$E_7(a_2)$ & $\langle 1 \rangle$ & $\varnothing$
& $1$ & $\langle 1 \rangle$ & $\varnothing$
\\
$E_8(a_6)$ & $A_2$ & $(1^3)$
& $A_2$ & $A_1$ & $(1^2)$
\\
& &$(21)$
& $A_1$ & $\langle 1 \rangle$ & $\varnothing$
\\
& &$(3)$
& $1$ & $\langle 1 \rangle$ & $\varnothing$
\\
$E_6{+}A_1$ & $\langle 1 \rangle$ & $\varnothing$
& $1$ & $\langle 1 \rangle$ & $\varnothing$
\\
$D_7(a_1)$ & $A_1$ & $(1^2)$
& $4A_1$ & $\langle 1 \rangle$ & $\varnothing$
\\
& &$(2)$
& $1$ & $\langle 1 \rangle$ & $\varnothing$
\\
$E_6$ & $\langle 1 \rangle$ & $\varnothing$
& $1$ & $\langle 1 \rangle$ & $\varnothing$
\\
$E_7(a_3)$ & $A_1$ & $(1^2)$
& $A_1$ & $\langle 1 \rangle$ & $\varnothing$
\\
& &$(2)$
& $1$ & $\langle 1 \rangle$ & $\varnothing$
\\
$E_8(b_6)$ & $A_2$ & $(1^3)$
& $A_2{+}A_1$ & $A_1$ & $(1^2)$
\\
& &$(21)$
& $2A_2{+}2A_1$ & $\langle 1 \rangle$ & $\varnothing$
\\
& &$(3)$
& $1$ & $\langle 1 \rangle$ & $\varnothing$
\\
$D_6$ & $\langle 1 \rangle$ & $\varnothing$
& $1$ & $\langle 1 \rangle$ & $\varnothing$
\\
$E_6(a_1){+}A_1$ & $A_1$ & $(1^2)$
& $A_1$ & $\langle 1 \rangle$ & $\varnothing$
\\
& &$(2)$
& $1$ & $\langle 1 \rangle$ & $\varnothing$
\\
$A_7$ & $\langle 1 \rangle$ & $\varnothing$
& $1$ & $\langle 1 \rangle$ & $\varnothing$
\\
$D_7(a_2)$ & $A_1$ & $(1^2)$
& $A_1$ & $\langle 1 \rangle$ & $\varnothing$
\\
& &$(2)$
& $1$ & $\langle 1 \rangle$ & $\varnothing$
\\
$E_6(a_1)$ & $A_1$ & $(1^2)$
& $A_1$ & $\langle 1 \rangle$ & $\varnothing$
\\
& &$(2)$
& $1$ & $\langle 1 \rangle$ & $\varnothing$
\\
$D_5{+}A_2$ & $A_1$ & $(1^2)$
& $2A_2$ & $A_1$ & $(1^2)$
\\
& &$(2)$
& $1$ & $\langle 1 \rangle$ & $\varnothing$
\\
$E_7(a_4)$ & $A_1$ & $(1^2)$
& $4A_1$ & $\langle 1 \rangle$ & $\varnothing$
\\
& &$(2)$
& $1$ & $\langle 1 \rangle$ & $\varnothing$
\\
$A_6{+}A_1$ & $\langle 1 \rangle$ & $\varnothing$
& $1$ & $\langle 1 \rangle$ & $\varnothing$
\\
$D_6(a_1)$ & $A_1$ & $(1^2)$
& $A_1$ & $\langle 1 \rangle$ & $\varnothing$
\\
& &$(2)$
& $1$ & $\langle 1 \rangle$ & $\varnothing$
\\
$A_6$ & $\langle 1 \rangle$ & $\varnothing$
& $1$ & $\langle 1 \rangle$ & $\varnothing$
\\
$E_8(a_7)$ & $A_4$ & $(1^5)$
& $E_8(a_7)$ & $A_4$ & $(1^5)$
\\
& &$(21^3)$
& $D_4(a_1)$ & $A_2$ & $(1^3)$
\\
& &$(2^21)$
& $A_2{+}A_1$ & $A_1$ & $(2)$
\\
& &$(31^2)$
& $A_2$ & $A_1$ & $(1^2)$
\\
& &$(32)$
& $2A_1$ & $\langle 1 \rangle$ & $\varnothing$
\\
& &$(41)$
& $A_1$ & $\langle 1 \rangle$ & $\varnothing$
\\
& &$(5)$
& $1$ & $\langle 1 \rangle$ & $\varnothing$
\\
$D_5{+}A_1$ & $\langle 1 \rangle$ & $\varnothing$
& $1$ & $\langle 1 \rangle$ & $\varnothing$
\\
$E_7(a_5)$ & $A_2$ & $(1^3)$
& $A_2$ & $A_1$ & $(1^2)$
\\
& &$(21)$
& $A_1$ & $\langle 1 \rangle$ & $\varnothing$
\\
& &$(3)$
& $1$ & $\langle 1 \rangle$ & $\varnothing$
\end{tabular}
\end{minipage}
\quad\quad\quad\quad
\begin{minipage}{.45\linewidth}
\centering
\begin{tabular}{l|l|l|l|l|l}
$D_5$ & $\langle 1 \rangle$ & $\varnothing$
& $1$ & $\langle 1 \rangle$ & $\varnothing$
\\
$E_6(a_3){+}A_1$ & $A_1$ & $(1^2)$
& $A_1$ & $\langle 1 \rangle$ & $\varnothing$
\\
& &$(2)$
& $1$ & $\langle 1 \rangle$ & $\varnothing$
\\
$D_6(a_2)$ & $A_1$ & $(1^2)$
& $A_1$ & $\langle 1 \rangle$ & $\varnothing$
\\
& &$(2)$
& $1$ & $\langle 1 \rangle$ & $\varnothing$
\\
$E_6(a_3)$ & $A_1$ & $(1^2)$
& $A_1$ & $\langle 1 \rangle$ & $\varnothing$
\\
& &$(2)$
& $1$ & $\langle 1 \rangle$ & $\varnothing$
\\
$A_5{+}A_1$ & $\langle 1 \rangle$ & $\varnothing$
& $1$ & $\langle 1 \rangle$ & $\varnothing$
\\
$D_5(a_1){+}A_2$ & $\langle 1 \rangle$ & $\varnothing$
& $1$ & $\langle 1 \rangle$ & $\varnothing$
\\
$A_4{+}A_3$ & $\langle 1 \rangle$ & $\varnothing$
& $1$ & $\langle 1 \rangle$ & $\varnothing$
\\
$D_4{+}A_2$ & $A_1$ & $(1^2)$
& $4A_1$ & $\langle 1 \rangle$ & $\varnothing$
\\
& &$(2)$
& $1$ & $\langle 1 \rangle$ & $\varnothing$
\\
$A_5$ & $\langle 1 \rangle$ & $\varnothing$
& $1$ & $\langle 1 \rangle$ & $\varnothing$
\\
$A_4{+}A_2{+}A_1$ & $\langle 1 \rangle$ & $\varnothing$
& $1$ & $\langle 1 \rangle$ & $\varnothing$
\\
$D_5(a_1){+}A_1$ & $\langle 1 \rangle$ & $\varnothing$
& $1$ & $\langle 1 \rangle$ & $\varnothing$
\\
$D_5(a_1)$ & $A_1$ & $(1^2)$
& $A_1$ & $\langle 1 \rangle$ & $\varnothing$
\\
& &$(2)$
& $1$ & $\langle 1 \rangle$ & $\varnothing$
\\
$A_4{+}A_2$ & $\langle 1 \rangle$ & $\varnothing$
& $1$ & $\langle 1 \rangle$ & $\varnothing$
\\
$A_4{+}2A_1$ & $A_1$ & $(1^2)$
& $2A_1$ & $\langle 1 \rangle$ & $\varnothing$
\\
& &$(2)$
& $1$ & $\langle 1 \rangle$ & $\varnothing$
\\
$D_4{+}A_1$ & $\langle 1 \rangle$ & $\varnothing$
& $1$ & $\langle 1 \rangle$ & $\varnothing$
\\
$A_4{+}A_1$ & $A_1$ & $(1^2)$
& $A_1$ & $\langle 1 \rangle$ & $\varnothing$
\\
& &$(2)$
& $1$ & $\langle 1 \rangle$ & $\varnothing$
\\
$2A_3$ & $\langle 1 \rangle$ & $\varnothing$
& $1$ & $\langle 1 \rangle$ & $\varnothing$
\\
$D_4$ & $\langle 1 \rangle$ & $\varnothing$
& $1$ & $\langle 1 \rangle$ & $\varnothing$
\\
$A_4$ & $A_1$ & $(1^2)$
& $A_1$ & $\langle 1 \rangle$ & $\varnothing$
\\
& &$(2)$
& $1$ & $\langle 1 \rangle$ & $\varnothing$
\\
$D_4(a_1){+}A_2$ & $A_1$ & $(1^2)$
& $A_1$ & $\langle 1 \rangle$ & $\varnothing$
\\
& &$(2)$
& $1$ & $\langle 1 \rangle$ & $\varnothing$
\\
$A_3{+}A_2{+}A_1$ & $\langle 1 \rangle$ & $\varnothing$
& $1$ & $\langle 1 \rangle$ & $\varnothing$
\\
$A_3{+}A_2$ & $A_1$ & $(1^2)$
& $4A_1$ & $\langle 1 \rangle$ & $\varnothing$
\\
& &$(2)$
& $1$ & $\langle 1 \rangle$ & $\varnothing$
\\
$D_4(a_1){+}A_1$ & $A_2$ & $(1^3)$
& $A_2$ & $A_1$ & $(1^2)$
\\
& &$(21)$
& $A_1$ & $\langle 1 \rangle$ & $\varnothing$
\\
& &$(3)$
& $1$ & $\langle 1 \rangle$ & $\varnothing$
\\
$A_3{+}2A_1$ & $\langle 1 \rangle$ & $\varnothing$
& $1$ & $\langle 1 \rangle$ & $\varnothing$
\\
$D_4(a_1)$ & $A_2$ & $(1^3)$
& $A_2$ & $A_1$ & $(1^2)$
\\
& &$(21)$
& $A_1$ & $\langle 1 \rangle$ & $\varnothing$
\\
& &$(3)$
& $1$ & $\langle 1 \rangle$ & $\varnothing$
\\
$2A_2{+}2A_1$ & $\langle 1 \rangle$ & $\varnothing$
& $1$ & $\langle 1 \rangle$ & $\varnothing$
\\
$A_3{+}A_1$ & $\langle 1 \rangle$ & $\varnothing$
& $1$ & $\langle 1 \rangle$ & $\varnothing$
\\
$2A_2{+}A_1$ & $\langle 1 \rangle$ & $\varnothing$
& $1$ & $\langle 1 \rangle$ & $\varnothing$
\\
$A_3$ & $\langle 1 \rangle$ & $\varnothing$
& $1$ & $\langle 1 \rangle$ & $\varnothing$
\\
$2A_2$ & $A_1$ & $(1^2)$
& $A_1$ & $\langle 1 \rangle$ & $\varnothing$
\\
& &$(2)$
& $1$ & $\langle 1 \rangle$ & $\varnothing$
\\
$A_2{+}3A_1$ & $\langle 1 \rangle$ & $\varnothing$
& $1$ & $\langle 1 \rangle$ & $\varnothing$
\\
$A_2{+}2A_1$ & $\langle 1 \rangle$ & $\varnothing$
& $1$ & $\langle 1 \rangle$ & $\varnothing$
\\
$A_2{+}A_1$ & $A_1$ & $(1^2)$
& $A_1$ & $\langle 1 \rangle$ & $\varnothing$
\\
& &$(2)$
& $1$ & $\langle 1 \rangle$ & $\varnothing$
\\
$4A_1$ & $\langle 1 \rangle$ & $\varnothing$
& $1$ & $\langle 1 \rangle$ & $\varnothing$
\\
$A_2$ & $A_1$ & $(1^2)$
& $A_1$ & $\langle 1 \rangle$ & $\varnothing$
\\
& &$(2)$
& $1$ & $\langle 1 \rangle$ & $\varnothing$
\\
$3A_1$ & $\langle 1 \rangle$ & $\varnothing$
& $1$ & $\langle 1 \rangle$ & $\varnothing$
\\
$2A_1$ & $\langle 1 \rangle$ & $\varnothing$
& $1$ & $\langle 1 \rangle$ & $\varnothing$
\\
$A_1$ & $\langle 1 \rangle$ & $\varnothing$
& $1$ & $\langle 1 \rangle$ & $\varnothing$
\\
$1$ & $\langle 1 \rangle$ & $\varnothing$
& $1$ & $\langle 1 \rangle$ & $\varnothing$
\\
\end{tabular}
\end{minipage}
\end{table}
\endgroup

\bibliographystyle{alpha}
\bibliography{main}

\end{document}